\documentclass[oneside,11pt]{amsart}
\usepackage{amsmath,amssymb,amsthm, mathrsfs, mathtools}
\usepackage{amscd,mathtools}
\usepackage{enumitem} \setlist{nosep} % Change enumeration labelling
\usepackage[utf8]{inputenc}
\usepackage[english]{babel}
\usepackage{color}
\usepackage[pagebackref,breaklinks,hidelinks,unicode]{hyperref}
\usepackage{float}
\usepackage{comment}
\usepackage{mathtools}
\usepackage{csquotes}
\usepackage[font=small,justification=centering]{caption}
\usepackage{subcaption}
\usepackage[dvipsnames]{xcolor}
\usepackage{mathabx} % Includes \sqsubset etc.
\usepackage{xfrac} % Nice quotients
\usepackage[normalem]{ulem} % Contains \sout
\usepackage{dirtytalk}

\usepackage[margin=3cm]{geometry}

\usepackage{tikz}
\usepackage{tikz-cd}
\usetikzlibrary{automata}
\usetikzlibrary{calc,decorations.pathreplacing}
\usetikzlibrary{arrows}
\usetikzlibrary{decorations.pathreplacing}
\usetikzlibrary{intersections}

\makeatletter
    \newtheorem*{rep@theorem}{\rep@title}
    \newcommand{\newreptheorem}[2]{%
    \newenvironment{rep#1}[1]{%
    \def\rep@title{#2 \ref{##1}}%
    \begin{rep@theorem}}%
    {\end{rep@theorem}}}
\makeatother

\newtheorem{theorem}{Theorem}[section]
\newtheorem{proposition}[theorem]{Proposition}
\newreptheorem{proposition}{Proposition}
\newtheorem{corollary}[theorem]{Corollary}
\newtheorem{lemma}[theorem]{Lemma}

\newtheorem{Set up}[theorem]{Set-up}

\newtheorem{introthm}{Theorem}
\newtheorem{introcor}[introthm]{Corollary}

\theoremstyle{definition}
\newtheorem{introdef}[introthm]{Definition}
\newtheorem{introprop}[introthm]{Proposition}
\newtheorem{definition}[theorem]{Definition}

\newtheorem{example}[theorem]{Example}

\newtheorem{remark}[theorem]{Remark}

\newtheorem*{answer*}{Answer}
\newtheorem*{application*}{Application}
\newtheorem*{warning}{Warning}

\DeclarePairedDelimiterX{\Norm}[1]{\lVert}{\rVert}{#1}
\theoremstyle{definition}

  \newcommand{\gothic}{\mathfrak}
  \newcommand{\go}{{\gothic o}}
\renewcommand*{\backrefalt}[4]{\ifcase #1 (Not cited).\or (Cited p.~#2).\else (Cited pp.~#2).\fi} % Cited on...

\makeatletter\def\subsection{\@startsection{subsection}{1}\z@{.7\linespacing\@plus\linespacing}
    {.5\linespacing}{\normalfont\scshape\centering}}\makeatother % centrally aligned subsections

\newcounter{shcount}
\newcommand*{\bsh}[1]{\theoremstyle{definition}\newtheorem{subhead\theshcount}[theorem]{#1}
    \begin{subhead\theshcount}} 
\newcommand*{\esh}{\end{subhead\theshcount}\stepcounter{shcount}} % On-the-fly environments

\newcounter{enumlabelcount}
\makeatletter\newcommand\enumlabel[1][]{\item[#1]
    \refstepcounter{enumlabelcount}\def\@currentlabel{#1}}\makeatother

\definecolor{harrycomment}{rgb}{0.6,0,0.4}
\newcommand*{\hp}[1]{{\color{harrycomment}#1}}

\DeclareMathOperator{\dist}{\mathsf{d}}
\DeclareMathOperator{\Dist}{\mathsf{D}}
\DeclareMathOperator{\diam}{diam}

\DeclareMathOperator{\isom}{Isom}
\DeclareMathOperator{\Out}{Out}
\DeclareMathOperator{\PU}{PU}

\newcommand*{\B}{\mathcal B}
\newcommand*{\Chain}{\mathcal C}

\newcommand*{\R}{\mathbf R}

\newcommand*{\T}{\cal T}
\newcommand*{\X}{X_{\Dist}}
\newcommand*{\Z}{\mathbf Z}

\newcommand*{\cal}{\mathcal}
\renewcommand{\bf}{\mathbf}

\newcommand{\eps}{\varepsilon}
\newcommand*{\mc}{\mathcal}

\newcommand*{\ora}{\to_\omega }

\newcounter{claimcount}

\newenvironment{claim*}[1]{\par\vspace{2mm}\noindent
    \underline{Claim:}\hspace{2mm}#1}{}
\newenvironment{claimproof}[1]{\par\vspace{2mm}\noindent\underline{Proof:}\hspace{2mm}#1}
    {\leavevmode\unskip\penalty9999\hbox{}\nobreak\hfill\quad\hbox{$\diamondsuit$}\vspace{2mm}}

\newcommand*{\dcomment}[1]{ {\color{blue} [#1]}}

\title[hyperbolic models for CAT(0) spaces]{Hyperbolic models for CAT(0) spaces} 
\hypersetup{pdftitle = Hyperbolic models for CAT(0) spaces}
\author{Harry Petyt, Davide Spriano, and Abdul Zalloum}
\begin{document}

\maketitle

\begin{abstract} We introduce two new tools for studying CAT(0) spaces: \emph{curtains}, versions of cubical hyperplanes; and the \emph{curtain model}, a counterpart of the curve graph. These tools shed new light on CAT(0) spaces, allowing us to prove a dichotomy of a rank-rigidity flavour, establish Ivanov-style rigidity theorems for isometries of the curtain model, find isometry-invariant copies of its Gromov boundary in the visual boundary of the underlying CAT(0) space, and characterise rank-one isometries both in terms of their action on the curtain model and in terms of curtains. Finally, we show that the curtain model is universal for WPD actions over all groups acting properly on the CAT(0) space.
\end{abstract}

\setcounter{tocdepth}{1}\tableofcontents\setcounter{tocdepth}{2}

\section{Introduction}

Two of the most well-studied topics in geometric group theory are CAT(0) cube complexes and mapping class groups. This is in part because they both admit powerful combinatorial-like structures that encode interesting aspects of their geometry: hyperplanes for the former and curve graphs for the latter. In recent years, analogies between the two theories have become more and more apparent. For instance: there are counterparts of curve graphs for CAT(0) cube complexes \cite{hagen:weak,genevois:hyperbolicities} and rigidity theorems for these counterparts that mirror the surface setting \cite{ivanov:automorphisms,fioravanti:automorphisms}; it has been shown that mapping class groups are quasiisometric to CAT(0) cube complexes \cite{petyt:mapping}; and both can be studied using the machinery of hierarchical hyperbolicity \cite{behrstockhagensisto:hierarchically:2}. However, the considerably larger class of CAT(0) \emph{spaces} is left out of this analogy, as the lack of a combinatorial-like structure presents a difficulty in importing techniques from those areas. In this paper, we bring CAT(0) spaces into the picture by developing versions of hyperplanes and curve graphs for them.

\medskip

CAT(0) cube complexes have been studied both via group actions and as interesting spaces in their own right \cite{nibloreeves:geometry,chepoi:graphs,sageevwise:tits,capracesageev:rank,chatterjifernosiozzi:median,huang:top,chalopinchepoi:safe}, and led to groundbreaking advances in 3--manifold theory \cite{wise:structure,agol:virtual}. Since revolutionary work of Sageev \cite{sageev:ends,sageev:codimension} it has become increasingly clear that their geometry is entirely encoded by their hyperplanes and how they interact with one another. Notably, we are not aware of any cases where methods from the world of cube complexes have been successfully exported to the full CAT(0) setting. One explanation could be that CAT(0) cube complexes seem to be rather unrepresentative of the more general class of CAT(0) spaces. For instance, many CAT(0) groups have property~(T), but no group admitting an unbounded action on a CAT(0) cube complex can have property~(T) \cite{nibloreeves:groups}.

The main new notion we introduce is that of \emph{curtains}, which are CAT(0) analogues of hyperplanes.

\begin{introdef}
Let $X$ be a CAT(0) space. A \emph{curtain} is $\pi_\alpha^{-1}(P)$, where $\alpha$ is a geodesic, $\pi_\alpha$ is the closest-point projection, and $P$ is a length-one subinterval of the interior of $\alpha$.
\end{introdef}

\noindent Each curtain delimits two natural path-connected ``halfspaces'', which it separates from each other, and just as in CAT(0) cube complexes, one can use curtains to measure the distance between two points (see Lemma~\ref{lem:chain_distance}). Curtains also present key differences from hyperplanes. The two most noteworthy are that the set of curtains is uncountable, and that curtains are not convex in general (Remark~\ref{rem:nonconvex}). However, such differences are necessary, as Sageev's construction \cite{sageev:codimension} produces cube complexes under very weak conditions. If we are to consider the more general class of CAT(0) spaces, curtains must not satisfy such conditions. Moreover, nonconvexity is even to be expected by comparison with complex hyperbolic space, see Remark~\ref{rem:nonconvex}.

\medskip

The other analogy considered is with \emph{curve graphs} of surfaces. The discovery that curve graphs \cite{harvey:boundary} are hyperbolic \cite{masurminsky:geometry:1} has been one of the most influential results in the theory of mapping class groups, and it has had many important repercussions \cite{farbmosher:convex,bowditch:tight,maher:random,brockcanaryminsky:classification,masurschleimer:geometry,bestvinabrombergfujiwara:constructing}. Since then, analogous spaces have been introduced in different settings with great effect, notably $\Out F_n$ \cite{hatchervogtmann:complex,bestvinafeighn:hyperbolicity,hatcher:homological,kapovichlustig:geometric,handelmosher:free:1,mann:hyperbolicity}, free products \cite{horbez:hyperbolic}, graph products \cite{genevois:cubical}, cocompactly cubulated groups \cite{hagen:weak,genevois:hyperbolicities}, and various classes of Artin groups \cite{kimkoberda:embedability,calvezwiest:curve,cumplidogebhardtgonzalezmeneseswiest:onparabolic,morriswright:parabolic,martinprzytycki:acylindrical,hagenmartinsisto:extra}. 

Given an arbitrary CAT(0) space $X$, we use curtains to define a new metric $\Dist$ on $X$, and it will be seen from the results described below that the space $\X = (X, \Dist)$, which we call the \emph{curtain model} of $X$, is a very close analogue of the curve graph, as they share many fundamental properties. Moreover, the full isometry group $\isom X$ acts on $\X$, providing a new avenue to study isometry groups of CAT(0) spaces (which are uncountable in general) and not only groups with a geometric action on a CAT(0) space. Finally, let us remark that this is the first construction of a space with properties as strong as the curtain model even for the cubical setting, despite the fact that general CAT(0) spaces are considerably more wild than cube complexes. In the first place, we show the following.

\begin{introthm} \label{ithm:XD_hyperbolic}
There exists some $\delta$ such that for every CAT(0) space $X$, the curtain model $\X$ is $\delta$-hyperbolic and $\isom X\le\isom\X$. Furthermore, there is some $k$ such that each CAT(0) geodesic is an unparametrised $(1,k)$--quasigeodesic in $\X$.
\end{introthm}

\begin{warning} 
The construction of the curtain model involves an infinite family of auxiliary spaces. We postpone a discussion of these spaces to Section~\ref{subsec:aux_spaces}, but the reader should be aware that most of the results of the paper are first established for those auxiliary spaces, and only in Section~\ref{sec:universal} will the data be combined to obtain statements about the curtain model.
\end{warning}

%%%%%%%%%%%%%%%%%%%%%%%%%%%%%%%%%%%%%%%%%%%%%%%%%%
\subsection{Hyperbolic isometries}

Both surfaces and proper CAT(0) spaces have automorphisms that can naturally be considered hyperbolic-like, namely pseudo-Anosov homeomorphisms and rank-one isometries. Pseudo-Anosovs are exactly those mapping classes that act loxodromically on the curve graph \cite{masurminsky:geometry:1}, and, in the cubical setting, rank-one isometries are those that skewer a pair of \emph{separated} hyperplanes \cite{capracesageev:rank,genevois:contracting}. For non-proper CAT(0) spaces, the natural generalisation of ``rank-one'' is ``contracting'' \cite{bestvinafujiwara:characterization}. The notion of separation carries over to the setting of curtains (Definition~\ref{def:separated}), allowing us to bring the two perspectives together in CAT(0) spaces. 

\begin{introthm}\label{thmi:rank-one_iff_skewers_iff_QI}
Let $g$ be a semisimple isometry of a CAT(0) space $X$. The following are~\mbox{equivalent}.
\begin{itemize}
\item   $g$ is contracting (or equivalently rank-one if $X$ is proper). 
\item   $g$ skewers a pair of separated curtains.
\item   $g$ acts loxodromically on the curtain model $\X$.
\end{itemize}
\end{introthm}
\noindent We remark that a version of Theorem~\ref{thmi:rank-one_iff_skewers_iff_QI} holds even without periodicity; namely a geodesic is contracting if and only if it is quasiisometrically embedded in $\X$ (see Corollary~\ref{cor:contracting_qie:D}).

There is a standard technique that constructs an action on a hyperbolic space with a given finite set of contracting elements acting loxodromically, namely the celebrated \emph{projection complex} construction \cite{bestvinabrombergfujiwara:constructing, bestvinabrombergfujiwarasisto:acylindrical}. Our construction uses different methods, which rely on the original space being CAT(0), and this allows us to get a significantly stronger result: we obtain a hyperbolic space where \emph{all} contracting elements act loxodromically. That is to say, the space $\X$ is a hyperbolic space that is \emph{universal with respect to rank-one elements}. The existence of such a space was completely open even under the stronger assumptions of a group $G$ acting geometrically and of $X$ being a CAT(0) cube complex. 

Theorem~\ref{thmi:rank-one_iff_skewers_iff_QI} has the following simple consequence. If $\X$ is unbounded and a group $G$ acts coboundedly on $X$, then Gromov's classification of actions on hyperbolic spaces yields an element of $G$ acting loxodromically on $\X$, which is a contracting isometry by the above theorem.

When considering subgroups instead of single elements, an important notion of hyperbolic-type behaviour is given by \emph{stability}; an isometry of a CAT(0) space is contracting if and only if the cyclic subgroup it generates is stable. Stability was introduced by Durham-Taylor \cite{durhamtaylor:convex}, who showed that a subgroup of the mapping class group is stable if and only if it is convex cocompact in the sense of Farb--Mosher \cite{farbmosher:convex}. Hence, a subgroup of the mapping class group is stable exactly when its orbit maps to the curve graph are quasiisometric embeddings \cite{kentleininger:shadows,hamenstadt:word}. Other instances of this perspective include \cite{bestvinabrombergkentleininger:undistorted,aougabdurhamtaylor:pulling,koberdamangahastaylor:geometry,abbottbehrstockdurham:largest,chesser:stable}. We show that the same result holds for the curtain model.

\begin{introthm}\label{thmi:stable}
A subgroup of a group acting properly coboundedly on a CAT(0) space $X$ is stable if and only if it is finitely generated and its orbit maps to the curtain model $\X$ are quasiisometric embeddings.
\end{introthm}

Theorem~\ref{thmi:rank-one_iff_skewers_iff_QI} gives a complete characterisation of elements of $\isom X$ that act loxodromically on the curtain model. This is a property of the individual isometry in question, and does not give any global information. In this setting, a natural global notion of directional faithfulness is \emph{weak proper discontinuity (WPD)}: we say that $g\in G$ acting loxodromically on $Y$ is WPD if for each $\eps>0$ and each $x\in Y$, there exists $m>0$ such that
\[\left\vert\{h \in G \::\: \dist(x,hx),\, \dist(g^mx,hg^mx)<\eps\}\right\vert < \infty.	
\]

It is straightforward to construct examples where no element of $\isom X$ is WPD for the action on $\X$ of the full isometry group $\isom X$, simply because $\isom X$ is too large. Hence the WPD notion can only be expected when there is some restriction on the action.

\begin{introthm} \label{introthm:non_uniform_acyl}
Let $X$ be a CAT(0) space. If $G\leq\isom X$ acts properly on $X$, then every rank-one element of $G$ is a WPD loxodromic of the curtain model $\X$ with respect to $G$.
\end{introthm}

Theorems~\ref{thmi:stable} and~\ref{introthm:non_uniform_acyl} can both be viewed as stemming from Theorem~\ref{thmi:rank-one_iff_skewers_iff_QI}: Theorem~\ref{thmi:stable} extends the result to subgroups that are not necessarily cyclic, and Theorem~\ref{introthm:non_uniform_acyl} strengthens its conclusion. Our next result marries the two, using the notion of A/QI triple, a natural extension of WPD actions for (not necessarily cyclic) subgroups that was introduced by Abbott--Manning \cite{abbottmanning:acylindrically}.

\begin{introthm} \label{ithm:A/QI}
The following are equivalent for a finitely generated subgroup $H$ of a group $G$ acting properly coboundedly on a CAT(0) space $X$.
\begin{itemize}
\item   $H$ is stable.
\item   $(G,\X,H)$ is an A/QI triple.
\item   There exists an A/QI triple $(G,Z,H)$ for some hyperbolic space $Z$.
\end{itemize}
\end{introthm}

When restricted to cyclic subgroups, Theorem~\ref{ithm:A/QI} implies that if $G$ is a CAT(0) group acting on a hyperbolic space $Z$ and $g\in G$ is a WPD loxodromic for that action, then $g$ is a WPD loxodromic for the action of $G$ on $\X$. Therefore, the curtain model can be thought of as a universal hyperbolic space for WPD actions. And in fact Theorem \ref{ithm:A/QI} states that a similar statement holds for any stable subgroup.

We also point out that Theorem~\ref{introthm:non_uniform_acyl} recovers that the existence of a rank-one element implies acylindrical hyperbolicity of the group if it is not virtually cyclic \cite{sisto:contracting}. Thus it is natural to ask for a geometric criterion implying the existence of a rank-one element. We show that curtains provide such a criterion under the natural (and standard) assumption of the \emph{geodesic extension property}. Recall that a CAT(0) space has the geodesic extension property if every geodesic segment appears in some biinfinite geodesic.

\begin{introthm} \label{ithm:dichotomy} 
Let $X$ be a CAT(0) space with the geodesic extension property and let $G$ be a group acting properly cocompactly on $X$. Then $G$ contains a rank-one element if and only if $X$ contains a pair of separated curtains. 
\end{introthm}

As a direct consequence of \cite[Prop.~3.3]{kapovichleeb:3manifold}, we obtain the following corollary, which extends a result of Levcovitz \cite{levcovitz:divergence} for CAT(0) cube complexes (see also \cite{gersten:quadratic}). It can additionally be interpreted in terms of $\X$, generalising a result of Hagen about contact graphs of CAT(0) cube complexes \cite{hagen:simplicial}.

\begin{introcor}
Let $X$ be a CAT(0) space with the geodesic extension property admitting a proper cocompact group action. If  $X$ has a pair of separated curtains, then the divergence of $X$ is at least quadratic.
\end{introcor}

%%%%%%%%%%%%%%%%%%%%%%%%%%%%%%%%%%%%%%%%%%%%%%%%%%%%%%%%%%%%

\subsection{Ivanov's theorem} \label{isubsec:ivanov}

Aside from hyperbolicity, one of the most important results about the curve graph is \emph{Ivanov's theorem} \cite{ivanov:automorphisms,korkmaz:automorphisms,luo:automorphisms}, which states that every automorphism of the curve graph is induced by some mapping class. This has been the fundamental tool in the proofs of some very strong theorems about mapping class groups, such as quasiisometric rigidity \cite{behrstockkleinerminskymosher:geometry,behrstockhagensisto:quasiflats,bowditch:large:mapping} and commensuration of the Johnson kernel, Torelli group, and other more general normal subgroups \cite{brendlemargalit:commensurations,brendlemargalit:normal}. There are also versions of Ivanov's theorem in the $\Out F_n$ context \cite{aramayonasouto:automorphisms,horbezwade:automorphisms} and for the contact graphs of many CAT(0) cube complexes \cite{fioravanti:automorphisms}. We prove the following analogue for CAT(0) spaces. 

\begin{introthm} \label{ithm:ivanov}
Let $X$ be a proper CAT(0) space with the geodesic extension property. If any one of the following holds, then $\isom X=\isom\X$.
\begin{itemize}
\item    $X$ admits a proper cocompact action by a group that is not virtually free, or
\item    $X$ is a tree that does not embed in $\R$, or
\item    $X$ is one-ended.
%\item   Every geodesic of $X$ is \emph{virtually higher-rank} (see Section~\ref{subsec:Ivanov}).
\end{itemize}
\end{introthm}

\noindent Note that Theorem~\ref{ithm:ivanov} gives exact \emph{equalities} between $\isom X$ and $\isom\X$, rather than just isomorphisms, because $\isom X$ is always a subset of $\isom\X$. We also point out that Theorem~\ref{ithm:ivanov} applies in particular to higher-rank symmetric spaces and buildings, which may be surprising, as it follows from Theorem~\ref{thmi:rank-one_iff_skewers_iff_QI} that their curtain models are bounded.

The statement of Theorem~\ref{ithm:ivanov} cannot hold in full generality, as one can see by considering the real line. Indeed, for any order-preserving bijection $\phi:(0,1)\to(0,1)$, there is an element of $\isom \R_{\Dist}$ whose restriction to each component of $\R\smallsetminus\Z$ is $\phi$. If $\phi$ is not the identity, then this map is not an isometry of $\R$, it is merely a $(1,1)$--quasiisometry. This also shows that the following result is optimal.

\begin{introthm} \label{ithm:1.1}
Let $X$ be a CAT(0) space. Every isometry of $\X$ is a $(1,1)$--quasiisometry of~$X$. 
\end{introthm}

\noindent There is often an important distinction to be made between quasiisometries with multiplicative constant one (also known as \emph{rough isometries}) and more general quasiisometries, as rough isometries tend to preserve more geometric structure. For instance, four-point hyperbolicity is preserved by rough isometries, but not by general quasiisometries \cite[E.g.~11.36]{drutukapovich:geometric}.

Our route to proving Theorem~\ref{ithm:ivanov} is to connect the group $\isom \X$ with Andreev's work in CAT(0) spaces on Aleksandrov's problem \cite{andreev:aleksandrov}. The problem asks for which metric spaces it is the case that any self-map that setwise sends unit spheres to unit spheres is necessarily an isometry \cite{aleksandrov:mappings}. This problem originates from the Beckman--Quarles theorem \cite{beckmanquarles:onisometries}, which shows that this holds for Euclidean $n$--space ($n>1$).

%%%%%%%%%%%%%%%%%%%%%%%%%%%%%%%%%%%%%%%%
\subsection{Large-scale properties} \label{subsec:large-scale}
One of the main open problems in CAT(0) geometry is the \emph{rank-rigidity} conjecture \cite{ballmannbuyalo:periodic}, which asks for a CAT(0) version of the celebrated theorem for Hadamard manifolds \cite{ballmannbrineberlein:structure,ballmannbrinspatzier:structure,ballmann:nonpositively,burnsspatzier:manifolds,eberleinheber:differential}. Even partial progress on this has been quite hard, leading Behrstock--Dru\c tu to ask the simpler question of whether a proper cocompact CAT(0) space $X$ without rank-one isometries must be wide, i.e. no asymptotic cone of $X$ has a cut point \cite[Q.~6.10]{behrstockdrutu:divergence}. We show the following.

\begin{introthm}[Rank dichotomy] \label{icor:rr}
Let $G$ be a group acting properly cocompactly on a CAT(0) space $X$. Exactly one of the following holds.
\begin{itemize}
\item 	The curtain model has bounded diameter, in which case $G$ is wide.
\item 	The curtain model is unbounded, in which case $G$ has a rank-one element, and if $G$ is not virtually cyclic then it is acylindrically hyperbolic.
\end{itemize}
\end{introthm}

In particular, the above corollary provides a positive answer to the question of Behrstock--Dru\c tu. After establishing Theorem~\ref{icor:rr}, we were surprised to discover that their question had already been answered, by Kent--Ricks \cite[p.1467]{kentricks:asymptotic}. However, our proof uses completely different methods, being based on the notions and tools we develop in this paper, including curtains, the curtain model $\X$, and the characterisations of rank-one isometries via separated curtains (Theorem~\ref{thmi:rank-one_iff_skewers_iff_QI} and Theorem~\ref{ithm:dichotomy}).

In the setting of CAT(0) cube complexes, the full rank-rigidity conjecture was proved by Caprace--Sageev \cite{capracesageev:rank}, with the proof relying heavily on the discrete combinatorial structure of hyperplanes. More recently, Stadler has established an important part of the conjecture for CAT(0) spaces of rank two \cite{stadler:cat}.

\medskip

The curtain model also allows us to gain new insights on the visual boundary $\partial X$ of $X$. Indeed, $\X$ comes equipped with its Gromov boundary $\partial \X$, and one can ask how $\partial \X$ is related to the visual boundary $\partial X$ of $X$. The relationship turns out to be remarkably strong.

\begin{introthm} \label{ithm:boundaries}
Let $X$ be a proper CAT(0) space. The space $\partial \X$ embeds homeomorphically as an $\isom X$-invariant subspace of $\partial X$, and every point in the image of $\partial \X$ is a visibility point of $\partial X$. The embedding is induced by the change-of-metric map $\X\to X$. If $\partial\X$ is nonempty and $X$ is cobounded, then $\partial\X$ is dense in $\partial X$.
\end{introthm}

The existence of a contracting geodesic in $X$ implies that $\partial\X$ is nonempty, but it is not a necessary condition; see Example~\ref{eg:XD_boundary}. Results similar to Theorem~\ref{ithm:boundaries} for curve graphs \cite{klarreich:boundary,hamenstadt:train} and the free factor complex \cite{bestvinareynolds:boundary} have proved to be rather useful, for instance as a tool for studying random walks on mapping class groups and $\Out F_n$ \cite{kaimanovichmasur:poisson,maher:random,horbez:poisson}. Work of Le Bars relies on Theorem~\ref{ithm:boundaries} to analyse the asymptotic behaviour of random walks in CAT(0) spaces, as discussed in Section~\ref{subsec:further_directions}. We remark that it follows from the non-periodic version of Theorem~\ref{thmi:rank-one_iff_skewers_iff_QI} (Corollary~\ref{cor:contracting_qie:D}) that $\partial\X$ contains the \emph{Morse boundary} of $X$, \cite{charneysultan:contracting,cordes:morse}; in general it will be a much larger subspace of $\partial X$ than the Morse boundary. In fact, upcoming work of Vest on \emph{sublinearly Morse boundaries} shows that the $\kappa$--boundary of $X$ continuously injects in $\partial\X$ for most sublinear functions $\kappa$ \cite{vest:curtain}.

%%%%%%%%%%%%%%%%%%%%%%%%%%%%%%%%%%%%%%%%
\subsection{Constructing {\texorpdfstring{$\X$}{XD}}, and non-geodesic hyperbolic spaces} \label{subsec:aux_spaces}

As mentioned, although we stated everything in terms of $\X$, the majority of the body of the paper is concerned with a family $\{X_L\}_{L=1}^\infty$ of auxiliary spaces, and only in the last section (Section~\ref{sec:universal}) do we distil all of this information into the single space $\X$ and show that many of the results proven for the family $\{X_L\}$ carry over.

Roughly speaking, the spaces $X_L = (X, \dist_L)$ are defined as follows. We say that a \emph{chain} is a disjoint set of curtains separating two points, and two curtains $h_1, h_2$ are \emph{$L$--separated} if any chain whose elements all intersect both $h_1$ and $h_2$ has size at most $L$. The distance $\dist_L$ between two points is essentially the size of the largest chain separating them whose elements are pairwise $L$--separated. Various notions of (cubical) separation have appeared in the literature, for example \cite{capracesageev:rank,behrstockcharney:divergence,charneysultan:contracting,chatterjimartin:note,genevois:hyperbolicities}. Our notion is a CAT(0) version of the one introduced by Genevois, which he uses to define a hyperbolic distance for cube complexes that is similar to $\dist_L.$ 

Purely for intuitional purposes, requiring that two curtains are $L$--separated can be thought of as requiring that there is a \say{bottleneck} between them of width at most $L$, and then $\dist_L(x,y)$ measures the time spent in such bottlenecks by the geodesic from $x$ to $y$. Since increasing $L$ allows for larger bottlenecks, the identity map $(X, \dist_{L+1}) \to (X, \dist_{L})$ is 1--Lipschitz. If $X$ is hyperbolic, then one would expect that the distance $\dist_L$ would coarsely coincide with $\dist$, and indeed:

\begin{introthm}\label{thmi:X_hypebolic_iff_X_is_XL}
A CAT(0) space $X$ is hyperbolic if and only if there exists $L_0$ such that for each $L\geq L_0$, the identity map is a quasiisometry between each pair among $X$, $X_L$, and $\X$.
\end{introthm}

In particular, Theorem~\ref{thmi:X_hypebolic_iff_X_is_XL} provides a case where the family $\{X_L\}$ \say{stabilises}, and a little experimentation shows that in many other examples it also happens that $X_{L}$ and $X_{L+1}$ are quasiisometric for large enough $L$. Thus, a first approach to obtaining a single space that encodes all the data collectively witnessed by the family $\{X_L\}$ might be to take a limit $\lim_{L\to \infty}\dist_L(x,y)$. However, there are simple examples where the family does not stabilise and the resulting space is not hyperbolic. To solve this, we need a more conceptually sophisticated procedure.  

Choose a full-support probability distribution on the natural numbers and let $Y$ be the associated random variable. (For technical reasons we require the distribution to have finite second moment.) Then, for $x,y \in X$, we define $\Dist (x,y)$ to be the expected value $\mathbf{E}[\dist_Y(x,y)]$. More concretely, this means that for each $L\in\mathbf N$, we choose a weight $\lambda_L>0$ (satisfying certain properties) and define $\Dist(x,y) = \sum_{L=1}^\infty \lambda_L\dist_L(x,y)$. While the second formulation is more suited for computations, we hope that the first can be exported to other situations where geometric properties are encoded by a family of metrics.

A technical point that is worth mentioning is that the spaces $X_L$ and $\X$ are discrete, and in particular not geodesic (although they are roughly geodesic by Theorem~\ref{ithm:XD_hyperbolic}). There are a few inequivalent notions of hyperbolicity for non-geodesic metric spaces. In the interests of clarity and self-containment, we record the following proposition, which does not seem to appear in the literature. We denote by $E(X)$ the \emph{injective hull} of $X$, and refer to the appendix for precise definitions.

\begin{introprop}\label{iprop:defn_hyperbolic}
Let $X$ be a metric space. The following are equivalent. 
\begin{itemize}
\item   $X$ is quasihyperbolic (quasigeodesic triangles are thin) and weakly roughly geodesic. \label{itemi:quasihyperbolic}
\item   $X$ is hyperbolic in the sense of the four-point condition and weakly roughly~geodesic.\label{itemi:four_point}
\item   $X$ is $\isom X$--equivariantly roughly isometric to the geodesic space $E(X)$, and $E(X)$ is \mbox{hyperbolic}. \label{itemi:coarsely_injective}
\end{itemize}
\end{introprop} 

For us a space will be hyperbolic if it satisfies Proposition~\ref{iprop:defn_hyperbolic}.

%%%%%%%%%%%%%%%%%%%%%%%%%%%%%%%%%%%%%%%%
\subsection{Further directions} \label{subsec:further_directions}

The framework of curtains and curtain models allows for results that are strikingly similar to ones for cube complexes and mapping class groups. This opens up a large range of potential directions for further study. We briefly discuss a few of these below.

\smallskip

\noindent\textbf{Random walks.}

Thanks to work of Maher--Tiozzo \cite{mahertiozzo:random}, much can be said about random walks on groups that act by isometries on hyperbolic spaces. Since every action on a CAT(0) space $X$ descends to an action on the curtain model, this now includes the class of CAT(0) groups. As suggested in \cite{lebars:random}, it would be natural to try to use curtain models to obtain results about random walks on CAT(0) groups, for example a description of the Poisson boundary similar to those for mapping class groups \cite{kaimanovichmasur:poisson} and $\Out F_n$ \cite{horbez:poisson}. One could also make use of the cubical perspective, and try to use curtains and Theorem~\ref{ithm:boundaries} to emulate the strategy of \cite{fernoslecureuxmatheus:random} to prove a central limit theorem for random walks on CAT(0) groups. This strategy has recently been successfully implemented by Le Bars~\cite{lebars:central} (see also \cite{choi:limit}).

\smallskip

%\noindent\textbf{Stabilisation of the spaces $X_L$.}

%One of the downsides of the construction of the spaces $X_L$ is that we end up with a family of hyperbolic spaces that potentially represent infinitely many isometry- (or even \mbox{quasiisometry-)} types. Whilst much can still be said just from knowing that every contracting geodesic embeds in \emph{some} $X_L$, there are potential applications where it would be useful to know that there is some $L_0$ beyond which the $X_L$ stop changing, if only up to quasiisometry. 

%The exact stabilisation happens in the case of universal covers Salvetti complexes of right-angled Artin groups, and the coarse stabilisation occurs for hyperbolic spaces (Corollary~\ref{cor:hyp_iff_XL_and_X_QI}). However, it seems unlikely that this should hold in full generality. Indeed, ongoing work of Shepherd shows that, even in the better-behaved cubical case, there exists a cocompact CAT(0) cube complex where Genevois's cubical hyperbolic models \cite{genevois:hyperbolicities} do not stabilise, even up to quasiisometries \cite{shepherd:cubulation}. On the other hand, if the cube complex has a factor system then Genevois's spaces do stabilise \cite{genevois:hyperbolicities,murrayqingzalloum:sublinearly}. It would be interesting to have criteria for these stabilisations to occur.

%\smallskip
 
\noindent\textbf{Quasiisometry-invariance.}

If $X$ and $Y$ are quasiisometric cobounded CAT(0) spaces, are their curtain models quasiisometric? What if $X$ and $Y$ both admit proper cocompact actions by a common group? Examples of Croke--Kleiner show that such a statement is not a given \cite{crokekleiner:spaces}. A positive answer would provide new group invariants for CAT(0) groups, such as the curtain model and its boundary. Note that there are quasiisometric CAT(0) spaces $X$ and $Y$, produced by Cashen \cite{cashen:quasiisometries}, such that $\X$ is (roughly) an infinite wedge of rays but $Y_{\Dist}$ is $\R$ with a ray glued at each integer. These spaces are not cobounded.

% Given two quasiisometric CAT(0) spaces $X$ and $Y$, it is natural to ask whether the families $\{X_L\}$ and $\{Y_L\}$ are related to each other. If one considers only a single value of $L$, then in general $Y_L$ is not quasiisometric to $X_L$. Indeed $X=\R$ and $Y=\R\times[0,1]$ are quasiisometric, but $X_0$ is unbounded and $Y_0$ is bounded. However, $X_L$ and $Y_L$ are quasiisometric for $L\geq 1$. This suggests asking: if $X$ and $Y$ are quasiisometric, is it true that for each~$L$ there exists $L'$ such that $X_L$ quasiisometrically embeds into $Y_{L'}$ and $Y_L$ quasiisometrically embeds into $X_{L'}$? What if $X$ and $Y$ both admit proper cocompact actions by a common group? Examples of Croke--Kleiner show that such a statement is not a given \cite{crokekleiner:spaces}.

\smallskip
 
%\noindent \textbf{Acylindricity of the action.} \dcomment{In favor of removing}

%Let $G$ be a group acting properly on $X$, so that $G$ acts non-uniformly acylindrically on every $X_L$ by Theorem~\ref{introthm:non_uniform_acyl}. In general, these actions may fail to be acylindrical \cite{shepherd:cubulation}, but are there conditions that guarantee acylindricity? If, in addition, the spaces $X_L$ stabilise, i.e. there exists some $X_L$ quasiisometric to $\X$, then this would provide a candidate for a largest acylindrical action of $G$ \cite{abbottbalasubramanyaosin:hyperbolic}. To obtain acylindricity, an obstruction that needs to be addressed is a lower bound on the stable translation length of the $G$-action on $X_L$. Indeed, by \cite[Lem.~2.2]{bowditch:tight}, all acylindrical actions have such a lower bound, otherwise one could find arbitrarily many elements that coarsely stabilise any two points on an axis. 
 
%\smallskip

\noindent\textbf{Roller boundaries.}

In CAT(0) cube complexes there is a notion of boundary defined using hyperplanes, namely the Roller boundary. In it, a point is defined as a choice of orientation for each halfspace, or, more technically, as a non-principal ultrafilter over the set of orientations of hyperplanes.  By including curtains in one of their halfspaces, the construction can be extended to the general CAT(0) setting. One can hope to exploit this boundary to import cubical results. For instance, in \cite{fernoslecureuxmatheus:contact} it is shown that the set of regular points of the Roller boundary is homeomorphic to the Gromov boundary of the contact graph. % If it turns out that the spaces $X_L$ are preserved under quasiisometry, a statement of the previous type would suggest that (part of) the Roller boundary is also preserved under quasi-isometry, potentially providing a new quasiisometry-invariant.

\smallskip

\noindent\textbf{Sublinearly Morse geodesics.}

Recent work of Qing--Rafi--Tiozzo introduces the notion of \emph{sublinearly Morse} geodesics, which generalise contracting geodesics in the CAT(0) setting, and uses them to define the sublinearly Morse boundary \cite{qingrafitiozzo:sublinearly:2}. In Theorem~\ref{thmi:rank-one_iff_skewers_iff_QI} we characterise contracting geodesics in terms of separated curtains and $\X$. Is it possible to obtain a characterisation for the sublinear case similar to the ones in \cite{durhamzalloum:geometry} and \cite{murrayqingzalloum:sublinearly}? And what can be said about the topology on the boundary? Upcoming work of Vest addresses some of these questions \cite{vest:curtain}.

\smallskip

\noindent\textbf{Combination theorems.} 

Given some kind of geometric decomposition of a CAT(0) space $X$ into CAT(0) pieces, what can be said about the curtain model of $X$ from the smaller pieces? One instance would be a space $X$ admitting a geometric action by a group $G$ that admits a graph of groups decomposition. Another would be a CAT(0) space $X$ hyperbolic relative to a family $\{Y^i\}$ of CAT(0) subspaces: is $\X$ hyperbolic relative to the family $\{Y_{\Dist}^i\}$? Is the cone-off of $X$ with respect to $\{Y^i\}$ quasiisometric to the cone-off of $\X$ with respect to $\{Y_{\Dist}^i\}$?

\smallskip

\noindent\textbf{Acylindrical actions.}

Theorem~\ref{introthm:non_uniform_acyl} states that the curtain model is universal for WPD actions. In practice, this result is deduced from the stronger fact that any group acting properly on a CAT(0) space $X$ acts \emph{non-uniformly acylindrically} on $\X$ (Proposition~\ref{prop:nonuniform_acyl:D}). In the cubical setting, Shepherd has recently constructed an example of a rank-one cocompactly cubulated group that does not act acylindrically on the contact graph \cite{shepherd:cubulation}. This raises the question of whether non-uniform acylindricity is optimal: we suspect that (perhaps a variant of) Shepherd's group does not act acylindrically on its curtain model.

%%%%%%%%%%%%%%%%%%%%%%%%%%%%%%%%%%%%%%%%
\subsection*{Summary of proof locations}

\noindent We list the statements in the body of the paper from which the results described above \mbox{follow}. 
$\,$ Theorem~\ref{ithm:XD_hyperbolic} is Theorem~\ref{thm:D_hyperbolic} and Proposition~\ref{prop:cat0_rough_geodesic:D}. 
$\,$ Theorem~\ref{thmi:rank-one_iff_skewers_iff_QI} is Corollary~\ref{cor:contracting_qie:D} and Theorem~\ref{thm:rank_one_characterisation}. 
$\,$ Theorem~\ref{thmi:stable} is Corollary~\ref{cor:stable:D}. 
$\,$ Theorem~\ref{introthm:non_uniform_acyl} is Proposition~\ref{prop:nonuniform_acyl:D}. 
$\,$ Theorem~\ref{ithm:A/QI} is Theorem~\ref{thm:A/QI:D}.
$\,$ Theorem~\ref{ithm:dichotomy} is Theorem~\ref{thm:separated_implies_contracting}. 
$\,$ Theorem~\ref{ithm:ivanov} is Corollary~\ref{cor:one-ended}, Theorem~\ref{thm:ivanov}, and Remark~\ref{rem:ivanov:D}.
$\,$ Theorem~\ref{ithm:1.1} is Corollary~\ref{cor:quasiivanov}.
$\,$ Theorem~\ref{icor:rr} is Theorem~\ref{thm:wide_dichotomy} and Remark~\ref{rem:XD_diameter}.
$\,$ Theorem~\ref{ithm:boundaries} is Theorem~\ref{thm:boundary:D} and Corollary~\ref{cor:contracting_qie:D}.
$\,$ Theorem~\ref{thmi:X_hypebolic_iff_X_is_XL} is Corollary~\ref{cor:hyp_iff_XL_and_X_QI} and Proposition~\ref{prop:qie_in_XD_iff_XL}. 
$\,$ Proposition~\ref{iprop:defn_hyperbolic} is Proposition~\ref{prop:Rips_coarsely_injective}.

%%%%%%%%%%%%%%%%%%%%%%%%%%%%%%%%%%%%%%%%
\subsection*{Acknowledgements} 

\noindent We are pleased to thank Carolyn Abbot, Macarena Arenas, Daniel Berlyne, Pierre-Emmanuel Caprace, Sami Douba, Matt Durham, Talia Fern\'os, Thomas Haettel, Mark Hagen, Sam Hughes, Annette Karrer, Alice Kerr, Corentin Le Bars, Johanna Manhagas, Mahan Mj, Thomas Ng, Jacob Russell, Michah Sageev, Sam Shepherd, Alessandro Sisto, Samuel Taylor, Yvon Verberne, and especially Anthony Genevois for interesting questions, comments, and discussions related to this work. DS and AZ are grateful to the LabEx CARMIN (ANR-10-LABX-59-01) and the Institut Henri Poincar\'e (UAR 839 CNRS-Sorbonne Université) for their hospitality during the thematic trimester of Spring 2022, where part of this work was carried out. HP was partially supported by an EPSRC DTP at the University of Bristol and DS was partially supported by Christ Church Research Centre.

%%%%%%%%%%%%%%%%%%%%%%%%%%%%%%%%%%%%%%%%%%%%%%%%%%%%%%%%%%%%%%%%%%%%%%%%%%%%%%%%%%%%%%%%%%%%%%%%%%%%
\section{Curtains and the \texorpdfstring{$L$}{L}--separation space}

We refer the reader to \cite[Part II]{bridsonhaefliger:metric} for a thorough treatment of CAT(0) spaces. There are two main properties that we shall use directly. The first is that for any geodesic $\alpha:I\to X$, where $I\subset\R$ is an interval, the closest-point projection map, written $\pi_\alpha:X\to\alpha$, is 1--Lipschitz. The second is the fact that CAT(0) spaces are geodesic metric spaces with a \emph{convex} metric. That is, given any two geodesic segments $\alpha,\beta:[0,1]\to X$, parametrised proportional to arc-length, the function $t\mapsto\dist(\alpha(t),\beta(t))$ is convex. An immediate consequence of this is that each pair of points $(x,y)$ is joined by a unique geodesic. We denote this geodesic $[x,y]$.

%%%%%%%%%%%%%%%%%%%%%%%%%%%%%%%%%%%%%%%%
\subsection{Curtains}

The main ingredient of this paper is the concept of \emph{curtains}, which morally mimic hyperplanes in CAT(0) cube complexes. The following appears in simplified form in the introduction.

\begin{definition}[Curtain, pole] 
Let $X$ be a CAT(0) space and let $\alpha:I\to X$ be a geodesic. For a number $r$ with $[r-\frac12,r+\frac12]$ in the interior of $I$, the \emph{curtain dual to $\alpha$ at $r$} is
\[
h=h_\alpha=h_{\alpha,r}=\pi^{-1}_\alpha(\alpha[r-\frac12,r+\frac12]).
\] 
The segment $\alpha[r-\frac12,r+\frac12]$ is called the \emph{pole} of the curtain.
\end{definition}

Although the analogy between curtains and hyperplanes is not perfect, they do share a number of important properties. For instance, curtains separate the space into two halfspaces.  

\begin{definition}[Halfspaces, separation] 
Let $X$ be a CAT(0) space and let $h=h_{\alpha,r}$ be a curtain. The \emph{halfspaces} determined by $h$ are $h^-=\pi_\alpha^{-1}\alpha(I\cap(-\infty,r-\frac12))$ and $h^+=\pi_\alpha^{-1}\alpha(I\cap(r+\frac12,\infty))$. Note that $\{h^-,h,h^+\}$ is a partition of $X$. If $A$ and $B$ are subsets of $X$ such that $A \subseteq h^{-}$ and $B \subseteq h^{+}$, then we say that $h$ \emph{separates} $A$ from $B$.
\end{definition}

\begin{remark}[Curtains are thick and closed] \label{rem:thick_curtains}
Because $\pi_\alpha$ is 1--Lipschitz, we have $\dist(h_\alpha^-,h_\alpha^+)=1$. Moreover, curtains are closed subsets because closest-point projection to a geodesic is continuous and $[r-\frac12,r+\frac12]$ is closed.
\end{remark}

\begin{remark}[Failure of convexity] \label{rem:nonconvex}
Unlike hyperplanes in a CAT(0) cube complex, curtains can fail to be convex: it is possible to have two geodesic segments $\alpha\colon [a,b] \to X$ and $ \beta$ such that $\pi_\beta(\alpha(a)) = \pi_\beta(\alpha(b))$, but $\pi_\beta(\alpha)$ is not constant; in particular, it can happen that $\alpha(a),\alpha(b)\in h_\beta^-$ but $\alpha(t)\in h_\beta^+$ for some $t\in(a,b)$. We thank Michah Sageev for informing us of this fact and providing an example. Another example appears as \cite[Ex.~5.1]{piatek:viscosity}---it is reproduced in Figure~\ref{Figure:Failure_convexity}. Note that such a configuration appears, for instance, in the product of two (non--2-valent) trees.

Whilst the lack of convexity may seem like a failure, it is actually to be expected, for a similar phenomenon occurs in complex hyperbolic geometry. Indeed, every totally geodesic submanifold of $\bf H^n_{\bf C}$ is either totally real or complex-linear \cite[\S3.1.11]{goldman:complex}, and in particular has real-codimension at least 2 when $n\ge2$. Nonetheless, Mostow considered bisectors $\{z\in\bf H^n_{\bf C}\::\:\dist(x,z)=\dist(z,y)\}$, calling them \emph{spinal surfaces} \cite{mostow:onremarkable} (for $n=2$ these were earlier considered by Giraud \cite{giraud:surcertaines}), and using them in the construction of nonarithmetic lattices in $\PU(2,1)$. These spinal surfaces are spiritually similar to the curtains we consider here.
\end{remark}

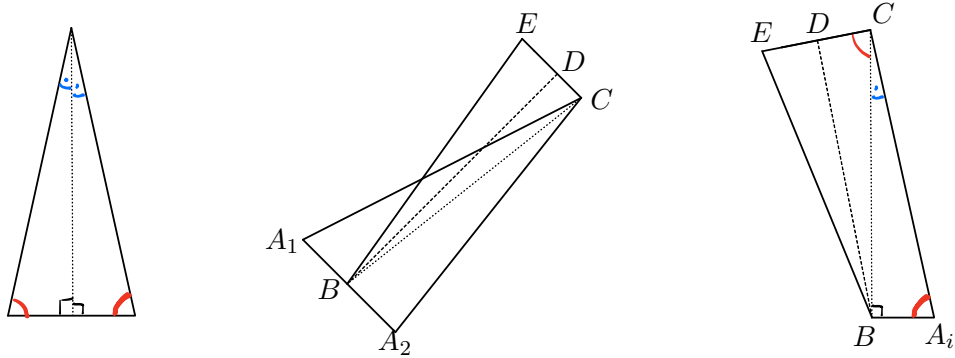
\begin{figure}[ht]
\centering
\begin{tikzpicture}
\node[anchor=south west,inner sep=0] (image) at (0,0) {\includegraphics[width=0.8\textwidth]{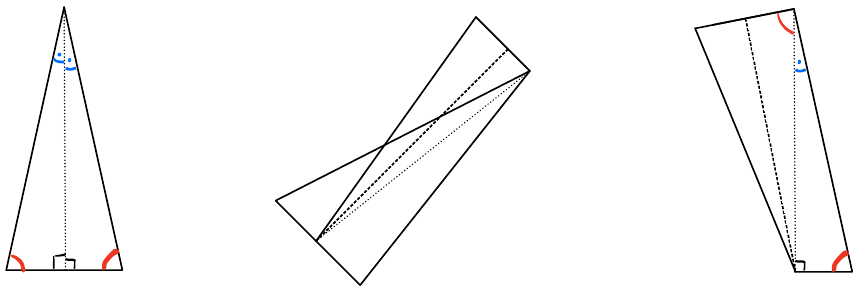}};
    \begin{scope}[x={(image.south east)},y={(image.north west)}]
    \node at (0.3, 0.3) {$A_1$};
    \node at (0.35, 0.15) {$B$};
    \node at (0.42, -0.02) {$A_2$};
    \node at (0.64, 0.75) {$C$};
    \node at (0.56, 0.99) {$E$};
    \node at (0.61, 0.87) {$D$};
    \node at (0.81, 0.96) {$E$};
    \node at (0.87, 0.99) {$D$};
    \node at (0.94, 1.02) {$C$};
    \node at (0.92, 0) {$B$};
    \node at (1, 0) {$A_i$};
    \end{scope}
\end{tikzpicture}
\caption{Glue together two copies of the isosceles triangle on the left to obtain the central CAT(0) space \cite[Ex.~5.1]{piatek:viscosity}. The right-hand quadrilateral has angle $\frac\pi2$ at $C$, so $\pi_{[E,C]}(A_1) = \pi_{[E,C]}(A_2) = C$, but $\pi_{[E,C]}(B) = D$.} \label{Figure:Failure_convexity}
\end{figure}

Although curtains need not be convex, Lemmas~\ref{lem:curtains_separate},  \ref{lem:convex_curtains}, and \ref{lem:controlling_curtains} below show that they do enjoy some convexity-like features. 

\begin{lemma}[Curtains separate] \label{lem:curtains_separate}
Let $h=h_{\alpha,r}$ be a curtain, and let $x \in h^{-}$, $y\in h^{+}$. For any continuous path $\gamma:[a,b]\to X$ from $x$ to $y$ and any $t\in[r-\frac12,r+\frac12]$, there is some $c\in[a,b]$ such that $\pi_\alpha\gamma(c)=\alpha(t)$.
\end{lemma}

\begin{proof}
The map $f=\alpha^{-1}\pi_\alpha\gamma:[a,b]\to I$ is continuous, with $f(a)<t<f(b)$, so $c$ is provided by the intermediate value theorem.
\end{proof}

\begin{lemma}[Star convexity] \label{lem:convex_curtains}
Let $h$ be a curtain with pole $P$. For every $x\in h$, the geodesic $[x, \pi_Px]$ is contained in $h$. In particular, $h$ is path-connected.
\end{lemma}

\begin{proof}
Let $h = h_\alpha$. Since $x\in h$, we have  $\pi_\alpha(x)\in P$. By the triangle inequality, $\pi_\alpha[x,\pi_\alpha(x)]=\pi_\alpha(x)$, so $[x,\pi_\alpha(x)]\subset h$.
\end{proof}

\noindent The same argument shows that $h^+$ and $h^-$ are path-connected.

\begin{lemma}[No bigons for related curtains]\label{lem:controlling_curtains}
Let $\alpha=[x_1,x_3]$ be a geodesic and let $x\notin \alpha$. For any $x_2 \in \alpha$, if $h$ is a curtain dual to $[x_2,x]$ that meets $[x_1,x_2]$, then $h$ does not meet $[x_2,x_3]$.
\end{lemma}

\begin{proof}
If $h$ does this then there exist $p_1\in h\cap[x_1,x_2]$, $p_3\in h\cap[x_2,x_3]$ with $\pi_{[x,x_2]}(p_1)=\pi_{[x,x_2]}(p_3)$. Because $x_2\in[x,x_2]\smallsetminus h$, we have $\dist(p_i,x_2)>\dist(p_i,\pi_{[x,x_2]}(p_i))$. But this contradicts the fact that $\alpha$ is a geodesic. 
\end{proof}
 
We also record the following basic property of closest-point projections.

\begin{lemma} \label{lem:comparing_projections_2}
Let $\alpha$ be a geodesic, and let $x\in X$. For any $y\in\alpha$, we have $\pi_{[x,\pi_\alpha x]}(y)=\pi_\alpha(x)$.
\end{lemma}

\begin{proof}
By the triangle inequality, we have $\pi_\alpha(z)=\pi_\alpha(x)$ for all $z\in[\pi_\alpha x,x]$. Because $\pi_{[x,\pi_\alpha x]}$ is 1--Lipschitz, we therefore have $\dist(y,\pi_\alpha x)\le\dist(y,z)$ for all $z\in[\pi_\alpha x,x]$, and the inequality must be strict for $z\ne\pi_\alpha(x)$ as balls are strictly convex.
\end{proof}

Curtains can be used to define a new family of distances on $X$. The first that we shall consider is the \emph{chain distance}.

\begin{definition}[Chain, chain distance] 
A set $\{h_i\}$ of curtains is a \emph{chain} if $h_i$ separates $h_{i-1}$ from $h_{i+1}$ for all $i$. We say that $\{h_i\}$ separates $A,B\subset X$ if every $h_i$ does. The \emph{chain distance} from $x\in X$ to $y\in X\smallsetminus\{x\}$ is
\[
\dist_\infty(x,y)=1+\max\{\hspace{1mm}|c|:c\text{ is a chain separating }x\text{ from }y\}.
\]
\end{definition}

The fact that $\dist_\infty$ is a metric follows from Lemma~\ref{lem:chain_distance} below. Unless specified otherwise, for a curtain $h_i$ in a chain $\{h_i\}$ we orient the half-spaces such that $h_j^+\subset h_i^+$ for all $j>i$. In other words, the order on the curtains induced by the indices and the one induced by the halfspaces coincide. In particular, we can talk of maximal and minimal elements. The following lemma states that the chain distance $\dist_\infty$ and the original distance $\dist$ on $X$ differ by at most 1. Apart from the statement that curtains really do encode ``meaningful'' distances on $X$, the fact that $\dist$ and $\dist_\infty$ are comparable turns out to be surprisingly useful, and will come up in multiple places.

\begin{lemma} \label{lem:chain_distance}
For any $x,y\in X$, there is a chain $c$ of curtains dual to $[x,y]$ such that  $1+|c|=\dist_\infty(x,y)$ and $1+|c|=\lceil\dist(x,y)\rceil$. 
\end{lemma}

\begin{proof}
Let $D=\lceil\dist(x,y)\rceil$, and let $\delta=\dist(x,y)-\lfloor\dist(x,y)\rfloor$ be the fractional part. For $i\in\{1,\dots,D-1\}$, let $r_i=i-\frac12+\frac {i\delta}{D}$. Observe that the intervals $[r_i-\frac12,r_i+\frac12]$ are pairwise disjoint. This implies that the curtains $h_{\alpha,r_i}$ form a chain $c$ of cardinality $\lceil\dist(x,y)\rceil-1$. Hence $\dist_\infty(x,y)\ge\lceil\dist(x,y)\rceil$. On the other hand, for any curtain $h$ we have $\dist(h^-,h^+)=1$. As curtains are closed, it follows that $\dist_\infty(x,y)\le\lceil\dist(x,y)\rceil$.
\end{proof}

%%%%%%%%%%%%%%%%%%%%%%%%%%%%%%%%%%%%%%%%
\subsection{\texorpdfstring{$L$}{L}--separation}

We are now ready to introduce the notion of $L$--separation, which will be fundamental to this article. If curtains are reminiscent of hyperplanes and wall-spaces, then $L$--separation mimics the behaviour of such objects in hyperbolic spaces, and thus should be thought of as a source of negative curvature. For instance, it is possible to induce a wall-space structure on the hyperbolic plane $\mathbf{H}^2$ by considering a surface quotient and using the lifts of appropriate curves to define walls. Such lifts will enjoy strong separation properties, for instance the closest-point projections between them will have uniformly bounded diameter. In our case, curtains are not necessarily convex, so closest-point projections are not always well defined. However, one can achieve similar results by considering how curtains interact with one other.

\begin{definition}[$L$--separated, $L$--chain] \label{def:separated}
Let $L\in\mathbf N$. Disjoint curtains $h$ and $h'$ are said to be \emph{$L$--separated} if every chain meeting both $h$ and $h'$ has cardinality at most $L$. Two disjoint curtains are said to be \emph{separated} if they are $L$-separated for some $L.$ If $c$ is a chain of curtains such that each pair is $L$--separated, then we refer to $c$ as an \emph{$L$--chain}. 
\end{definition}

This notion is inspired by one due to Genevois for cube complexes \cite{genevois:contracting} (see also \cite{behrstockcharney:divergence,charneysultan:contracting}). The next two lemmas will be a staple asset during the paper---their proofs are purely combinatorial and do not use any CAT(0) geometry. 

\begin{lemma}[Gluing disjoint $L$--chains] \label{lem:gluing_disjoint_chains}
Suppose that $c$ and $c'$ are $L$--chains, let $h$ be the maximal element of $c$, and let $h'$ be the minimal element of $c'$. If $c'$ is contained in $h^+$ and $c$ is contained in $h'^-$, then $c\cup c'\smallsetminus\{h\}$ and $c\cup c'\smallsetminus\{h'\}$ are $L$--chains.
\end{lemma}

\begin{proof}
Observe that $c\cup c'$ is a chain: $h$ separates $c\smallsetminus\{h\}$ from $c'$, and $h'$ separates $c$ from $c'\smallsetminus\{h'\}$. The argument is the same for both chains in the statement, so we just consider the former. Let $h''$ be the maximal element of $c\smallsetminus\{h\}$. It suffices to check that $h''$ and $h'$ are $L$--separated, which holds because any chain meeting both $h''$ and $h'$ must meet both $h''$ and $h$, because $h$ separates $h''$ from $h'$.
\end{proof}

From the above proof, note that the technical assumption in Lemma~\ref{lem:gluing_disjoint_chains} is equivalent to assuming that $c\cup c'$ is a chain.

\begin{lemma}[Gluing $L$--chains] \label{lem:gluing_chains}
Suppose that $c=\{\dots, h_{-1},h_0\}$ and $c'=\{h'_1,h_2', \dots\}$ are $L$--chains with $|c|>1$ and $|c'|>L+1$, oriented in the natural way. If $h_0^+\cap h'_j\ne\varnothing$ for all $j$ and ${h'_1}^-\cap h_i\neq\varnothing$ for all $i$, then $c''=\{\dots,h_{-1},h'_{L+2},\dots\}$ is an $L$--chain. Its cardinality is $|c|+|c'|-L-2$ when $c$ and $c'$ are finite.
\end{lemma}

\begin{proof}
Suppose that $h_{-1}$ meets $h'_{L+1}$. Since ${h'_1}^-\subset{h'_{L+1}}^-$ intersects $h_{-1}$, we must also have that $h_{-1}$ meets $h'_1$. Hence $h_{-1}$ meets $h'_j$ for all $j\le L+1$. Similarly, the fact that $h_0^+\subset h_{-1}^+$ intersects $h'_j$ means that $h_0$ meets $h'_j$ for all $j\le L+1$ But this contradicts the fact that $c'$ is an $L$--chain. This shows that $h_{-1}$ is disjoint from $h'_{L+1}$.

Now, the fact that $h_0^+\subset h_{-1}^+$ intersects $h'_{L+1}$ means that $c'\smallsetminus\{h'_1,\dots,h'_{L}\}\subset h_{-1}^+$. Moreover, the fact that ${h'_1}^-\subset{h'_{L+1}}^-$ intersects $h_{-1}$ means that $c\smallsetminus\{h_0\}\subset{h'_{L+1}}^-$. We are therefore in the setting of Lemma~\ref{lem:gluing_disjoint_chains}, and applying it yields the result.
\end{proof}

The following lemma is a good first example of using curtains to obtain strong geometric features via simple combinatorial arguments.

\begin{lemma}[Bottleneck]\label{lem:bounding_distance_with_dual_curtains} 
Suppose that $A,B$ are two sets which are separated by an $L$-chain $\{h_1,h_2,h_3\}$ all of whose elements are dual to a geodesic $b=[x_1,y_1]$ with $x_1 \in A$ and $y_1 \in B$. For any $x_2 \in A$ and $y_2 \in B$, if $p\in h_2 \cap [x_2,y_2]$, we have $\dist(p, \pi_{b}(p))\le2L+1$.
\end{lemma}

\begin{proof} 
Let $c$ be a chain dual to $[p,\pi_b(p)]$ that realises $1+|c|=\lceil\dist(p,\pi_b(p))\rceil$, as provided by Lemma~\ref{lem:chain_distance}. According to Lemma~\ref{lem:comparing_projections_2}, every $z\in b$ has $\pi_{[p,\pi_b(p)]}(z)=\pi_b(p)$, so no element of $c$ meets $b$. Furthermore, Lemma~\ref{lem:controlling_curtains} shows that no element of $c$ can meet both $[x_2,p]$ and $[p,y_2].$ Therefore, each element of $c$ must meet either $h_1$ or $h_2$. Since $\{h_1,h_2,h_3\}$ is an $L$-chain, this means that $c$ has at most $2L$ elements. By the choice of $c$ we have $\dist(p,\pi_b(p)) \leq 2L+1$.
\end{proof}

Given a CAT(0) space $X$, we now use curtains and $L$--separation to define a family of metrics on $X$, similarly to \cite{genevois:contracting}. The metric spaces produced are the auxiliary spaces alluded to in Section~\ref{subsec:aux_spaces}.

\begin{definition}[$L$--metric]
Given distinct points $x,y\in X$, set $\dist_L(x,x)=0$ and define 
\[
\dist_L(x,y) = 1+\max\{\hspace{1mm}|c|:c \text{ is an } L\text{--chain separating }x\text{ from }y\}.
\]
\end{definition}

\begin{remark} \label{rmk:new_metric_bounded}
Since $L$--chains are chains, we have $\dist_L(x,y)\le\dist_\infty(x,y)<1+\dist(x,y)$.
\end{remark}

Let us show that $\dist_L$ is a metric. For $L=\infty$, this also follows from Lemma~\ref{lem:chain_distance}.

\begin{lemma}
The function $\dist_L$ is a metric for every $L\in\mathbf N\cup\{\infty\}$.
\end{lemma}

\begin{proof}
The map $\dist_L$ is clearly symmetric and separates points. Given $x,y,z\in X$, let $c$ be a chain realising $\dist_L(x,y)=1+|c|$. We have $z\in h$ for at most one $h\in c$, and every other element of $c$ separates $z$ from at least one of $x$ and $y$. Let $c'\subset c$ be the subchain of curtains separating $z$ from $x$. We have shown that $\dist_L(x,z)+\dist_L(z,y)\ge (1+|c'|)+(1+|c|-|c'|)-1=1+|c|=\dist_L(x,y)$.
\end{proof}

\bsh{Notation}
For $L\in\mathbf N$, we write $X_L$ to denote the metric space $(X,\dist_L)$.
\esh

% The following is a simple consequence of the lemmas on gluing $L$--chains. Recall that a metric space is \emph{weakly roughly geodesic} if there is a constant $C$ such that for any $x,y\in X$ and any nonnegative $r\le\dist(x,y)$, there is a point $z\in X$ such that $|\dist(x,z)-r|\le C$ and $\dist(x,z)+\dist(z,y)\le\dist(x,y)+C$.

% \begin{lemma} \label{lem:weakly_roughly_geodesic} 
% $X_L$ is weakly roughly geodesic, with constant $L+5$. 
% \end{lemma}

% \begin{proof}
% Let $\{h_1,\dots,h_n\}$ be an $L$--chain realising $\dist_L(x,y)$. Given $0\leq r\leq d(x,y)$, let $z\in h_{\lceil r\rceil}$. Let $c$, $c'$ be $L$--chains realising $\dist_L(x,z)$ and $\dist_L(z,y)$. We know that $|c|\ge r-1$ and $|c'|\ge n-r-1$. According to Lemma~\ref{lem:gluing_chains}, we also have that $|c|+|c'|-(L+3)\le n$, and this shows that $\dist_L(x,z)=|c|+1\le r+L+5$.
% \end{proof}

% Since every weakly roughly geodesic space is quasigeodesic, this implies that $X_L$ is quasigeodesic. In Section~\ref{sec:hyperbolicity}, we shall give more precise information by showing that CAT(0) geodesics of $X$ are uniform quasigeodesics of $X_L$. 

\begin{example} \label{eg:tree_of_flats}
An instructive example to consider is the \emph{tree of flats}, i.e. the Cayley complex $C$ of the right-angled Artin group $\Z^2*\Z=\langle a,b\rangle*\langle c\rangle$. Consider the geodesic $\alpha=[a^{-1},a]$, with curtain $h=h_{\alpha,0}$. The vertices $C^0\cap h$ in $h$ are exactly those corresponding to reduced words whose first letter is not $a$ or $a^{-1}$. See Figure~\ref{fig:tree_of_flats}.

There are two noteworthy things here. Firstly, $h$ is not Hausdorff--close to any hyperplane of $C$ in the cubical sense. Secondly, $h$ contains points that are arbitrarily far from $h^-\cup h^+$, even in the metric $\dist_L$. 

\begin{figure}[ht]
\includegraphics[width=10cm]{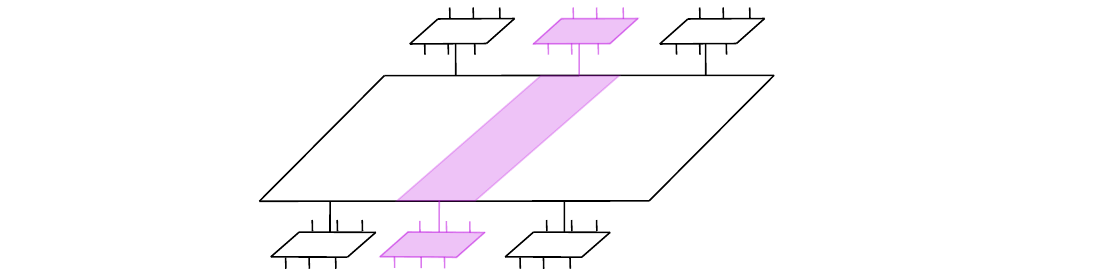}\centering
\caption{A curtain in a CAT(0) cube complex that is not close to any hyperplane.} \label{fig:tree_of_flats}
\end{figure}

Note that there are only three ways that a curtain $h$ can intersect a flat $F$ in this example: the intersection $h\cap F$ is either empty, equal to $F$, or a strip of width at most 1. From this it can be seen that $X_L$ is quasiisometric to the Bass--Serre tree for every $L\ge2$.
\end{example}

The difference between curtains and hyperplanes raises the following question.

\bsh{Question}
How do the spaces $X_L$ defined here using curtains compare to the spaces defined by Genevois using hyperplanes \cite{genevois:hyperbolicities} in the case where $X$ is a CAT(0) cube complex? 
\esh

It is possible to construct quasiline subcomplexes of $\Z^2$ where any correspondence depends on some coboundedness constant. Indeed, let $\gamma\subset\Z^2$ be the ``zigzag'' geodesic that passes through $(n,n)$ and $(n+1,n)$ for all $n$. For each natural number $k$, let $X(k)$ be the CAT(0) cube complex bounded by $\gamma$ and the translation of $\gamma$ by $(0,k)$, which has a $\frac k2$--cobounded $\Z$--action. It can easily be seen that the contact graph of $X(k)$ is a quasiline, so Genevois's spaces are all unbounded \cite[Fact~6.50]{genevois:hyperbolicities}, but the space $X(k)_L$ is bounded for all $L\le\frac{k-3}2$.

\medskip

A disadvantage of the metric $\dist_L$ is that it provides no information about any of the chains of curtains realising the distance between a given pair of points, apart from their length. We conclude the section by proving that in many situations, up to a linear loss in length, we can replace a given $L$--chain by one dual to a fixed geodesic. This will simplify arguments in a number of places.

\begin{lemma}[Dualising chains] \label{lem:updating_curtains} 
Let $L,n\in\mathbf N$, let $\{h_1,\dots,h_{(4L+10)n}\}$ be an $L$--chain, and suppose that $A,B\subset X$ are separated by every $h_i$. For any $x\in A$ and $y\in B$, the sets $A$ and $B$ are separated by an $L$--chain of length at least $n+1$ all of whose elements are dual to $[x,y]$ and separate $h_1$ from $h_{(4L+10)n}$. 
\end{lemma}

\begin{proof} 
Let us first prove the statement for $n=1$, illustrated in Figure~\ref{fig:updating}. Writing $\alpha=[x,y]$, let $a_1$ and $a_2$ be the first points of $\alpha\cap h_3$ and $\alpha\cap h_{2L+8}$ respectively, and let $b_1$ and $b_2$ be the last points of $\alpha\cap h_{2L+3}$ and $h_{4L+8}$, respectively. Since curtains are thick, we have $\dist(a_i,b_i)>2L+1$, so Lemma~\ref{lem:chain_distance} provides chains $\{k^i_0,\dots,k^i_{2L}\}$ dual to $\alpha$ that separate $a_i$ from $b_i$. Since $h_1$ and $h_2$ are $L$--separated, the curtain $k^1_L$ is disjoint from $h_1$, and similarly $k^1_L$ is disjoint from $h_{2L+5}$. The same argument shows that $k^2_L$ is disjoint from $h_{2L+6}\cup h_{4L+10}$. This shows that the $k^i_L$ separate $A$ from $B$, but it also shows that $k^1_L$ and $k^2_L$ are $L$--separated, because any curtain meeting both must also meet $h_{2L+5}$ and $h_{2L+6}$.

Now suppose that $n>1$, and again write $\alpha=[x,y]$. For $j\in\{1,\dots,n-1\}$, let $x_j\in\alpha \cap h_{(4L+10)j}^+ \cap h_{1+(4L+10)j}^-$. Let $x_0=x$, $x_n=y$. The $n=1$ case provides, for each $j<n$, a pair of $L$--separated curtains $k_j^1$ and $k_j^2$ that are dual to $\alpha$ and separate $x_j$ from $x_{j+1}$. Since the $k_j^i$ are all dual to $\alpha$, they are pairwise disjoint and form a chain, so we can repeatedly apply Lemma~\ref{lem:gluing_disjoint_chains} to complete the proof.
\end{proof}

\begin{figure}[ht]
\begin{center} \makebox[0pt]{\begin{minipage}{1.4\textwidth}
\includegraphics[width=\textwidth]{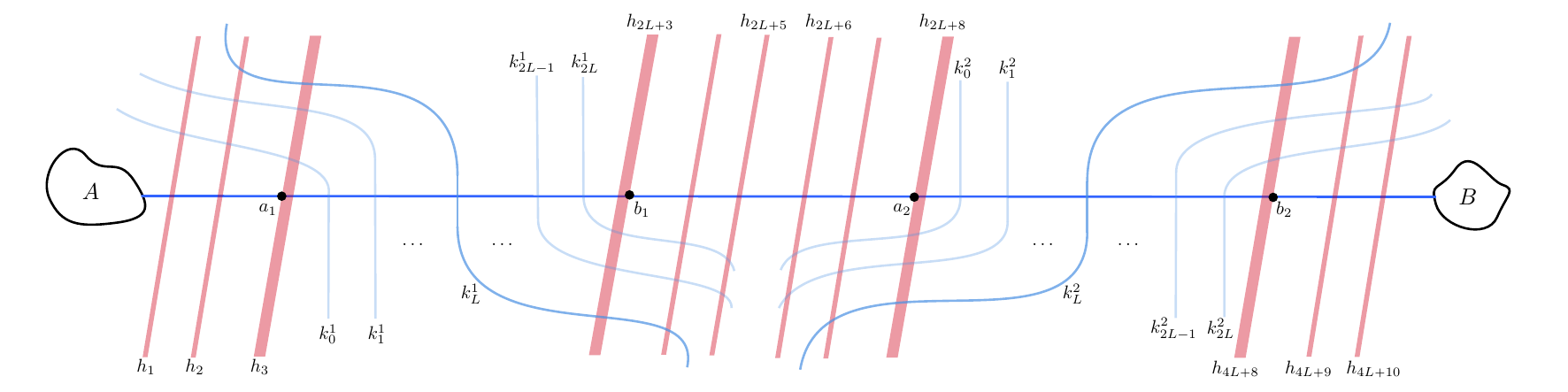}\centering
\caption{The base case of Lemma~\ref{lem:updating_curtains}} \label{fig:updating}
\end{minipage}}\end{center}
\end{figure}

\begin{corollary} \label{cor:diverging_geodesics} 
Let $b,c$ be geodesic rays with $b(0)=c(0)$. If some infinite $L$--chain is crossed by both $b$ and $c$, then $b=c$.
\end{corollary}

\begin{proof}
By Lemma~\ref{lem:updating_curtains}, $b$ and $c$ cross an infinite $L$--chain $\{h_i\}$ dual to $b$. By Lemma~\ref{lem:bounding_distance_with_dual_curtains}, if $c(t_i)\in h_i$, then $\dist(c(t_i),b)<2L+1$. By convexity of the metric, this implies that $c=b$.
\end{proof}

%%%%%%%%%%%%%%%%%%%%%%%%%%%%%%%%%%%%%%%%%%%%%%%%%%%%%%%%%%%%%%%%%%%%%%%%%%%%%%%%%%%%%%%%%%%%%%%%%%%%
\section{Hyperbolicity and isometries} \label{sec:hyperbolicity} 

In this section we begin establishing, for a CAT(0) space $X$, some of the properties of the spaces $X_L=(X,\dist_L)$ that mirror those of the curve graph. In Section~\ref{subsec:hyperbolicity}, we prove that every $X_L$ is hyperbolic (in the sense of the four-point condition). The strategy for this is to apply ``guessing geodesics'', Proposition~\ref{prop:guessing_geodesics}, to the CAT(0) geodesics of $X$. In Section~\ref{subsec:Ivanov}, we prove results on CAT(0) analogues of Ivanov's theorem, as discussed in Section~\ref{isubsec:ivanov}.

%%%%%%%%%%%%%%%%%%%%%%%%%%%%%%
\subsection{Hyperbolicity of the models} \label{subsec:hyperbolicity}

In order to apply the ``guessing geodesics'' criterion to CAT(0)-geodesic triangles, we need to understand their interaction with $L$--separated curtains. We start by bounding the amount that a CAT(0) geodesic can backtrack through separated curtains. Whilst curtains are not themselves convex, this can be thought of as showing that a pair of separated curtains is (almost) a convex object in its own right.

\begin{lemma} \label{lem:convexity_around_L_chains}
Let $h$ and $k$ be $L$--separated, and let $\alpha$ be a CAT(0) geodesic. If there exist $t_1<t_2<t_3<t_4$ satisfying either 
\begin{align*}
&	\alpha(t_1)\in h, \: \alpha(t_2)\in k, \: \alpha(t_3)\in k, \: \alpha(t_4)\in h \\
\text{or }\:\: &	\alpha(t_1)\in h, \: \alpha(t_2)\in k, \: \alpha(t_3)\in h, \: \alpha(t_4)\in k,
\end{align*}
then $t_3-t_2\le L+1$.
\end{lemma}

\begin{proof}
The two cases are treated similarly, so we just consider the former. Let $c$ be a chain dual to $\alpha$ that realises $\dist_\infty(\alpha(t_2),\alpha(t_3))=1+|c|$, as given by Lemma~\ref{lem:chain_distance}. Every curtain in $c$ separates $\alpha(t_1)$ from $\alpha(t_4)$, and hence meets $h$ because $h$ is path-connected. Similarly, every curtain in $c$ meets $k$. Thus $|c|\le L$, so 
\[
\dist(\alpha(t_2),\alpha(t_3)) \le \dist_\infty(\alpha(t_2),\alpha(t_3)) = 1+|c| \le 1+L. \qedhere
\]
\end{proof}

\begin{corollary} \label{cor:double_crossing}
If $\alpha$ is a CAT(0) geodesic and $t_1<t_2<t_3$, then any $L$--chain $c$ separating $\alpha(t_2)$ from  $\{\alpha(t_1),\alpha(t_3)\}$ has cardinality at most $L'=1+\lfloor\frac L2\rfloor$.
\end{corollary}

\begin{proof}
Assume that $|c|\ge2$ and that $\alpha(t_2)\in h^+$ for every $h\in c$. Since $h$ is path-connected, both $\alpha|_{(t_1,t_2)}$ and $\alpha|_{(t_2,t_3)}$ cross $h$. Let $h_1,h_2\in c$ be minimal (namely every element of $c- \{h_1\}$ separates $h_1$ from $\alpha(t_2)$). According to Lemma~\ref{lem:convexity_around_L_chains}, the length of $\alpha\cap(h_2\cup h_2^+)$ is at most $L+1$. Recall that curtains are closed, and that $\dist(h^-,h^+)=1$ for every $h\in c$. For $L\in\{0,1\}$, this gives a contradiction with the fact that $\alpha(t_2)\in h_2^+$, so $|c|\le1$. Otherwise, it implies that $\alpha\cap h_2^+$ has length at most $L-1$, and hence we obtain $|c|\le2+\lfloor\frac{L-2}{2}\rfloor$.
\end{proof}

This lack of backtracking allows us to show that, up to parametrisation, CAT(0) geodesics of $X$ are uniform rough geodesics of $X_L$. (Recall that a $k$--rough isometry is a $(1,k)$--quasiisometry, and a $k$--rough geodesic is a $(1,k)$--quasigeodesic.)

\begin{proposition} \label{prop:cat0_rough_geodesic}
If $\alpha$ is a CAT(0) geodesic and $x,y,z\in\alpha$ have $z\in[x,y]$, then $\dist_L(x,y)\ge\dist_L(x,z)+\dist_L(z,y)-3L-3$. In particular, for each $L$ there is a constant $q_L$, linear in $L$, such that CAT(0) geodesics are unparametrised $q_L$--rough geodesics of $X_L$.
\end{proposition}

\begin{proof}
Let $c_x$ and $c_y$ be maximal $L$--chains separating $z$ from $x$ and $y$, respectively. By Corollary~\ref{cor:double_crossing}, only the first $L$ elements of $c^L_y$ can fail to separate $y$ from $x$. Let $b_y$ be the complement in $c_y$ of those elements, so that all elements of $b_y$ separate $y$ from $[x,z]$. We can similarly remove at most $L$ elements of $c_x$ to obtain an $L$--chain $b_x$ whose elements separate $x$ from $[z,y]$. 

Let $h$ be the last element of $b_x$. Since $h^+$ contains $[z,y]$, the half-space $h^+$ intersects all the elements of $b_y$. Using this and the symmetric argument, we can apply Lemma~\ref{lem:gluing_chains} to remove at most $L+2$ elements of $b_x\cup b_y$ and obtain an $L$--chain $b$ separating $x$ from $y$. We have $|b|\ge|c_x|+|c_y|-3L-2$, so by definition of $c_x$ and $c_y$ it follows that $\dist_L(x,y)\ge\dist_L(x,z)+\dist_L(z,y)-3L-3$.
\end{proof}

\begin{remark}[Synchronous fellow-travelling]
Given two CAT(0) geodesic segments $\alpha$ and $\beta$ issuing from the same point, the convexity of the CAT(0) metric bounds $\dist(\alpha(t), \beta(t))$ in terms of the distance between the endpoints. Quite surprisingly, it turns out that $\dist(\alpha(t), \beta(t))$ can also be bounded in terms of \emph{the $X_L$--distance} between the endpoints, at least on some initial subsegment. In particular, this tells us that CAT(0) geodesics synchronously fellow-travel in $X_L$, with respect to the CAT(0) parametrisation. For reasons of exposition, we postpone the precise statement and proof of this fact to Lemmas~\ref{lem:tails} and~\ref{lem:synchronous}, but the curious reader can skip ahead, as the proofs only require tools developed up to this point.  
\end{remark}

The next step is to show that triangles in $X_L$ whose edges are CAT(0) geodesics are necessarily thin. As we shall be working with triangles, it is useful to  introduce the notion of \emph{triangular chains}, which are a three-directional analogue of ordinary chains. Whilst this notion is not strictly necessary, in the sense that every result proven with triangular chains could be proven by considering three ordinary chains, considering a single object instead of three allows for a number of simplifications.
 
\begin{definition}[Triangular chain, corner] \label{def: triangular chain}
Let $x_1, x_2, x_3\in X$ be distinct points. A \emph{triangular chain} for $\{x_i\}$ is a collection $\cal{T}$ of pairwise-disjoint curtains such that each $h\in \cal{T}$ separates $x_i$ from $\{x_{i+1}, x_{i+2}\}$ for some $i$ (indices taken mod 3). The \emph{$x_i$--corner} of $\T$ is the subchain of $\T$ consisting of all curtains separating $x_i$ from $\{x_{i+1}, x_{i+2}\}$. A \emph{triangular $L$--chain} is a triangular chain where every two curtains are $L$--separated. We say that a triangular ($L$--)chain is \emph{maximal} if it is not contained in a larger triangular ($L$--)chain.
\end{definition} 

\noindent Note that the union of the $x_1$-- and $x_2$--corners is a chain separating $x_1$ from $x_2$.
 
In general, there might be many, very different, maximal triangular chains between three points. However, when restricting our attention to maximal triangular $L$--chains, the situation changes, and we can use such a maximal chain to approximate the distance in $X_L$. For points $x_1,x_2,x_3\in X$, we write $[x_1,x_2,x_3]$ for the geodesic triangle $[x_1,x_2]\cup[x_2,x_3]\cup[x_3,x_1]$.
 
\begin{definition}[$\T$--count]
Let $x_1,x_2,x_3\in X$, and let $\cal{T}$ be a triangular chain for $\{x_1,x_2,x_3\}$. For points $x,y \in [x_1,x_2,x_3]$, we denote $|x,y|_\T = |\{h\in \cal{T}\,:\, \  h \text{  separates } x \text{ from } y\}|$. 
\end{definition}

\begin{lemma} \label{lem: triangular distance same as L distance}
Let $T =[x_1, x_2, x_3]$ be a geodesic triangle and $\T$ a maximal triangular $L$--chain for $\{x_1,x_2,x_3\}$. For any $x,y\in T$ we have $\dist_L(x,y) \leq |x,y|_\T + 5L +12$. 
\end{lemma}

\begin{figure}[ht] 
\begin{subfigure}{0.3\textwidth}
    \includegraphics[width=4.5cm, trim = 3cm 11cm 4cm 6cm]{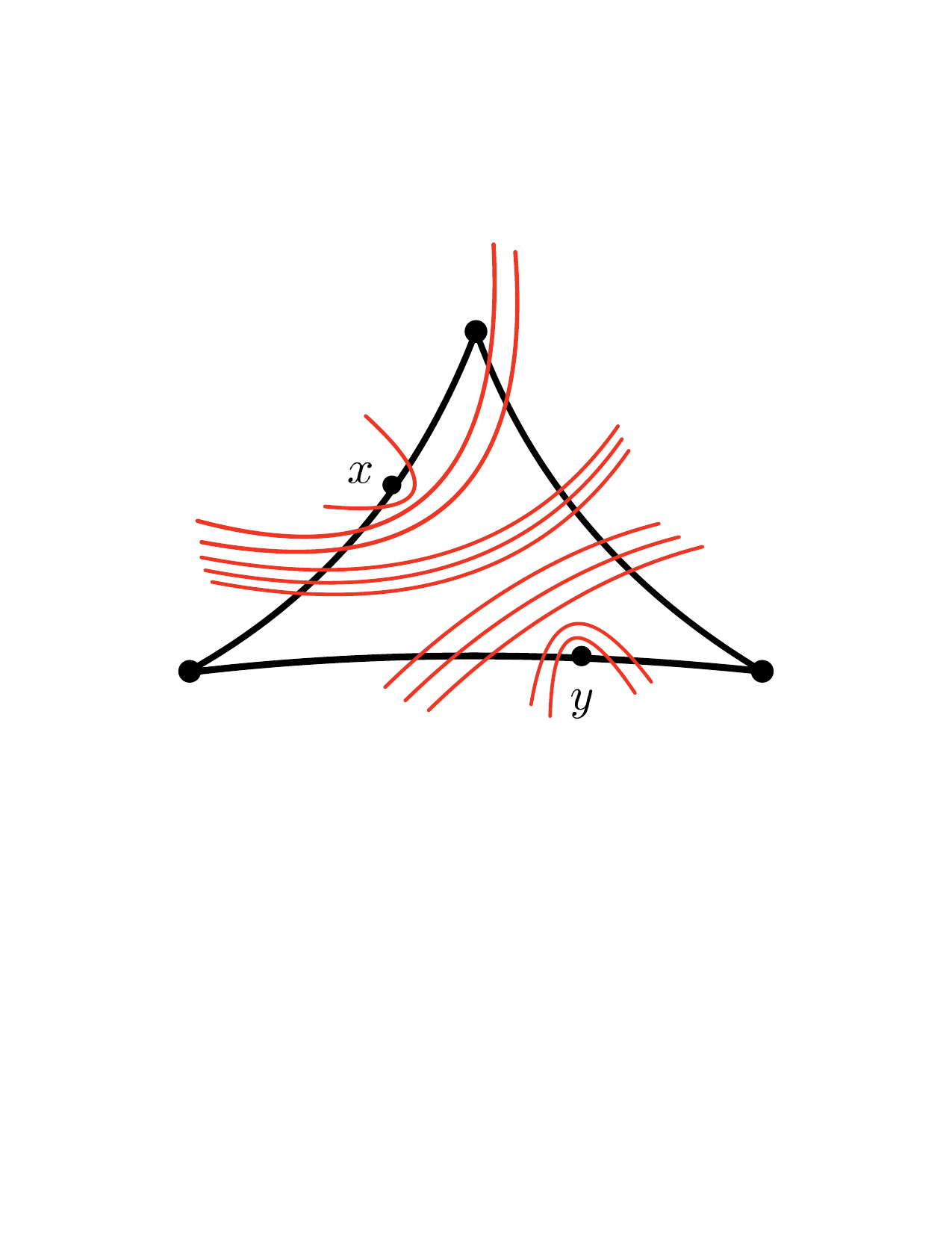}
    \caption*{The $L$--chain $\theta$. \\ { }\quad { }}
\end{subfigure}
\hfill
\begin{subfigure}{0.3\textwidth}
    \includegraphics[width=4.5cm, trim = 3cm 11cm 4cm 6cm]{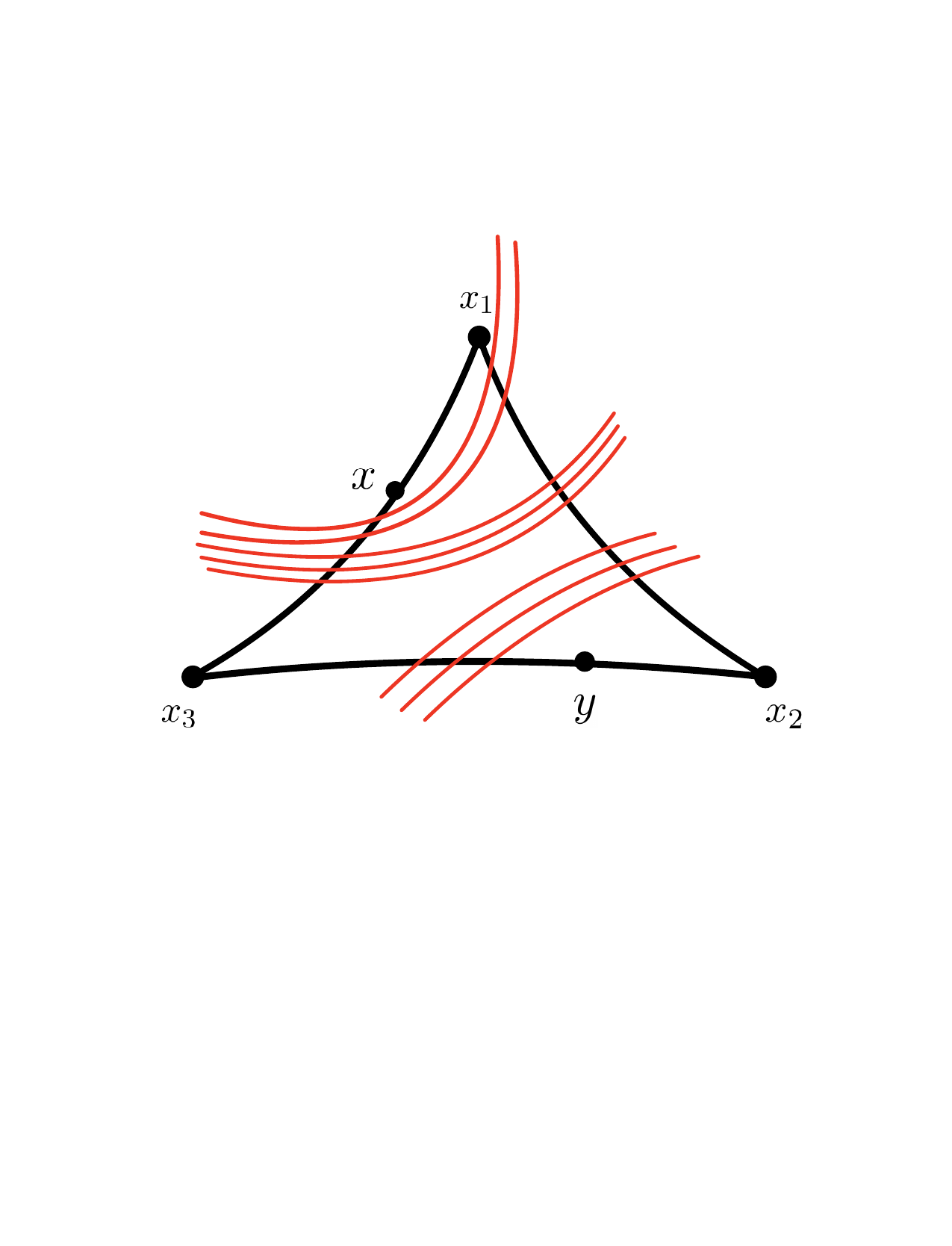}
    \caption*{The vertices can be relabelled so that $\theta'$ looks like this.}
\end{subfigure}
\hfill
\begin{subfigure}{0.3\textwidth}
    \includegraphics[width=5.6cm, trim = 5cm 5.3cm 4cm 6.5cm]{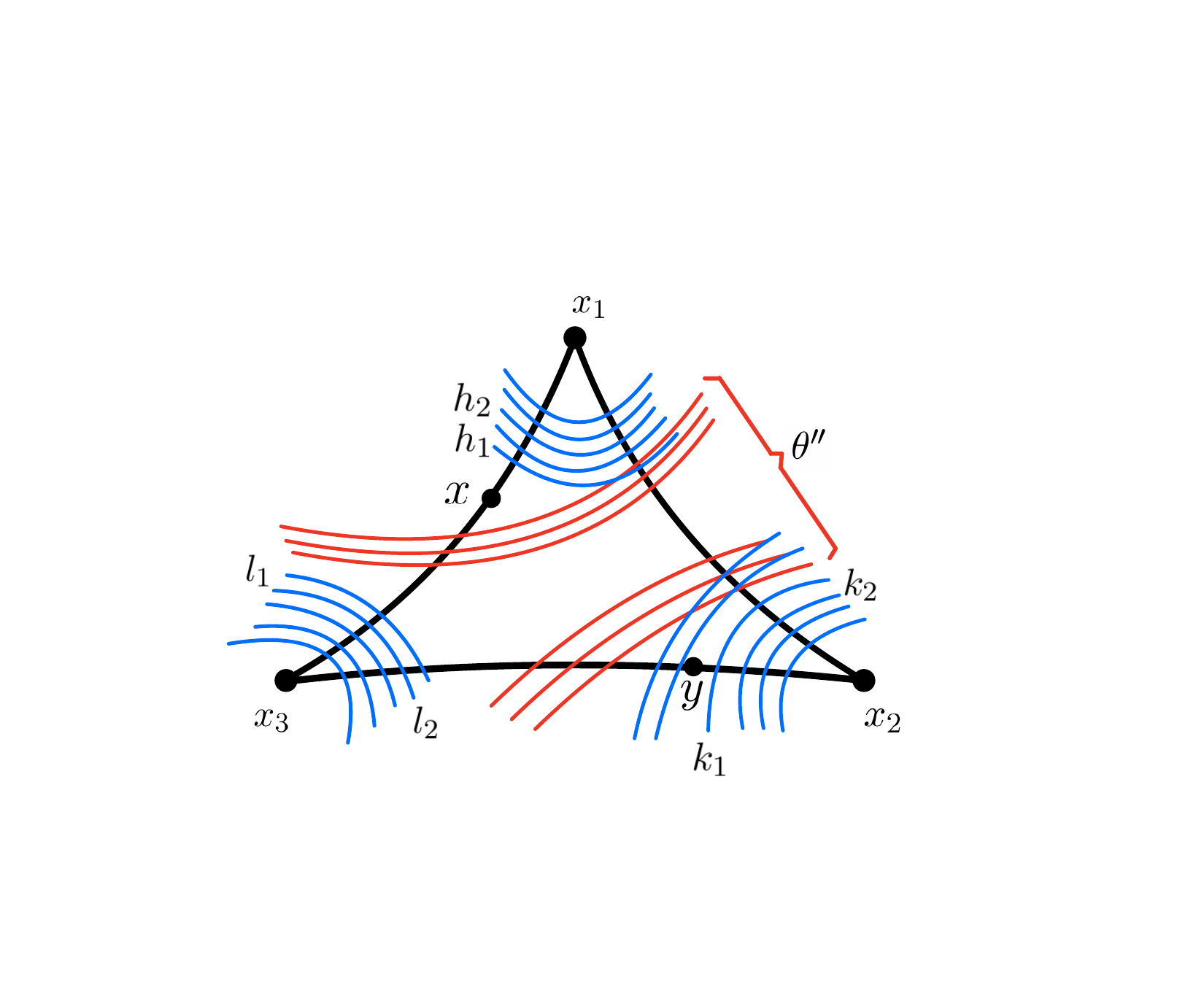}
    \caption*{$\theta''$ is disjoint from $\T'$. \\ { }\quad { }}
\end{subfigure}
\caption{The proof of Lemma~\ref{lem: triangular distance same as L distance}. The curtain combinatorics are the same if $x$ and $y$ lie on a common side of $T$.} \label{fig:triangle_same_as_L}
\label{fig:three graphs}
\end{figure}

\begin{proof}
Let $\theta$ be a maximal $L$--chain separating $x$ from $y$. Each $x_i$ lies in at most one curtain of $\theta$. Remove all such curtains. By Corollary~\ref{cor:double_crossing}, we can remove up to $2L$ curtains to ensure that for each element of the remaining subset $\theta'$ there is some $i$ such that it separates $x_i$ from $\{x_{i+1}, x_{i+2}\}$. Observe that there can be at most two values of $i$ appearing in this way: given three points $a,b,c$ and a chain disjoint from them, we cannot find three curtains of the chain where the first separates $a$ from $\{b,c\}$, the second $b$ from $\{c,a\}$, and the third $c$ from $\{a,b\}$. It is therefore possible to relabel the $x_i$ so that each element of $\theta'$ separates either $\{x,x_1\}$ from $\{y,x_2,x_3\}$, or $\{y,x_2\}$ from $\{x,x_1,x_3\}$. See Figure~\ref{fig:triangle_same_as_L}. In particular, every element of $\theta'$ separates $x_1$ from $x_2$.

Assume that $|x_1,x|_\T\geq2$, $|x_2,y|_\T\geq 2$, and the $x_3$--corner of $\mc{T}$ has length at least 2. Let $h_1$ and $h_2$ be the curtains of $\mc{T}$ separating $x_1$ from $x$ that are closest and second closest to $x$, respectively, and define analogously $k_1$ and $k_2$ for $y$ and $x_2$. Let $l_1$ and $l_2$ be the curtains in the $x_3$--corner that are furthest from $x_3$. Since every element of $\theta'$ separates $x_1$ from $x_2$, any element of it that intersects $l_2$ must intersect $l_1$, and there are therefore at most $L$ such elements of $\theta'$. Moreover, because every element of $\theta'$ separates $x$ from $y$, any element of it that intersects $h_2$ or $k_2$ must intersect $h_1$ r $k_1$, respectively, and so there are at most $2L$ such elements of $\theta'$. Let $\theta''$ be obtained from $\theta'$ by removing all curtains just described. 

We have $|\theta|-|\theta''|\le5L+3$. Let us write $c_1$ for the subchain of the $x_1$--corner of $\T$ that is separated from $x$ by $h_1$. Define the analogous subsets $c_2$ and $c_3$ of the $x_2$-- and $x_3$--corners of $\T$. By construction, every element of $\theta''$ is disjoint from the triangular $L$--chain
\begin{align*}
\T' \,&=\, c_1\cup c_2\cup c_3 \\
    &=\, \T \smallsetminus \big(\{h_1,k_1,l_1\} \cup 
        \{b\in\T\,:\, b\text{ separates } x \text{ from } y \text{ or intersects } \{x,y\}\}\big).
\end{align*}
Because $c_1\cup\theta''$ is a chain, Lemma~\ref{lem:gluing_disjoint_chains} tells us that $c_1\cup\theta''\smallsetminus\{h_2\}$ is an $L$--chain. Similarly, $c_2\cup\theta''\smallsetminus\{k_2\}$ is an $L$--chain. For $i\in\{1,2\}$, let $\theta''_i$ be the subchain of $\theta''$ consisting of those curtains that separate $x_i$ from $x_3$. We have $\theta''=\theta'_1\sqcup\theta''_2$. Because $c_3\cup\theta''_i$ is a chain, Lemma~\ref{lem:gluing_disjoint_chains} implies that $c_3\cup\theta''_i\smallsetminus\{l_2\}$ is an $L$--chain. This shows that $\T'\sqcup\theta''\smallsetminus\{h_2,k_2,l_2\}$ is a triangular $L$--chain for $\{x_1,x_2,x_3\}$.

Because $\T$ was chosen to be a maximal triangular $L$--chain for $\{x_1,x_2,x_3\}$, we have 
\begin{align*}
|\T| \,&\ge\, |\T'|+|\theta''|-3 \\
    &\ge\, |\T|-3-(|x,y|_\T+2) + (|\theta|-5L-3) -3,
\end{align*}
from which we find that $|\theta|\le|x,y|_\T+5L+11$, and hence $\dist_L(x,y)\le|x,y|_\T+5L+12$.

The remaining cases where either $|x_1,x|_\T<2$, $|x_2,y|_\T<2$, or the $x_3$--corner has length less than 2 all use the same argument as above, but without needing to use Lemma~\ref{lem:gluing_disjoint_chains}.
\end{proof}

As a consequence, CAT(0)-geodesic triangles are thin in $X_L$.

\begin{proposition} \label{prop:g3}
If $[x,y,z]$ is a CAT(0)-geodesic triangle, then, as subsets of $X_L$, the set $[x,y]$ is contained in the $5L+10$--neighbourhood of $[x,z]\cup[z,y]$.
\end{proposition}

\begin{proof}
Fix a maximal triangular $L$--chain $\T$ for $\{x,y,z\}$. Given $p\in[x,y]$, we can find $q\in [x,z]\cup[z,y]$ with $|p,q|_\T= 0$ as follows. If $p$ lies in a curtain $h\in\T$ then take $q$ in the nonempty set $h\cap \left( [x,z]\cup[z,y]\right)$. If $|p,x|_\T = 0$ or $|p,y|_\T =0$, take $q = x$ or $q=y$ respectively. Otherwise there are $h_1,h_2\in\T$ such that $p \in h_1^+ \cap h_2^-$ and  no element of $\T$ separates $h_1$ from $h_2$. As curtains are closed and disjoint, we can take $q\in \left(h_1^+ \cap h_2^-\right)\cap \left([x,z]\cup[z,y]\right)$.
\end{proof}

We are now ready to prove hyperbolicity of the spaces $X_L$. It is clear than any isometry of $X$ is also an isometry of $X_L$, and we take this opportunity to point out how the actions of $\isom X$ on the various $X_L$ relate to one another. Recall that if a group $G$ acts on two metric spaces $X$ and $Y$, then the action on $X$ is said to \emph{dominate} the action on $Y$ if there is a $G$-equivariant, coarsely Lipschitz map $X\to Y$.

\begin{theorem} \label{thm:XL_hyperbolic}
For each $L<\infty$, the space $X_L$ is hyperbolic. Moreover, $\isom X\leq \isom X_L$, and the action of $\isom X$ on $X_L$ dominates the one on $X_{L-1}$.
\end{theorem}

\begin{proof}
Given $x,y\in X_L$, let $\eta_{xy}$ be the unique CAT(0) geodesic in $X$ from $x$ to $y$. Remark~\ref{rmk:new_metric_bounded} shows that the $\eta_{xy}$ are coarsely connected. Conditions \ref{ass:guessing1} and \ref{ass:guessing2} of Proposition~\ref{prop:guessing_geodesics} are immediate from Proposition~\ref{prop:cat0_rough_geodesic} because CAT(0) geodesics are unique. \ref{ass:guessing3} is provided by Proposition~\ref{prop:g3}. Thus $X_L$ has thin quasigeodesic triangles. As it is roughly geodesic, Proposition~\ref{prop:Rips_coarsely_injective} tells us that it is hyperbolic.
\end{proof}

% \begin{proposition} \label{prop:coarsely_injective}
% The injective hull $E(X_L)$ is a geodesic hyperbolic space, and $X_L$ is a coarsely dense subspace.
% \end{proposition}

% \begin{proof}
% The space $X_L$ is a quasigeodesic hyperbolic space by Theorem~\ref{thm:XL_hyperbolic}, and it is weakly roughly geodesic by Lemma~\ref{lem:weakly_roughly_geodesic}. The result is given by Proposition~\ref{prop:Rips_coarsely_injective}.
% \end{proof}

%%%%%%%%%%%%%%%%%%%%%%%%%%%%%%%%%%%%%%%%
\subsection{An Ivanov-type theorem}\label{subsec:Ivanov}

We have seen that $\isom X<\isom X_L$ for every CAT(0) space $X$ and every integer $L$. Moreover, since $X$ and $X_L$ have the same underlying set, every isometry of $X_L$ induces a bijection of $X$. Our purpose here is to address what can be said about these bijections. 

Given $x\in X$ and $r\in\R$, write $B(x,r)$ for the closed ball of radius $r$ centred on $x$, and write $S(x,r)$ for the sphere of radius $r$ centred on $x$. We say that a collection $\Chain$ of bijections $X\to X$ \emph{preserves $r$--balls} if $gB(x,r)=B(gx,r)$ for all $x\in X$, all $r\ge0$, and all $g\in\Chain$. We use similar terminology for spheres. Recall that $X$ has the geodesic extension property if every geodesic segment is contained in a biinfinite geodesic.

\begin{proposition} \label{prop:preserving}
Let $X$ be a CAT(0) space and let $n,L\in\mathbf N$. The group $\isom X_L$ preserves $n$--balls of the CAT(0) space $(X,\dist)$. If $X$ has the geodesic extension property, then $\isom X_L$ also preserves $n$--spheres.
\end{proposition}

\begin{proof}
First note that for any $x\in X$ we have $B_X(x,1)=B_{X_L}(x,1)$. Hence $\isom X_L$ preserves 1--balls. Now suppose that $\isom X_L$ preserves $(n-1)$--balls, and let $z\in B(x,n)$. There is some point $y\in[x,z]$ such that $\dist(x,y)\le n-1$ and $\dist(y,z)\le 1$, so by assumption we have $\dist(gx,gy)\le n-1$ and $\dist(gy,gz)\le1$ for all $g\in\isom X_L$. This shows that $gB(x,n)\subset B(gx,n)$ for all $g\in\isom X_L$. But now we have 
\[
gB(x,n) \subset B(gx,n) = gg^{-1}B(gx,n) \subset gB(x,n),
\]
so $\isom X_L$ preserves $n$--balls.

Now suppose that $X$ has the geodesic extension property. Given $x,y\in X$ with $\dist(x,y)=n\ge1$, let $z\in X$ be such that $y\in[x,z]$ and $\dist(y,z)=1$. This makes $y$ the unique element of $B(x,n)\cap B(z,1)$, so the fact that $g\in\isom X_L$ preserves $k$--balls implies that $gy$ is the unique element of $B(gx,n)\cap B(gz,1)$. The fact that these balls meet in a single point must mean that $\dist(gx,gz)=n+1$, and so $\dist(gx,gy)=n$. That is, $gS(x,n)\subset S(gx,n)$. As in the case of balls, it follows that $\isom X_L$ preserves $n$--spheres.
\end{proof}

Some additional assumption is certainly needed for the elements of $\isom X_L$ to preserve spheres, for if $X$ is a CAT(0) space of diameter at most one then every two distinct points of $X_L$ have distance exactly one.

\begin{corollary} \label{cor:quasiivanov}
Let $X$ be a CAT(0) space. Every isometry $g\in\isom X_L$ is a (1,1)--quasiisometry of $X$. 
\end{corollary}

\begin{proof}
For any pair $x,y\in X$ there is a unique integer $n$ such that $\dist(x,y)\in(n,n+1]$. Since $g$ preserves $n$--balls and $(n+1)$--balls, we have that $\dist(gx,gy)\in(n,n+1]$. This shows that $|\dist(x,y)-\dist(gx,gy)|<1$.
\end{proof}

\begin{remark}
The ``additive 1'' in Corollary~\ref{cor:quasiivanov} comes from our choice for curtains to have width one. Making a smaller choice of width would lead to a smaller error, so one may wonder whether by taking some kind of limit with increasingly fine curtains could yield a hyperbolic model $Y$ space such that every isometry of the CAT(0) space $X$ is an isometry of $Y$. Unfortunately, it is not clear how to take this limit whilst preserving hyperbolicity. This is related to the procedure we will describe in Section~\ref{sec:universal}.
\end{remark}

Note that in quasiisometrically rigid CAT(0) spaces, such as symmetric spaces and buildings \cite{kleinerleeb:rigidity}, Corollary~\ref{cor:quasiivanov} means that every isometry of $X_L$ is at bounded distance from an isometry of $X$. It turns out that for symmetric spaces and buildings, every isometry of $X_L$ coincides with an isometry of $X$ (Corollary~\ref{cor:one-ended}). It is natural to wonder whether this is the general picture. However, Corollary~\ref{cor:quasiivanov} is optimal in general: consider the real line, as discussed in Section~\ref{isubsec:ivanov}.

On the other hand, Proposition~\ref{prop:preserving} provides a route to obtaining stronger results under additional assumptions that rule out the real line. Our approach relies on Andreev's contributions \cite{andreev:aleksandrov} to Aleksandrov's problem in CAT(0) spaces, which asks whether every self-bijection that preserves 1--spheres is an isometry.

\begin{definition}[Diagonal tube]
Let $(X,\dist)$ be a metric space. The \emph{diagonal tube} of $\dist$ is the set $V_{\dist}=\{(x,y)\in X\times X\,:\,\dist(x,y)\le1\}$. We say that a metric $\dist'$ \emph{realises} a subset $V\subset X\times X$ if $V=V_{\dist'}$.
\end{definition}

In the case that $(X,\dist)$ is a CAT(0) space with the geodesic extension property, let $V=V_{\dist}$. For any element $\phi$ of any $\isom X_L$, we can consider the pullback metric $\dist'(x,y)=\dist(\phi(x),\phi(y))$. We know from Proposition~\ref{prop:preserving} that $\dist'$ realises $V$. Therefore, in order to show that $\isom X=\isom X_L$, it is sufficient to show that $\dist$ is the unique metric realising $V$.

For the remainder of this section, geodesics will be understood to be biinfinite. Say that geodesics $a$ and $b$ are \emph{asymptotic} if they share an endpoint in the visual boundary $\partial X$ (the collection of all equivalence classes of geodesic rays, where two rays are equivalent if they are at finite Hausdorff-distance). Say that geodesics $a$ and $c$ are \emph{virtually asymptotic} if there is a sequence $a_0=a,a_1,\dots,a_n=c$ such that $a_i$ is asymptotic to $a_{i-1}$ for all $i$.

\begin{lemma}[{\cite[Lem.~2.4]{andreev:aleksandrov}}] \label{lem:virtually_asymptotic}
Let $X$ be a proper CAT(0) space with the geodesic extension property and suppose that $\dist'$ is a metric on $X$ realising $V$. Let $a$ and $c$ be virtually asymptotic geodesics. If $\dist'|_a=\dist|_a$, then $\dist'|_c=\dist|_c$.
\end{lemma}

This lemma provides a route to proving that $\dist$ is the unique metric realising $V$. Namely, one can try to show that every geodesic of $X$ is virtually asymptotic to some geodesic whose metric is uniquely determined by $V$.

Following the terminology of \cite{andreev:aleksandrov}, say that a geodesic is \emph{higher-rank} if it bounds a Euclidean strip, and \emph{strictly rank-one} if it is not virtually asymptotic to any higher-rank geodesic.

\begin{proposition}[{\cite[Cor.~3.2]{andreev:aleksandrov}}] \label{prop:virtually_higher_rank}
Let $X$ be a CAT(0) space, and suppose that $\dist'$ is a metric on $X$ realising $V$. If $a$ is a higher-rank geodesic, then $\dist'|_a=\dist|_a$.
\end{proposition}

Combined with Lemma~\ref{lem:virtually_asymptotic}, this shows that if $X$ is a proper CAT(0) space with the geodesic extension property and every geodesic of $X$ is virtually higher-rank, then $\dist$ is the unique metric realising $V$, so $\isom X=\isom X_L$. This applies in particular to universal covers of Salvetti complexes of non-free right-angled Artin groups. Note that it is also straightforward to show from Proposition~\ref{prop:preserving} that $\isom X=\isom X_L$ when $X$ is a tree that does not embed in $\R$.

\begin{proposition}[{\cite[Thm~4.7]{andreev:aleksandrov}}] \label{prop:one-ended}
Let $X$ be a proper, one-ended CAT(0) space, and suppose that $\dist'$ is a metric on $X$ realising $V$. If $a$ is a geodesic that is strictly rank-one, then there is a geodesic $b$ virtually asymptotic to $a$ such that $\dist'|_b=\dist|_b$.
\end{proposition}

Importantly for us, Proposition~\ref{prop:one-ended} does not assume the geodesic extension property. Combining Propositions~\ref{prop:virtually_higher_rank} and~\ref{prop:one-ended} with Lemma~\ref{lem:virtually_asymptotic} and Proposition~\ref{prop:preserving} gives the following.

\begin{corollary} \label{cor:one-ended}
If $X$ is a proper, one-ended CAT(0) space with the geodesic extension property, then $\isom X=\isom X_L$ for all $L$.
\end{corollary}

This covers both symmetric spaces and buildings. In fact, if $X$ is cobounded, then it necessarily has a \emph{rank-one isometry} (a semisimple isometry with an axis that doesn't bound a half flat) if it is not one-ended, so this covers all higher-rank examples. However, one can generalise this result in the presence of a geometric action. 
% This is our next goal. We use the following lemma, which appears as \cite[Lem.~2.3]{hagensusse:onhierarchical}, \cite[Prop.~7.2]{hruskaruane:connectedness}.
% 
% \begin{lemma}\label{lem:stab_acts_cocompact}
% Let $X$ be a metric space and let $G$ be a group acting cocompactly on $X$. If $P\subset X$ has the property that for any ball $B\subset X$ only finitely many translates of $P$ meet $B$, then the stabiliser $\Stab_G(P)$ acts cocompactly on $P$.
% \end{lemma}
% 
Recall that for a subspace $Y$ of a CAT(0) space $X$, the \emph{convex hull} of $Y$ is defined to be the intersection of all convex sets containing $Y$, which is easily verified to be a CAT(0) subspace. Equivalently, let $Y^0=Y$, and given $Y^i$, let $Y^{i+1}$ be the union of all geodesics joining points of $Y^i$. The convex hull of $Y$ is $\bigcup Y^i$.

\begin{theorem} \label{thm:ivanov}
Let $X$ be a CAT(0) space with the geodesic extension property, and suppose that a group $G$ acts properly cocompactly on $X$. If $G$ is not virtually free, then $\isom X=\isom X_L$.
\end{theorem}

\begin{proof}
If $G$ is one-ended, then this is Corollary~\ref{cor:one-ended}. Otherwise, observe that $G$ is finitely presented and hence is accessible by Dunwoody's theorem \cite{dunwoody:accessibility}. Hence, by Stallings' theorem \cite{stallings:ontorsionfree,stallings:group}, there is a nontrivial (finite) graph of groups decomposition of $G$ such that edge groups are finite and, as $G$ is not virtually free, some vertex group $G_v$ is one-ended. Let $f:G\to X$ be an orbit map $g\mapsto gx_0$ with quasiinverse $\bar f:X\to G$, and let $X_v$ be the CAT(0) convex hull in $X$ of $f(G_v)=G_v\cdot x_0$. Note that $X_v$ is a proper CAT(0) space, and the action of $G_v$ on $f(G_v)$ extends to an action on $X_v$.

\begin{claim*}
The action of $G_v$ on $X_v$ is cocompact.
\end{claim*}

\begin{claimproof}
% We want to verify the hypotheses of Lemma~\ref{lem:stab_acts_cocompact}. 
% Since the edge groups of the decomposition are finite, for any other vertex group $G_w$ there is a ball $B_1\subset G$ of uniformly finite radius and at uniformly bounded distance from $G_v$ such that any path in $G$ from a point of $\bar ff(G_w)$ to a point of $\bar ff(G_v)$ must pass through $B_1$. In particular, there is a ball $B_2\subset X$ uniformly close to $f(G_v)$ such that if $z\in f(G_w)\cap (f(G_v))^1$, then $z$ lies on a geodesic between two points of $B_2$. As balls in $X$ are convex, this shows that the intersection $(f(G_v))^1\cap f(G_w)$ is contained in the convex hull of $B_2$. By the construction of the convex hull, iterating shows that $X_v\cap f(G_w)$  is uniformly bounded. As $f(G)$ coarsely coincides with $X$, we obtain that the Hausdorff distance between $X_v$ and $f(G_v)$ is uniformly bounded. Since $X_v$ is proper, the action of $G_v$ on $X_v$ is cocompact. 
Our goal is to show that the Hausdorff distance between $X_v$ and $f(G_v)$ is uniformly bounded. Since the edge groups of the decomposition are finite, for any other vertex group $G_w$ there is a ball $B_1\subset G$ of uniformly bounded radius and at uniformly bounded distance from $G_v$ such that any path in $G$ from a point of $\bar ff(G_w)$ to a point of $\bar ff(G_v)$ must pass through $B_1$. In particular, there is a ball $B_2\subset X$ (essentially any ball large enough to contain $f(B_1)$) that is uniformly close to $f(G_v)$ and such that if $z \in f(G_w)$ is a point on a geodesic with endpoints on $f(G_v)$ then $z\in B_2$. By the iterative construction of the convex hull, we can inductively prove that if $z \in f(G_w)$ is a point on a geodesic with endpoints on $f(G_v)^n$, then $z$ belongs to $B_2$. Thus, $X_v\cap f(G_w)\subset B_2$. As $f(G)$ coarsely coincides with $X$, we obtain that the Hausdorff distance between $X_v$ and $f(G_v)$ is uniformly bounded. Since $X_v$ is proper, the action of $G_v$ on $X_v$ is cocompact. 
% Thus $G_v$ is the unique vertex group having unbound ed intersection with $\bar f(X_v)$. This implies that for any ball $B_3\subset G$, only finitely many translates of $\bar f(X_v)$ can meet $B_3$. Hence, for any ball $B_4\subset X$, only finitely many translates of $X_v$ can meet $B_4$. According to Lemma~\ref{lem:stab_acts_cocompact}, this means that $\Stab_G(X_v)$ acts cocompactly on $X_v$. Because $\bar f(X_v)$ contains no nontrivial translates of $G_v$, we deduce that $G_v$ acts cocompactly on $X_v$.
\end{claimproof}

In particular, $f|_{G_v}$ is a quasiisometry from $G_v$ to $X_v$, so $X_v$ is one-ended. Since $X_v$ is proper and unbounded, one can use the Arzel\`a--Ascoli theorem to find a biinfinite geodesic $a\subset X_v$.
% Moreover, the action of $G_v$ on $X_v$ is semisimple \cite[Prop.~II.6.10]{bridsonhaefliger:metric} and $G_v$ is not torsion \cite[Thm~11]{swenson:cut}. Hence $X_v$ contains a geodesic, namely an axis $a$ of a hyperbolic element of $G_v$.

Now, suppose that $\xi$ and $\zeta$ are points of the visual boundary of $X$ with the property that there is a compact set $B$ such that any path $\alpha$ with $(\alpha(-n))\to\xi$ and $(\alpha(n))\to\zeta$ must pass through $B$. It is a simple consequence of the Arzel\`a--Ascoli theorem and convexity of the metric that there is a geodesic $\beta$ with $\beta(-\infty)=\xi$ and $\beta(\infty)=\zeta$. By repeatedly applying this fact, it can be seen that any geodesic in $X$ is virtually asymptotic to $a$. 

According to Propositions~\ref{prop:virtually_higher_rank} and~\ref{prop:one-ended}, for any metric $\dist'$ realising $V$ there is a geodesic $b\subset X_v$ that is virtually asymptotic to $a$ and has $\dist|_b=\dist'_b$. Because $X$ has the geodesic extension property, Lemma~\ref{lem:virtually_asymptotic} shows that $\dist'=\dist$. Hence $\dist$ is the unique metric realising $V$, so every element of $\isom X_L$ is an isometry of $X$.
\end{proof}

\section{Contracting geodesics and stability}

Let $X$ be a CAT(0) space. In this section, we consider geodesics in $X$ and groups of isometries of $X$ that can be considered ``negatively curved''. 

\begin{definition}[Contracting]
We say that a geodesic $\gamma$ is \emph{$D$--contracting} if for any ball $B$ disjoint from $\gamma$ we have $\diam\pi_\gamma(B)\leq D$. A hyperbolic isometry of $X$ is contracting if one, hence all, of its axes are contracting. 
\end{definition}

To simplify the constants in certain places, we shall always assume that $D\ge1$ unless otherwise specified.

In fact, the definition of a contracting geodesic makes sense in any metric space: given a geodesic $\gamma$ and a point $x$, the set of points in $\gamma$ realising $\dist(x,\gamma)$ is nonempty, because one only needs to consider a compact subinterval of $\gamma$ by the triangle inequality. One then finds that the closest-point projection to a contracting geodesic in a metric space is coarsely unique.
Charney--Sultan showed that, in CAT(0) spaces, a geodesic being contracting is equivalent to its being \emph{Morse} \cite[Thm~2.14]{charneysultan:contracting}. In the setting of proper CAT(0) spaces, Bestvina--Fujiwara showed an isometry is \emph{rank-one} if and only if its axes are contracting \cite[Thm~5.4]{bestvinafujiwara:characterization}. 

The first result of this section is the following, which sums up Lemma~\ref{lem:contracting_crosses} and Proposition~\ref{prop:cross_contracting}; it is worth noting that the proof does not directly use $X_L$ or the hyperbolicity thereof. We say that a geodesic $\alpha$ meets a chain $\{h_i\}$ of hyperplanes \emph{$r$--frequently} if there are $\alpha(t_i)\in h_i$ such that $t_{i+1}-t_i\le r$ for all $i$.

\begin{theorem} \label{thm:contracting_curtain_characterisation}
Let $X$ be a CAT(0) space. If $\alpha\subset X$ is a $D$--contracting geodesic, then there is a $(10D+3)$--chain of curtains met $8D$--frequently by $\alpha$. Conversely, if a geodesic $\beta$ meets an $L$--chain of curtains $T$--frequently, where $T\ge1$, then $\beta$ is $16T(L+4)$--contracting.
\end{theorem}

Recall the following version of bounded geodesic image for CAT(0) spaces.

\begin{lemma}[{\cite[Lem.~4.5]{charneysultan:contracting}}] \label{lem:pass_close_to_contracting}
Let $\alpha$ be a $D$--contracting geodesic in a geodesic space $X$. If $x,y \in X$ have $\dist(\pi_\alpha(x), \pi_\alpha(y))\geq 4D$, then any geodesic $\gamma$ from $x$ to $y$ satisfies $\pi_\alpha(\gamma) \subseteq \cal N_{5D}(\gamma)$.
\end{lemma}

We now prove the forward direction of Theorem~\ref{thm:contracting_curtain_characterisation}. Recall that $h_{\alpha,r}$ denotes the curtain dual to the geodesic $\alpha$ centred around the point $\alpha(r)$.

\begin{lemma} \label{lem:contracting_crosses}
Let $\alpha:I\to X$ be a $D$--contracting geodesic. The chain $\{h_i=h_{\alpha,8Di}\,:\,8Di\in I\}$ is a $(10D+3)$--chain dual to $\alpha$ such that $\alpha(8Di)\in h_i$.
\end{lemma}

\begin{proof}
Let $k$ be a curtain meeting both $h_i$ and $h_{i+1}$, and let $\beta$ be its pole. See Figure~\ref{fig:dual}. Because $\diam\pi_\alpha\beta\le\diam\beta=1$, there exists $t\in[8Di+4D-1,8Di+4D+1]$ such that $\alpha(t)\not\in\pi_\alpha\beta$. Let $x\in k\cap h_i$ and $y\in k\cap h_{i+1}$. There are points $x',y'\in\beta$ such that $[x,x']\subset k$ and $[y,y']\subset k$. Since $\pi_\alpha$ is continuous, there is some $z\in[x,x']\cup[y,y']$ such that $\pi_\alpha(z)=\alpha(t)$. The cases are similar, so let us assume that $z\in[x,x']$. Lemma~\ref{lem:pass_close_to_contracting} tells us that $\alpha(t)=\pi_\alpha(z)\in\pi_\alpha[x,x']$ is $5D$--close to $[x,x']\subset k$.

Thus $\alpha(8Di+4D)$ is $(5D+1)$--close to every curtain $k$ meeting both $h_i$ and $h_{i+1}$. This shows that every chain of curtains meeting $h_i$ and $h_{i+1}$ has cardinality at most $10D+3$.
\end{proof}

\begin{figure}[ht]
\vspace{-1cm}\includegraphics[width=11cm]{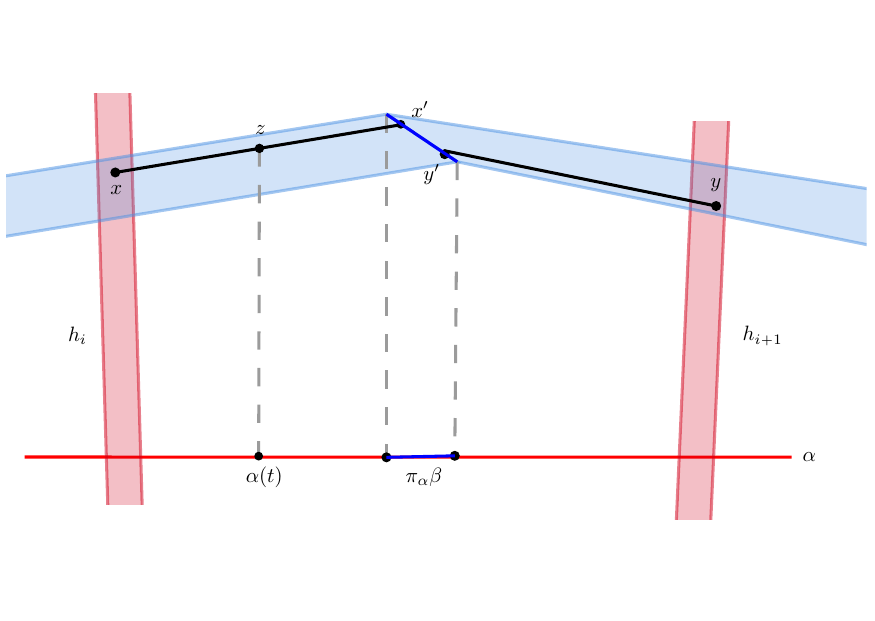}\centering\vspace{-1cm}
\caption{The proof of Lemma~\ref{lem:contracting_crosses}.} \label{fig:dual}
\end{figure}

For the reverse direction of Theorem~\ref{thm:contracting_curtain_characterisation}, we are given an $L$--chain of curtains that meet $\alpha$. We begin by showing that we can assume that they are actually dual to $\alpha$. This is similar to Lemma~\ref{lem:updating_curtains}, but includes the extra information of how frequent the crossing is.

\begin{lemma} \label{lem:crossers_are_dual}
Let $\alpha$ be a geodesic meeting an $L$--chain of curtains $c=\{h_i\}$ at points $\alpha(t_i)$ with $t_{i+1}-t_i\in[1,T]$. There is an $L$--chain $\{k_i=h_{\alpha,s_i}\}$ of curtains dual to $\alpha$ such that $s_{i+1}-s_i\le4T(L+3)$ for all $i$.
\end{lemma}

\begin{proof}
Let $m=2L+6$. If $\alpha$ has length at most $2T(L+4)$, then let $\{k_i\}$ consist of any one curtain dual to $\alpha$. Otherwise, there is some $i$ such that both $h_i$ and $h_{i+m-1}$ exist. For each such $i$, let $c'_i$ be a chain of curtains dual to $\alpha$ realising $|c'_i|+1=\dist_\infty(\alpha(t_{i+2}),\alpha(t_{i+m-3}))\ge m-5=2L+1$. Since $h_i$ and $h_{i+1}$ are $L$--separated, at most $L$ elements of $c'_i$ meet $h_i$, and similarly at most $L$ elements of $c'_i$ meet $h_{i+m-1}$. Let $g_i$ be any other element of $c'_i$, so that $\{h_i,g_i,h_{i+m-1}\}$ is a chain. 

Define $k_i=g_{mi}$ for every $i$ such that $h_{mi}$ and $h_{mi+m-1}$ exist, and consider the chain $c''=\{k_i\}$. For each $i$, let $s_i$ be the real number such that $k_i=h_{\alpha,s_i}$. We have $s_{i+1}-s_i\le \dist(h_{im},h_{im+2m})\le2Tm$, so it remains to show that $c''$ is an $L$--chain. Since $h_{i+m-1}$ and $h_{i+m}$ separate $k_i$ from $k_{i+1}$, any curtain meeting $k_i$ and $k_{i+1}$ must meet $h_{i+m-1}$ and $h_{i+m}$. Because $c$ is an $L$--chain, this implies that $c''$ is an $L$--chain.
\end{proof}

Lemma~\ref{lem:crossers_are_dual} allows us to apply Lemma~\ref{lem:bounding_distance_with_dual_curtains} to prove the reverse direction of Theorem~\ref{thm:contracting_curtain_characterisation}.

\begin{proposition} \label{prop:cross_contracting}
If $\alpha$ is a geodesic meeting an infinite $L$--chain of curtains $\{h_i\}$ at points $\alpha(t_i)$ such that $t_{i+1}-t_i\le T$, then $\alpha$ is $(16T(L+3)+3)$--contracting.
\end{proposition} 

\begin{proof}
According to Lemma~\ref{lem:crossers_are_dual}, after replacing $L$ by $4T(L+3)$, we may assume that $h_i=h_{\alpha,t_i}$. Suppose that $x,y\in X$ have $\dist(\pi_\alpha x,\pi_\alpha y)>2L+2$. Then (perhaps after relabelling) there exists $i$ such that $x\in h_{i-1}^-$ and $y\in h_{i+1}^+$. Let $z\in[x,y]\cap h_i$. Lemma~\ref{lem:bounding_distance_with_dual_curtains} tells us that $\dist(z,\pi_\alpha z)<2L+1$. As $\pi_\alpha z\in h_i$, we have
\[
\dist(x,y) \:\ge\: \dist(x,z) \:\ge\: \dist(x,\pi_\alpha z)-2L-1 \:\ge\: \dist(x,\pi_\alpha x)-2L-1.
\]
In particular, any point $y$ with $\dist(x,y)\le\dist(x,\pi_\alpha x)-2L-1$ has $\dist(\pi_\alpha x,\pi_\alpha y)\le2+2L$.

Let $B$ be a ball centred on $x$ that is disjoint from $\alpha$, and let $w\in B$. Let $y$ be the point of $[x,w]$ with $\dist(x,y)=\min\{\dist(x,w),\dist(x,\pi_\alpha x)-2L-1\}$. As $\dist(y,w)\le2L-1$, we have 
\[
\dist(\pi_\alpha x,\pi_\alpha w) \:\le\: \dist(\pi_\alpha x,\pi_\alpha y)+\dist(\pi_\alpha y,\pi_\alpha w)
\:\le\: (2L+2)+(2L+1),
\]
because $\pi_\alpha$ is $1$--Lipschitz.
\end{proof}

\begin{corollary}\label{cor:hyp_iff_XL_and_X_QI}
A CAT(0) space $X$ is hyperbolic if and only if the identity map induces a quasi-isometry between $X$ and $X_L$ for some $L.$
\end{corollary}

\begin{proof}
All geodesics in a hyperbolic space are uniformly contracting, so the forward direction follows from Lemma~\ref{lem:contracting_crosses}. The reverse is Theorem~\ref{thm:XL_hyperbolic}.
\end{proof}

When considering axes of isometries, the conclusion of Theorem~\ref{thm:contracting_curtain_characterisation} can be strengthened. 

\begin{definition}[Skewer]
An isometry $g$ is said to skewer two curtains $h_1,h_2$ if, $g^mh_1$, $h_2$, and $h_1$ are pairwise disjoint and, perhaps after relabelling, $g^mh_1^+ \subsetneq h_2^+ \subsetneq h_1^+$ for some $m \in \mathbf{N}$. 
\end{definition}

As discussed in the introduction, the equivalence between contraction and skewering of separated curtains mirrors a characterisation in the cubical setting \cite{capracesageev:rank,genevois:contracting}.

\begin{theorem} \label{thm:rank_one_characterisation}
Let $g$ be a semisimple isometry of $X$. The following are equivalent.
\begin{enumerate}
\item   $g$ is contracting. \label{cond:rank_one}
\item   $g$ acts loxodromically on some $X_L$. \label{cond:loxodromic}
\item   $g$ is hyperbolic and there exist $L,n$ such that $\dist_L(x,g^nx)>L+4$ for some $x$ lying in an axis $\alpha$ of $g$. \label{cond:big}
\item   $g$ skewers a pair of separated curtains. \label{cond:skewer}
\end{enumerate}
\end{theorem}

\begin{proof} 
Items \eqref{cond:loxodromic} and \eqref{cond:skewer} follow immediately from \eqref{cond:rank_one} by Theorem~\ref{thm:contracting_curtain_characterisation}. If $g$ acts loxodromically on $X_L$, then $g$ is necessarily a hyperbolic isometry of $X$, because $\dist_L\le1+\dist$ (Remark~\ref{rmk:new_metric_bounded}), and \eqref{cond:big} is immediate.

Assuming \eqref{cond:big}, let $c$ be an $L$--chain realising $|c|+1=\dist_L(x,g^nx)>L+4$. As $x$, $g^nx$, and $g^{2n}x$ lie on a geodesic,  Lemma~\ref{lem:gluing_chains} lets us find a nonempty subchain $c'$ such that $\bigsqcup_{n\in\Z}g^{jn}c'$ is an $L$--chain. Let $h_1$ be the minimal element of $c'$, and let $h_2$ be the maximal element. Let $s_1$ satisfy $\alpha(s_1)\in h_1$ and let $s_2$ satisfy $\alpha(s_2)\in gh_2$. Applying Theorem~\ref{thm:contracting_curtain_characterisation}, we see that $\alpha$ is contracting, with constant depending only on $L$ and $|s_2-s_1|$. Hence \eqref{cond:big} implies \eqref{cond:rank_one}.

Finally, we show that \eqref{cond:skewer} implies \eqref{cond:rank_one}. Let $h_1,h_2$ be curtains such that $g^mh_1^+ \subsetneq h_2^+ \subsetneq h_1^+$. Because any curtain crossing $h_1$ and $g^mh_1$ crosses $h_2$, the curtains $g^{km}h_1$ and $g^{(k+1)m}h_1$ are $L$--separated for every $k\in\Z$. That is, $\{g^{km}h_1:k\in\Z\}$ is an $L$--chain. In particular, we have a nested sequence of halfspaces
\[
\cdots g^{2m}h_1^+  \subsetneq g^mh_1^+ \subsetneq h_1^+  \subsetneq g^{-m}h_1^+  \subseteq g^{-2m}h_1^+ \cdots 
\]
If $p \in h_1$, then $\dist(p,g^{nm}p) \geq n$ by Remark~\ref{rem:thick_curtains}, so $\underset{k \rightarrow \infty}{\lim} \frac{\dist(p, g^{km}p)}{k} \geq 1$. Hence $g$ is hyperbolic, because it is semisimple. 

Any axis $\alpha$ of $g$ necessarily meets $h_1$. Indeed, let $x\in\alpha$. If $x\in h_1$ then we are done. Otherwise, there is some $n$ such that $h_1$ separates $x$ from $g^{nm}x$, and we apply Lemma~\ref{lem:curtains_separate}. Hence $\alpha=g^{km}\alpha$ meets every $g^{km}h_1$. Theorem~\ref{thm:contracting_curtain_characterisation} tells us that $\alpha$ is contracting.
\end{proof}

In particular, this gives a simple characterisation of CAT(0) spaces admitting a contracting isometry in terms of curtains and the spaces $X_L$.

\begin{lemma} \label{lem:unbounded_rank_one}
A group $G$ acting coboundedly by semisimple isometries on a CAT(0) space $X$ has a contracting element if and only if some $X_L$ is unbounded. 
\end{lemma}

\begin{proof}
If $G$ has a contracting element, then Theorem~\ref{thm:rank_one_characterisation} implies that some $X_L$ is unbounded. If $X_L$ is unbounded, then (using Proposition~\ref{prop:Rips_coarsely_injective}) Gromov's classification \cite{gromov:hyperbolic,capracecornuliermonodtessera:amenable} implies that $G$ contains an isometry acting loxodromically on $X_L$, and that isometry is contracting by Theorem~\ref{thm:rank_one_characterisation}.
\end{proof}

We finish this section by characterising stable subgroups of CAT(0) groups in terms of their orbits on the $X_L$. As discussed in the introduction, this is the same as the situation in mapping class groups \cite{durhamtaylor:convex,farbmosher:convex,kentleininger:shadows}, and, more generally, hierarchically hyperbolic groups \cite{abbottbehrstockdurham:largest}.

\begin{definition}[Stable]
Let $Y$ be a subset of a CAT(0) space $X$. We say that $Y$ is \emph{stable} if there exist $\mu, D \geq 0$ such that every geodesic between points of $Y$ is $D$--contracting and stays $\mu$--close to $Y$. A subgroup of a group acting properly coboundedly on $X$ is \emph{stable} if it has a stable orbit.
\end{definition}

We remark that the above definition is specialised to the case of CAT(0) spaces. In general, it is necessary to consider \emph{Morse} geodesics  instead of contracting ones, but in CAT(0) spaces the two notions are equivalent \cite[Thm~2.9]{charneysultan:contracting}.

\begin{proposition}\label{prop: stable versus models}
Let $G$ be a group acting properly coboundedly on a CAT(0) space $X$. A subgroup $H\leq G$ is stable if and only if it is finitely generated and there is some $L$ such that the orbit maps $H\to X_L$ are quasiisometric embeddings.
\end{proposition}

\begin{proof}
Assume that an orbit $H \cdot x$ is stable in $X$, and let $g,h\in H$. By assumption, there is a contracting geodesic $\alpha$ connecting $gx$ and $hx$. By Theorem~\ref{thm:contracting_curtain_characterisation}, $\alpha$ meets an $L$--chain of curtains $L$--frequently, where $L$ is determined by the contracting constant. Thus $\alpha$ uniformly quasiisometrically embeds in $X_L$ by definition, which gives a coarse-linear equivalence between $\dist_L(gx,hx)$ and $\dist(gx,hx)$. Since $H$ is stable, it is both finitely generated and undistorted in $G$ \cite{durhamtaylor:convex,tran:onstrongly}. Thus $\dist(hx,kx)$ and $\dist_H(h,k)$ are comparable, yielding the forward direction.

For the converse, suppose that $H\cdot x$ is quasiisometrically embedded in $X_L$. Since $H$ is finitely generated, its orbit maps on $X$ are coarsely Lipschitz, and because the remetrisation $X\to X_L$ is coarsely Lipschitz, this means that $H\cdot x$ is also quasiisometrically embedded in $X$. In particular, for any $g,h\in H$, the CAT(0) geodesic $[gx,hx]$ uniformly quasiisometrically embeds in $X_L$, and so is contracting by Theorem~\ref{thm:contracting_curtain_characterisation}. The contracting property implies that the quasiisometric image of any $H$--geodesic from $g$ to $h$ is uniformly Hausdorff-close to $[gx,hx]$. Hence $[gx,hx]$ stays uniformly close to $H\cdot x$.
\end{proof}

Recall that for a CAT(0) space $X$, the \emph{Morse boundary}, denoted $\partial^{\star}X,$ is defined to be the collection of all contracting geodesic rays starting at a fixed point $\go$. The topology on the Morse boundary is given in Definition~\ref{def:Morse_boundary}. The previous result shows that for any finite family of rank-one elements, there exists some $X_L$ where they all act loxodromically. Using the Morse boundary, we can upgrade that to families satisfying some topological condition.

\begin{lemma}
Let $X$ be a proper CAT(0) space and let $\cal{H} \subset\isom X$ be a collection of rank-one elements whose limit set is contained in a compact subspace of the Morse boundary $\partial^{\star}X$. There exists an $L$ such that each element $g \in \cal{H}$ acts loxodromically on $X_L$. \end{lemma}

\begin{proof} 
If a subspace $Y$ of $\partial^{\star} X$ is compact, then there exists some $D$ such that the geodesic representative for each element in $Y$ is $D$-contracting \cite[Lem.~3.3]{murray:topology}. It follows that if $\alpha_g:(-\infty, \infty) \to X$ denotes the axis of $g\in\cal H$, then the unique geodesic rays $[\go, \alpha_g^{-\infty}], [\go, \alpha_g^{\infty}]$ are $D$-contracting. Hence $\alpha_g$ is $D'$-contracting where $D'$ depends only on $D$ (\cite[Lem.~2.4]{charneycordesmurray:quasimobius} and \cite[Thm~2.14]{charneysultan:contracting}). Thus, since the collection of lines $\{\alpha_g \,:\, g \in \cal{H}\}$ are all $D'$-contracting, Theorem~\ref{thm:contracting_curtain_characterisation} implies that each such $g$ acts loxodromically on $X_L$, where $L=10D'+3$.
\end{proof}

%%%%%%%%%%%%%%%%%%%%%%%%%%%%%%%%%%%%%%%%%%%%%%%%%%%%%%%%%%%%%%%%%%%%%%
\section{Weak properness of actions} \label{sec:diameter} 

In the previous section, we explored properties of isometries and groups of isometries with respect to their actions on the spaces $X_L$. We will now consider how the actions on $X_L$ of various elements and subgroups interact with each other, and prove that they satisfy various notions of properness. 

\begin{definition}[WPD, non-uniform acylindricity]
Let $G$ be a group acting on a metric space $X$. An infinite order element $g \in G$ is called \emph{WPD} (weakly properly discontinuous) if it has a quasi-isometrically embedded orbit and for each $\eps>0$ and each $x\in X$, there exists $m>0$ such that
\[
|\{h \in G \::\: \dist(x,hx),\, \dist(g^mx,hg^mx)<\eps\}| < \infty.	
\]
The action is said to be \emph{non-uniformly acylindrical} if for each $\eps>0$ there exists $R$ such that for any $x,y\in X$ with $\dist(x,y)\ge R$, only finitely many $g\in G$ have $\max\{\dist(x,gx),\dist(y,gy)\}<\eps$. 
\end{definition}

Observe that every loxodromic isometry in a non-uniformly acylindrical action is WPD. Moreover, note that an action on a bounded metric space is always non-uniformly acylindrical.

We need the following lemma, which is similar to Lemma~\ref{lem:bounding_distance_with_dual_curtains}, but does not require the curtains to be dual to a single geodesic.

\begin{lemma} \label{lem:bottleneck}
Suppose that $\alpha$ and $\alpha'$ are geodesics that cross three pairwise $L$--separated curtains $h_1,h_2,h_3$, and let $x_i\in \alpha\cap h_i$, $y_i\in\alpha'\cap h_i$. If $\dist(x_1,x_2),\dist(x_2,x_3)\le T$, then $\dist(x_2,y_2) \leq 2L+\lceil T\rceil$.
\end{lemma}

\begin{proof}
Let $c$ be chain dual to $\beta=[x_2,y_2]$ that realises $\dist_\infty(x_2,y_2)=1+|c|$. In view of Lemma~\ref{lem:chain_distance}, it suffices to show that $|c|\le2L+\lceil T\rceil-1$.

Because $x_2,y_2 \in h_2$, every element of $c$ must intersect $h_2$ by Lemma~\ref{lem:curtains_separate}. Thus $L$--separation tells us that at most $2L$ elements of $c$ can intersect $h_1 \cup h_3$. Moreover, by Lemma~\ref{lem:controlling_curtains}, no curtain in $c$ can intersect both $[x_1,x_2]$ and $[x_2,x_3]$, and similarly for the $y_i$. Hence, perhaps after relabelling, all but at most $2L$ elements of $c$ must cross both $[x_1,x_2]$ and $[y_2,y_3]$. But $\dist(x_1,x_2)=T$, so $|c|\le2L+\lceil T\rceil-1$.
\end{proof}

We now have all the ingredients to prove that the action on $X_L$ is non-uniformly acylindrical.

\begin{proposition} \label{prop:nonuniform_acylindricity}
Any group $G$ acting properly on a CAT(0) space $X$ acts non-uniformly acylindrically on every $X_L$. In particular, if $g\in G$ is contracting, then $g$ is WPD on $X_L$ for every $L$ for which it is loxodromic.
\end{proposition}

\begin{proof}
If $\diam X_L<\infty$ then there is nothing to prove. Otherwise, given $\eps>0$, let $\eps'=\lceil\eps\rceil$ and let $R=4+2\eps'$. Suppose that $x,y\in X$ have $\dist_L(x,y)\ge R$, and let $b=[x,y]$. There is an $L$--chain $\{k_1,\dots,k_{\eps'},h_1,h_2,h_3,k'_1,\dots,k'_{\eps'}\}$ separating $x$ from $y$. Let $x_i\in b\cap h_i$, and let $B$ be the ball in $X$ with centre $x_2$ and radius $2L+\dist(x,y)+1$. Note that $b\subset B$.

For any $g\in G$ with $\dist_L(x,gx),\dist_L(y,gy)<\eps$, the curtains $h_1$, $h_2$, and $h_3$ all separate $gx$ from $gy$. From Lemma~\ref{lem:bottleneck}, we deduce that $\dist(x_2,gb)\le2L+\max\{\dist(x_1,x_2),\dist(x_2,x_3)\}+1$. In particular, $gb\subset gB$ meets $B$. By properness of the action of $G$ on $X$, there are only finitely many such $g$.
\end{proof}

\begin{remark}
In fact, when considering WPD elements, we believe that one could weaken the assumptions of Proposition~\ref{prop:nonuniform_acylindricity} with more work. Namely, if one drops the assumption that the action of $G$ on $X$ is proper, then we expect that any WPD contracting isometry of $X$ is WPD on $X_L$. This would generalise the situation of \cite[Prop.~6.61]{genevois:hyperbolicities}.
\end{remark}

Since the existence of a loxodromic WPD element for an action on a hyperbolic space is equivalent to acylindrical hyperbolicity \cite{dahmaniguirardelosin:hyperbolically,osin:acylindrically}, we obtain the following. This simplifies the proof from \cite{sisto:contracting}.

\begin{corollary}
If $G$ acts properly on a proper CAT(0) space and has a rank-one element, then $G$ is acylindrically hyperbolic or virtually cyclic.
\end{corollary}

When considering subgroups instead of single elements, that is to say subgroups that are not cyclic, there is an analogous notion of WPD action introduced by Abbott--Manning \cite{abbottmanning:acylindrically}. 

\begin{definition}[A/QI triple] 
Let $G$ be a group acting on a hyperbolic space $(X,\dist)$ and let $Y \subseteq X$. The action of $G$ on $X$ is said to be \emph{acylindrical along $Y$} if for each $\epsilon>0$ there exist $R=R(\epsilon,Y)$ and $M=M(\epsilon, Y)$ such that for any $y_1,y_2 \in Y$ with $\dist(y_1,y_2) > R$ we have
$$|\{g \in G \,:\, \max\{\dist(y_1,gy_1), \dist(y_2,gy_2)\}<\epsilon\}|<M.$$
If $H<G$ is finitely generated and $G$ acts acylindrically along an $X$--orbit that is a quasiisometric embedding, then the triple $(G,X,H)$ is said to be an \emph{A/QI triple.}
\end{definition}

With the next two lemmas, we show that the notions of stability and A/QI triples agree in CAT(0) groups.

\begin{lemma} \label{lem:synchronous_fellow_travelling:contracting}
Let $X$ be a CAT(0) space, let $D>1$, and let $\eps>0$. For each $D$--contracting geodesic $\alpha=[a,b]$ with $\dist(a,b)>16D\lceil\eps\rceil+9$, if $g\in\isom X$ has $\dist_{10D+3}(a,ga),\dist_{10D+3}(b,gb)\le\eps$, then $\dist(\alpha(t),g\alpha(t))<75D\lceil\eps\rceil$ for all $t$.
\end{lemma}

\begin{proof}
According to Lemma~\ref{lem:contracting_crosses}, there is a $(10D+3)$--chain $c=\{h_1,\dots,h_n\}$ dual to $\alpha$, with $n\ge2\lceil\eps\rceil+1$, such that $p_i=\alpha(8Di)\in h_i$. Because $\dist_{10D+3}(a,ga)<\eps$, at most $\lceil\eps\rceil-1$ elements of $c$ can separate $a$ from $ga$. By considering $b$, $gb$, and $gc$ similarly, we see that if $m$ lies in the set $\{\lceil\eps\rceil,\dots,n-\lceil\eps\rceil+1\}$, which has cardinality at least three, then $h_m$ and $gh_m$ both separate $\{a,ga\}$ from $\{b,gb\}$.

Lemma~\ref{lem:bounding_distance_with_dual_curtains} now shows that if $\eps+1\le i\le n-\eps$, then there is some $q_i\in h_i\cap g\alpha$ with $\dist(p_i,q_i)\le20D+7$. Writing $p_0=a$, $p_{n+1}=b$, there is some $j$ such that $q_i\in[gp_j,gp_{j+1}]$, and $j$ can differ from $i$ by at most $\lceil\eps\rceil$. Hence 
\[
\dist(p_i,gp_i) \,\le\, \dist(p_i,q_i)+\dist(q_i,gp_j)+\dist(gp_j,gp_i) \,\le\, 20D+7 + 8D + 8D\lceil\eps\rceil.
\]
For each $t$ there is some $i\in[\eps+1,n-\eps]$ such that $\dist(\alpha(t),p_i)\le8D(\lceil\eps\rceil+1)$, and from this we see that $\dist(\alpha(t),g\alpha(t))<(8D(2+\eps))+(28D+7+8D\lceil\eps\rceil)+(8D(2+\eps))$, which concludes the proof.
\end{proof}

We say that a group $G$ acts \emph{uniformly properly} on a space $X$ if for any radius $r$ there exists $N$ such that for any ball $B$ of radius $r$ we have \[\vert\{g\in G \,:\,B_r \cap gB_r\}\vert \leq N.\]
Note that if a group action is proper and cobounded, then it is uniformly proper. An immediate consequence of Lemma~\ref{lem:synchronous_fellow_travelling:contracting} is the following.

\begin{lemma} \label{lem:Uniform_acylindricity_bottom}
Suppose that a group $G$ acts uniformly properly on a CAT(0) space $X$. If $Y\subset X$ is such that $[y_1,y_2]$ is $D$--contracting for every $y_1,y_2\in Y$, then the action of $G$ on $X_L$ is acylindrical along $Y$ for all $L\ge10D+3$.
\end{lemma}

\begin{proof}
We can assume that $D>1$. Given $\epsilon>0$, suppose that $y_1,y_2\in Y$ have $\dist_L(y_1,y_2)>16D\lceil\eps\rceil+10$. This implies that $\dist(y_1,y_2)>16D\lceil\eps\rceil+9$. Because $L\ge10D+3$, if $g\in G$ satisfies $\dist_L(y_i,gy_i)\le\eps$, then $\dist_{10D+3}(y_i,gy_i)\le\eps$. By Lemma~\ref{lem:synchronous_fellow_travelling:contracting}, any such $g$ also satisfies $\dist(y_1,gy_1)\le75D\lceil\eps\rceil$. By uniform properness, there are uniformly finitely many such $g$.
\end{proof}

% \textcolor{brown}{Another point: for the above Lemma, we only need $G$ to act uniformly properly on $Y.$ However, $G$ doesn't ``act" on $Y.$ If we really want, we can say that $G$ satisfies that for any $C_1$ there is a $C_2$ such that for any $y \in Y$ only $C_2$ many $g \in G$ can satisfy $\dist(y,gy)<C_1$ ... but this is too general and a bit ugly. I think assuming uniform properness is a good balance here.}

Since the definition of a stable subgroup includes the assumption that pairs of points are joined by uniformly contracting geodesics, we can combine Lemma~\ref{lem:Uniform_acylindricity_bottom} with  \cite[Thm~1.5]{abbottmanning:acylindrically} to obtain the following.

\begin{proposition}
Let $G$ be a group acting properly coboundedly on a CAT(0) space $X$ and let $H<G$ be finitely generated. The subgroup $H$ is stable if and only if there exists an $L$ such that $(G,X_L,H)$ is an A/QI triple.
\end{proposition}

%%%%%%%%%%%%%%%%%%%%%%%%%%%%%%%%%%%%%%%%
\section{The diameter dichotomy}  \label{subsec:tits_boundary}

In the previous section we investigated the case where some $X_L$ is unbounded. As this does not always need to be the case, it is desirable to have some control on the case where every $X_L$ is bounded. In this section we show that in most natural situations the spaces $X_L$ are actually uniformly bounded.

\begin{proposition} \label{prop:uniform_diameter}
Let $X$ be a cobounded CAT(0) space with the geodesic extension property, and let $L\in\mathbf N$. Either $\diam X_L\le4$, or $X_L$ is unbounded.
\end{proposition}

\begin{proof}
We show that the existence of an $L$--chain of length $n\ge4$ implies the existence of an $L$--chain of length $n+1$. See Figure~\ref{fig:diameter_dichotomy}. Let $c=\{h_1,\dots,h_n\}$ be an $L$--chain with $n\ge4$. Let $B$ be a ball in $X$ that meets both $h_1^-$ and $h_n^+$. Let $\alpha$ be a biinfinite geodesic such that $h_n$ is dual to $\alpha$. Fix a point $p\in B$, and let $r$ be such that the translates of $p$ are $r$--dense in $X$. Let $q\in\alpha\cap h_n^+$ satisfy $\dist(q,h_n)=r+L+\diam B+2$, and fix $g\in\isom X$ such that $\dist(gp,q)\le r$.

Let $q'$ be the point in $\alpha\cap h_n^+$ with $\dist(q',h_n)=L+2$, and let $\{h'_1,\dots,h'_{L+1}\}$ be a chain dual to $\alpha$ separating $q'$ from $\alpha\cap h_n$. By the choice of $g$, we have $gB\subset{h'_{L+1}}^{\hspace{-4mm}+}$. Moreover, we have $h_n\subset{h'_1}^-$. As every $gh_i$ meets $gB$, any $gh_i$ meeting $h_n$ must meet every $h'_j$, so no two elements of $gc$ can meet $h_n$, as $gc$ is an $L$--chain. In particular, there is an $L$--subchain $c'=\{gh_{i_1}, gh_{i_2}\}\subset gc$ contained in $h_n^+$. After inverting the orientation of $c'$, Lemma~\ref{lem:gluing_disjoint_chains} produces an $L$--subchain of $c\cup c'$ of length $n+1$.
\end{proof} 

\begin{figure}[ht]
\includegraphics[width=14cm]{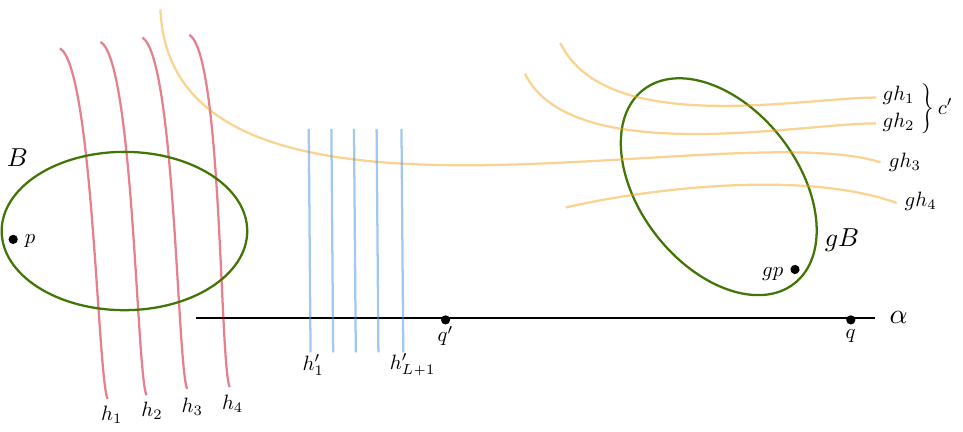}\centering
\caption{The proof of Proposition~\ref{prop:uniform_diameter}, illustrated with $n=4$.} \label{fig:diameter_dichotomy}
\end{figure}

In particular, if no $X_L$ is unbounded, then every $X_L$ has diameter at most 4.  When $X$ is proper we can strengthen the previous statement with Corollary~\ref{cor:diameter_2} below. This does not obsolete Proposition~\ref{prop:uniform_diameter}, though. For one thing, its proof is elementary and entirely self-contained. For another, Corollary~\ref{cor:diameter_2} does not say anything about the diameters of the bounded $X_L$ in the case where some $X_{L'}$ is unbounded.

Whilst the proof of Proposition~\ref{prop:uniform_diameter} is elementary, it requires that we begin with an $L$--chain of length at least four to conclude that $X_L$ is unbounded, leaving some mystery as to the significance of $L$--chains of length two and three. By using more advanced machinery, we shall show that the existence of even two separated curtains implies the existence of an unbounded $X_L$.

Recall that the \emph{visual boundary} $\partial X$ as a set is defined to be the collection of all equivalence classes of geodesic rays, where two rays are equivalent if they are at finite Hausdorff-distance. Equivalently, it is the set of all geodesic rays emanating from a fixed basepoint $\go$. Because of this, if the basepoint is understood then we shall often fail to distinguish between an element of $\partial X$ and the representative with that basepoint. We write $\alpha^\infty$ for the element of the visual boundary represented by a geodesic ray $\alpha$. 

\begin{definition}[Angle]
If $\sigma_1$ and $\sigma_2$ are geodesics with $\sigma_1(0)=\sigma_2(0)=x$, then $\angle_x(\sigma_1,\sigma_2)=\lim_{t\to0}\gamma(t)$, where $\gamma(t)$ is the angle at $x$ in the comparison triangle for $[x,\sigma_1(t),\sigma_2(t)]$. If $\xi_1,\xi_2\in\partial X$, then $\angle(\xi_1,\xi_2)=\sup\{\angle_x(\xi_1,\xi_2)\,:\,x\in X\}$.
\end{definition}

The angle defines a distance and hence a topology on the visual boundary. One of the prominent features of curtains is that they separate the CAT(0) space, and it turns out that curtains have a natural notion of limit that disconnects the boundary. 

\begin{definition}[Limit of a curtain]
For a curtain $h=h_\alpha$ with pole $P$, write 
\[
\Lambda(h)=\{\xi\in\partial X\,:\, \text{there exists }p\in P \text{ such that } \pi_\alpha[p,\xi]=p\}.
\]
\end{definition}

\begin{lemma} \label{lem:boundary_separation}
If $\partial X$ is path-connected and $X$ has the geodesic extension property then $\Lambda(h)$ is non-empty and separates $\partial X$ into two components.
\end{lemma}

\begin{proof}
Let $x\in P$, and let $\gamma\colon[a,b]\to\partial X$ be a path from $\alpha^{-\infty}$ to $\alpha^\infty$ which exist since $X$ has the geodesic extension property. By \cite[Prop.~II.9.2]{bridsonhaefliger:metric}, the map $\phi\colon[a,b]\to[0,\pi]$ defined by $\phi(t)= \angle_x (\alpha^\infty, \gamma(t))$ is continuous. Since $\phi(a)=\pi$ and $\phi(b)=0$, there is some $\eta=\gamma(t_0)$ such that $\angle_x(\alpha^\infty,\eta)=\frac{\pi}{2}$. We claim that $\pi_\alpha[x,\eta] = x$, which implies that $\gamma(t_0)\in \Lambda(h)$. 

Suppose that there is $y\in [x,\alpha(t_0)]$ such that $\pi_\alpha(y) \neq x$. By \cite[Prop.~II.2.4]{bridsonhaefliger:metric} we have $\angle_{\pi_\alpha(y)}(y,x) \geq \frac{\pi}{2}$. Moreover, as $\alpha^{-\infty}$ and $\alpha^\infty$ are opposite ends of a geodesic, we have $\angle_x(\eta,\alpha^{-\infty}) = \frac{\pi}{2}$ \cite[I.3.9, I.3.10]{ballmann:lectures}. Thus, the triangle $[x,y, \pi_{\alpha}(y)]$ has two angles of size at least $\frac{\pi}{2}$ and no ideal vertex, which contradicts the CAT(0)-inequality. Thus any path from $\alpha^{-\infty}$ to $\alpha^\infty$ must intersect $\Lambda(h)$, providing the result.
\end{proof}

Our goal now is to obtain lower bounds on the angles between various points in the visual boundary. We first note in Corollary~\ref{cor:perpendicular_to_dual} that the angle between the endpoints of a geodesic $\alpha$ and points of $\Lambda(h_\alpha)$ is at least $\frac\pi2$. 

\begin{lemma} \label{lem:backwards_means_perpendicular}
Let $\alpha$ be a geodesic ray. If $\beta$ is a geodesic ray based at $\alpha(t_0)$ such that $\pi_\alpha\beta\subset\alpha|_{[0,t_0]}$, then $\angle(\alpha^\infty,\beta^\infty)\ge\frac\pi2$. 
\end{lemma}

\begin{proof}
We have $\angle(\alpha^\infty,\beta^\infty)\ge\angle_{\alpha(t_0)}(\alpha^\infty,\beta^\infty)$ by definition. As $t\to0$, the fact that $\pi_\alpha\beta(t)\subset\alpha|_{[0,t_0]}$ implies that the comparison angle for $[\alpha(t_0),\beta(t),\alpha(t_0+t)]$ at $\alpha(t_0)$ is at least $\frac\pi2$.
\end{proof}

\begin{corollary} \label{cor:perpendicular_to_dual}
If $\partial X$ is connected and $h=h_\alpha$ is a curtain, then $\angle(\alpha^\infty,\xi)\ge\frac\pi2$ for all $\xi\in\Lambda(h)$.
\end{corollary}

Next we aim to bound the angle between a pair of separated curtains; this is done in Proposition~\ref{prop:separated_means_perpendicular}.  We recall the following lemma.

\begin{lemma}[{\cite[Thm~II.4.4]{ballmann:lectures}}] \label{lem:comparison_angle}
Let $\xi_1,\xi_2\in\partial X$, $x\in X$. Let $\sigma_i$ be the geodesic ray from $x$ to $\xi_i$. The quantity $c=\lim_{t\to\infty}\frac1t\dist(\sigma_1(t),\sigma_2(t))$ is independent of $x$. In the Euclidean triangle with sides of length 1, 1, and $c$, the angle opposite $c$ is $\angle(\xi_1,\xi_2)$.
\end{lemma}

We use the following version of Lemma~\ref{lem:backwards_means_perpendicular} that does not require any information on basepoints.

\begin{lemma} \label{lem:limiting_distance}
Let $\alpha$ and $\beta$ be geodesic rays. If $\pi_\alpha\beta\subset\alpha|_{[0,t_0]}$, then
\[
\lim_{t\to\infty}\frac1t\dist(\alpha(t),\beta(t))\ge\sqrt2.
\]
\end{lemma}

\begin{proof}
Let $\beta(t_1)=\pi_\beta\alpha(t_0)$ and write $\delta=\dist(\alpha(t_0),\beta(t_1))$. By the reverse triangle inequality, $\dist(\alpha(t_0),\beta(t))\ge\dist(\beta(t_1),\beta(t))-\delta$ for all $t$. 

Now, for $t\ge t_1$ let $\gamma_t$ be the geodesic from $\alpha(t_0)$ to $\beta(t)$. By convexity of the metric, $\gamma_t$ lies in the $\dist(\alpha(t_0),\pi_\alpha\beta(t))$--neighbourhood of $[\beta(t),\pi_\alpha\beta(t)]$, so the fact that $\pi_\alpha[\pi_\alpha\beta(t),\beta(t)]=\pi_\alpha\beta(t)$ means that $\pi_\alpha\gamma_t\subset\alpha|_{[0,t_0]}$. It follows that if $t>\max\{t_0,t_1\}$, then $\angle_{\alpha(t_0)}(\alpha(t),\beta(t))\ge\frac\pi2$. Using convexity of the metric, we use this to compute
\begin{align*}
\dist(\alpha(t),\beta(t))^2 \,&\ge\, \dist(\alpha(t),\alpha(t_0))^2 + \dist(\alpha(t_0),\beta(t))^2 \\
    &\ge\, (t-t_0)^2 + (t-t_1-\delta)^2,
\end{align*}
and the result follows immediately.
\end{proof}

\begin{proposition} \label{prop:separated_means_perpendicular}
Suppose that $X$ has the geodesic extension property and path-connected visual boundary. Let $h$, $h'$ be curtains with respective poles $P$ and $P'$. If $h$ and $h'$ are separated, then $\angle(\xi,\xi')\ge\frac\pi2$ for all $\xi\in\Lambda(h)$, $\xi'\in\Lambda(h')$.
\end{proposition}

\begin{proof}
Let $\xi\in\Lambda(h)$ and let $p\in P$ be such that $\pi_P[p,\xi]=p$. Let $\alpha=[p,\xi]$. Take any point $\xi'\in\Lambda(h')$, and let $\beta\subset h'$ be a geodesic ray based in $P'$ with $\beta^\infty=\xi'$. 

For $i\in\{0,1\}$, consider the chains $c_i=\{h_{\alpha,2n+i}:n\in\mathbf N\}$. Since $h$ and $h'$ are $L$--separated for some $L$, at most $L$ of each can meet $h'$, and hence $\beta$ can meet at most $L$ of each. Since $c_0\cup c_1$ is a cover of $\alpha$ and $\pi_\alpha\beta$ is nonempty, this means that there exists $t_0\ge0$ such that $\pi_\alpha\beta\subset\alpha|_{[0,t_0]}$. 

Let $\delta=\dist(\alpha(t_0),\beta(0))$. Now let $\gamma$ be the geodesic ray based at $\alpha(t_0)$ with $\gamma^\infty=\xi'$. By the flat strip theorem, $\dist(\gamma(t),\beta(t))\le\delta$ for all $t$. In particular, $c=\lim_{t\to\infty}\frac1t\dist(\alpha(t),\gamma(t)) = \lim_{t\to\infty}\frac1t\dist(\alpha(t),\beta(t))$. According to Lemma~\ref{lem:limiting_distance}, we have $c\ge\sqrt2$. Lemma~\ref{lem:comparison_angle} now tells us that $\angle(\xi,\xi')$ is equal to the isosceles angle in the Euclidean triangle with side lengths 1, 1, and $c$, which is at least $\frac\pi2$.
\end{proof}

Combining these results gives us information about the \emph{Tits boundary} in the case where $X$ has a pair of separated curtains.

\begin{definition}[Tits boundary]
The \emph{Tits metric} on $\partial X$ is the path-metric induced by $\angle(\cdot,\cdot)$. The \emph{Tits boundary} $\partial_TX$ of $X$ is the (extended) metric space obtained in this way.
\end{definition}

\begin{corollary} \label{cor:boundary_diameter}
Let $X$ be a CAT(0) space with the geodesic extension property and path-connected visual boundary. If $X$ has a pair of separated curtains, then the Tits boundary $\partial_TX$ of $X$ has diameter at least $\frac{3\pi}2$.
\end{corollary}

\begin{proof}
Let $h=h_\alpha$ and $k=k_\beta$ be $L$--separated curtains. We may assume that $k\subset h^-$ and $h\subset k^-$. By Lemma~\ref{lem:boundary_separation}, any path in $\partial X$ from $\alpha^\infty$ to $\beta^\infty$ must pass through both $\Lambda(h)$ and $\Lambda(k)$. Corollary~\ref{cor:perpendicular_to_dual} and Proposition~\ref{prop:separated_means_perpendicular} show that, in the angle metric on $\partial X$, the length of such a path must be at least $\frac{3\pi}2$.
\end{proof}

Bringing in group actions, we are now in a position to prove the main result of this section. The proof relies on work of Guralnick--Swenson \cite{guralnikswenson:transversal}, which itself relies on ideas from \cite{papasogluswenson:boundaries}.

\begin{theorem} \label{thm:separated_implies_contracting}
Let $X$ be a CAT(0) space with the geodesic extension property, and let $G$ be a group acting properly cocompactly on $X$. If $X$ has a pair of separated curtains, then $G$ has a rank-one element.
\end{theorem}

\begin{proof}
If $\partial X$ is not path-connected, then by definition of the Tits metric, for every $\xi\in\partial_TX$ there exists $\zeta\in\partial_TX$ at infinite Tits distance from $\xi$. Thus \cite[Prop.~1]{ballmannbuyalo:periodic} shows that if $\partial X$ is not path-connected, then $G$ has a rank-one element. On the other hand, if $\partial X$ is path-connected then from Corollary~\ref{cor:boundary_diameter} we see that $\diam\partial_TX\ge\frac{3\pi}2$. According to \cite[Thm~3.12]{guralnikswenson:transversal}, this also shows that $G$ has a rank-one element.
\end{proof}

In view of Theorem~\ref{thm:rank_one_characterisation}, we get the following dichotomy for the diameters of the $X_L$.

\begin{corollary} \label{cor:diameter_2}
Let $X$ be a CAT(0) space with the geodesic extension property and admitting a proper cocompact group action. If $\diam X_L>2$ for some $L$, then only finitely many $X_L$ are bounded.
\end{corollary}

%%%%%%%%%%%%%%%%%%%%%%%%%%%%%%%%%%%%%%%%%%%%%%%%%%%%%%%%%%%%%%%%%%%%%%
\section{Higher-rank CAT(0) spaces} \label{sec:acyl_rr}

In this section, we apply our machinery to the coarse geometry of CAT(0) spaces without rank-one isometries. The main goal of the section is to complete the proof of the weak rank-rigidity statement, Theorem~\ref{icor:rr}, by showing that if $X$ is a CAT(0) space admitting a proper cocompact group action and every $X_L$ is bounded, then $X$ is wide. This is Proposition~\ref{prop:wide}.

 We start with the main technical lemma, which allows us to shrink polygons to efficiently avoid balls.

\begin{lemma}[Circumnavigation Lemma] \label{lem:short_n-gon}
Let $x_1,\dots,x_n\in X$, with $n\ge3$, and write $x_{n+1}=x_1$. Let $p\in[x_1,x_2]$ and let $B$ be the closed $r$--ball about $p$. Suppose that every $[x_i,x_{i+1}]$ with $i>1$ is disjoint from the interior $\ring B$ of $B$. There is a path from $x_1$ to $x_2$ that avoids $\ring B$ and has length at most $8(nr+D)$, where $D=\dist(x_1,x_2)$.
\end{lemma}

\begin{proof}
We begin by modifying the set of $x_i$. See Figure~\ref{fig:short_n-gon}. Let $x_3'=x_3$. Given $x_i'$ for $i>2$, if $x_{i+2}$ exists then proceed as follows. If $[x_i',x_{i+2}]$ is disjoint from $B$, then delete $x_{i+1}$, relabel $x_j$ as $x_{j-1}$ for every $j>i+1$, and repeat with the new $x_{i+2}$ (if it exists). Otherwise, fix a point $x_{i+1}'\in[x_{i+1},x_{i+2}]$ such that $\dist([x_i',x_{i+1}'],m)=r$. After this process, we have points $x_1,x_2,x_3=x_3',x_4',\dots,x_m'$ with $m\le n$ such that $\dist([x_i',x_{i+1}'],p)=r$ for all $i\in\{3,\dots,m-1\}$.

\begin{figure}[ht]
\includegraphics[height=7cm]{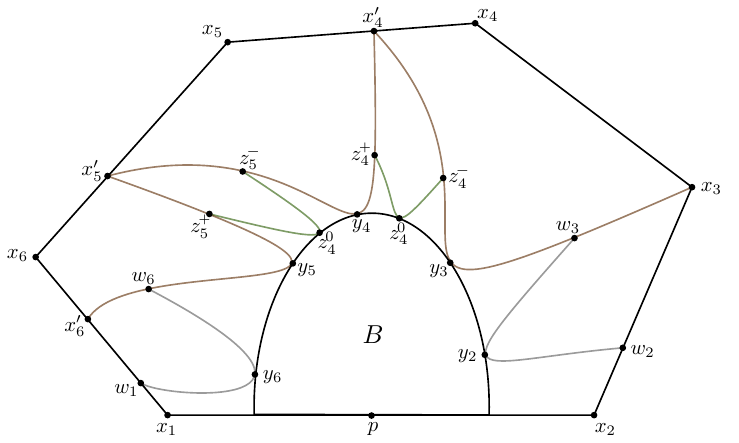}\centering
\caption{The construction in the proof of Lemma~\ref{lem:short_n-gon}, illustrated with $n=m=6$.} \label{fig:short_n-gon}
\end{figure}

Because $B$ is convex, there are unique points $y_i\in[x_i',x_{i+1}']$ with $\dist(y_i,p)=r$ for $i\in\{3,\dots,m-1\}$. (In the case $m=3$, set $y_2=x_2$, $y_3=x_1$.) Since $[x_1,y_{m-1}]$ meets $B$, by considering the family of geodesics with endpoints on $[x_m',x_1]$ and $[x_m',y_{m-1}]$, we can find $w_1\in[x_m',x_1]$ and $w_m\in[x_m',y_{m-1}]$ such that $[w_1,w_m]$ meets $B$ at a single point $y_m$. Similarly, we can find $w_2\in[x_3',x_2]$ and $w_3\in[x_3',y_3]$ such that $[w_2,w_3]$ meets $B$ at a single point $y_2$ (in the case $m=3$, these two paths are the same, so label it $[w_1,w_3]$ and write $w_2=w_3$). Note that by convexity of the metric, $\dist(w_1,w_m)\le\dist(x_1,y_{m-1})$ and $\dist(w_2,w_3)\le\dist(x_2,y_3)$, both of which are at most $r+D$. By the triangle inequality, we consequently get that $\dist(x_1,w_1)$ and $\dist(x_2,w_2)$ are at most $2(r+D)$.

Now, if $m\ge5$, for each $i\in\{4,\dots,m-1\}$ consider a family of geodesics with endpoints in $[x_i',y_{i-1}]$ and $[x_i',y_i]$. Because $y_i\in[x_i',x_{i+1}']$, from each one of these families we obtain points $z_i^-\in[x_{i-1}',x_i']$ and $z_i^+\in[x_i',x_{i+1}']$ such that the geodesic $[z_i^-,z_i^+]$ meets $B$ at a single point $z_i^0$. Again, convexity of the metric implies that $\dist(z_i^-,z_i^+)\le\dist(y_{i-1},y_i)\le2r$. Moreover, the triangle inequality gives $\dist(z_i^+,z_{i+1}^-)\le\dist(z_i^+,z_i^0)+\dist(z_i^0,z_{i+1}^0)+\dist(z_{i+1}^0,z_{i+1}^-)\le6r$.

Consider the path $P$ defined as the concatenation
\[
\hspace{-1.5cm}[x_1,w_1] \cup [w_1,w_m] \cup [w_m,z_{m-1}^+] \cup [z_{m-1}^+,z_{m-1}^-] \cup \bigcup_{i=2}^{m-4}\Big([z_{m-i+1}^-,z_{m-i}^+]\cup[z_{m-i}^+,z_{m-i}^-]\Big) \cup [z_4^-,w_3] \cup[w_3,w_2] \cup[w_2,x_2],
\]
ignoring any terms that are undefined if $m\le4$. It connects $x_1$ to $x_2$ and avoids $\ring B$. It suffices to bound the length $\ell(P)$. We have so far seen that
\[
\dist(x_1,w_1)\le2(r+D), \quad \dist(w_1,w_m)\le r+D, \quad \dist(z_i^+,z_i^-)\le2r,
\]\[
\dist(z_i^-,z_{i-1}^+)\le6r, \quad \dist(w_3,w_2)\le r+D, \quad \dist(w_2,x_2)\le2(r+D).
\]
This leaves us needing to bound $\dist(w_m,z_{m-1}^+)$ and $\dist(z_4^-,w_3)$. By the triangle inequality, $\dist(z_4^-,w_3)\le\dist(z_4^-,z_4^0)+\dist(z_4^0,y_2)+\dist(y_2,w_3)\le5r+D$, and similarly $\dist(w_m,z_{m-1}^+)\le5r+D$. Combining all of these, we get that 
\begin{align*}
\ell(P)\:&\le\: 2(r+D)+(r+D)+(5r+D)+(2r)+(m-5)(8r)+(5r+D)+(r+D)+2(r+D) \\
    &\le\: 18r+8D+8r\max\{0,m-5\} \:\le\: 18r+8D+8r(n-3). \qedhere
\end{align*}
\end{proof}

We now show that if $\diam X_L$ is uniformly bounded, we can verify the hypotheses of the circumnavigation lemma.

\begin{lemma} \label{lem:distant_curtain}
Let $L\ge2$. Suppose that curtains $h_1,h_2$ are not $L$--separated, and let $B$ be a ball in $X$ with radius $r$. If $r\le\frac{L-1}2$, then there is a curtain that meets $h_1$ and $h_2$ but not $B$.
\end{lemma}

\begin{proof}
By Remark~\ref{rem:thick_curtains}, any chain of curtains all of whose elements meet $B$ must have cardinality at most $\lceil2r\rceil+1$. By assumption, there is a chain $c$ of curtains of cardinality $L+1$ such that every element of $c$ meets both $h_1$ and $h_2$. If $r\le\frac{L-1}2$, then $\lceil2r\rceil+1\le L$, so some element of $c$ is disjoint from $B$.
\end{proof}

\begin{lemma}\label{lem:bounding_length}
Let $L\ge2$ and suppose that $\gamma$ is a geodesic with dual curtains $h_1$ and $h_2$ that are not $L$--separated. Let $x_1\in h_1$ and $x_2\in h_2$ be the points of $\gamma$ with $\dist(x_1,x_2)=\dist(h_1,h_2)=D$. If $p\in[x_1,x_2]$ is such that the interior of $B=B(p,\frac{L-1}2)$ is disjoint from $h_1$ and $h_2$, then there is a path from $x_1$ to $x_2$ of length at most $8(3L+D)$ that avoids the interior of $B$.
\end{lemma}

\begin{proof}
By Lemma~\ref{lem:distant_curtain}, there is a curtain $h$ meeting both $h_1$ and $h_2$ but not $B$. By star convexity there exist $x_3\in h\cap h_2$ and $x_6\in h\cap h_1$ with $[x_6,x_1]\in h_1$ and $[x_2,x_3]\in h$. Moreover, there are points $x_4$ and $x_5$ in the pole of $h$ such that $[x_3,x_4]\cup[x_4,x_5],[x_5,x_6]\subset h$. In particular, the conditions of Lemma~\ref{lem:short_n-gon} are met with $r=\frac{L-1}2$.
\end{proof}

Let us now set up the necessary notation for asymptotic cones. Asymptotic cones were introduced by Gromov in \cite{gromov:groups} and later clarified by van den Dries and Wilkie \cite{vandendrieswilkie:gromov's}. We refer the reader to \cite{drutusapir:tree-graded} for a more thorough treatment.

\begin{definition}[Asymptotic cone]
Let $\omega$ be a \emph{non-principal ultrafilter} and let $(\lambda_n)$ be a divergent sequence of positive numbers. Let $(X,\dist)$ be a metric space, and consider the sequence of metric spaces $X_n=(X,\frac1\lambda_n\dist)$. Define an extended pseudometric $\delta_\omega$ on $\prod_{i=1}^\infty X_n$ by setting $\delta_\omega((x_n),(y_n))=r$ if for all $\eps>0$ we have $\{n\,:\,r-\eps<\frac1{\lambda_n}\dist(x_n,y_n)<r+\eps\}\in\omega$, and $\delta_\omega((x_n),(y_n))=\infty$ if there is no such $r$. Fix a basepoint $\go\in X$. The metric quotient of the pseudometric space consisting of all $(x_n)$ with $\delta_\omega((x_n),(\go))<\infty$ is an \emph{asymptotic cone} of $X$. We denote this metric space by $(X_\omega,\dist_\omega)$, suppressing both the scaling sequence and the basepoint. If $(x_n)$ is a sequence of points in $X$ and $x_\omega\in X_\omega$, then we write $(x_n)\to_\omega x_\omega$ if $(x_n)$ is a representative of $x_\omega$.
\end{definition}

\begin{definition}[Wide]
Following \cite{drutusapir:tree-graded}, we say that a metric space is \emph{wide} if none of its asymptotic cones have cut-points.
\end{definition}

We are now ready to prove Proposition~\ref{prop:wide}. Note that this can also be obtained as a consequence of \cite[Prop.~1.1]{drutumozessapir:divergence} and the observation that Lemma~\ref{lem:bounding_length} essentially provides \emph{linear divergence}. We provide a proof in the interests of self-containment.

\begin{proposition} \label{prop:wide}
Let $X$ be a CAT(0) space admitting a proper cocompact group action. If no $X_L$ is unbounded, then $X$ is wide.
\end{proposition}

\begin{proof}
By Corollary~\ref{cor:diameter_2}, every $X_L$ has diameter at most 2. In other words, no pair of curtains are $L$--separated for any $L$. 

Suppose that for some ultrafilter $\omega$ and some scaling sequence $(\lambda_n)$, the asymptotic cone $X_\omega$ has a cut-point $p_\omega$. Note that $X_\omega$ is a CAT(0) space \cite[Cor.~II.3.10]{bridsonhaefliger:metric}. Let $x_\omega,y_\omega\in X_\omega$ be separated by $p_\omega$, and let $\eps=\min\{\dist_\omega(x_\omega,p_\omega),\dist_\omega(p_\omega,y_\omega)\}$. 
Fix sequences $(x_n)\ora x_\omega$ and $(y_n)\ora y_\omega$. Note that the set $U = \{n\,:\, \dist(x_\omega,y_\omega) - \frac\eps4 < \frac1{\lambda_n}\dist(x_n,y_n) < \dist(x_\omega,y_\omega) + \frac\eps4\}$ is an element of $\omega$. By removing finitely many elements of $U$, we may also assume that $\frac\eps4\lambda_n\ge\frac12$ for all $n\in U$.

For each $n\in U$, we can fix a point $p_n\in[x_n,y_n]$ with $\dist(x_n,p_n)=\lambda_n\dist_\omega(x_\omega,p_\omega)$. Note that we have $\dist(p_n,y_n)>\frac{3\eps}4\lambda_n$. Moreover, by construction we have $(p_n)\ora p_\omega$.
Let $z^1_n$ be the point on $[x_n,p_n]$ with $\dist(z^1_n,p_n)=\frac\eps2\lambda_n$, and let $z^2_n$ be the point on $[p_n,y_n]$ with $\dist(p_n,z^2_n)=\frac\eps2\lambda_n$, which exists because $n\in U$. We use these points to define curtains: for $i\in\{1,2\}$, let $h^i_n$ be the curtain dual to $[x_n,y_n]$ at $z^i_n$. Because $\frac\eps4\lambda_n\ge\frac12$, the curtain $h^1_n$ separates $x_n$ from $p_n$, and because $\dist(z^2_n,y_n)>\frac\eps4\lambda_n$, the curtain $h^2_n$ separates $p_n$ from $y_n$. 
Because no pair of curtains is $L$--separated for any $L$, the curtains $h^1_n$ and $h^2_n$ are not $\eps\lambda_n$--separated. Moreover, the $h^i_n$ were constructed so that they are disjoint from the interior of the ball $B_n=B(p_n,\frac{\eps\lambda_n-1}2)$. By Lemma~\ref{lem:bounding_length}, there is a path from $z^1_n$ to $z^2_n$ that avoids $B_n$ and has length at most $1+8(3\eps\lambda_n+(\eps\lambda_n-1))$.

In the limit we obtain a path in $X_\omega$ from $(z^1_n)_\omega$ to $(z^2_n)_\omega$ that avoids the interior of $B(p_\omega,\frac\eps2)$ and has length at most $32\eps$. Concatenating this with subintervals of $[x_\omega,y_\omega]$, we see that $p_\omega$ cannot separate $x_\omega$ from $y_\omega$, a contradiction.
\end{proof}

\begin{remark}
By using Proposition~\ref{prop:uniform_diameter} instead of Corollary~\ref{cor:diameter_2}, a similar argument to the above proof of Proposition~\ref{prop:wide} can be used to show that $X$ is wide under the assumptions that $X$ is cobounded (not necessarily proper), has the geodesic extension property, and no $X_L$ is unbounded.
\end{remark}

\begin{theorem} \label{thm:wide_dichotomy}
Let $G$ be a group acting properly cocompactly on a CAT(0) space $X$. One of the following holds. 
\begin{itemize}
\item   Some $X_L$ is unbounded, in which case $G$ has a rank-one element and is either virtually cyclic or acylindrically hyperbolic.
\item   Every $X_L$ has diameter at most 2, in which case $G$ is wide.
\end{itemize}
\end{theorem}

\begin{proof}
If every $X_L$ has diameter at most 2, then $G$ is wide by Proposition~\ref{prop:wide}. If some $X_L$ has diameter more than 2, then Corollary~\ref{cor:diameter_2} shows that some $X_L$ is unbounded. In this case, the consequences come from Lemma~\ref{lem:unbounded_rank_one} and Proposition~\ref{prop:nonuniform_acylindricity}. 
\end{proof}

%%%%%%%%%%%%%%%%%%%%%%%%%%%%%%%%%%%%%%%%%%%%%%%%%%%%%%%%%%%%%%%%%%%%%%
\section{Curtain boundaries} \label{section:curtain_boundaries}

The main goal of this section is to investigate the relationship between the Gromov boundaries of the spaces $X_L$ and the visual boundary of the corresponding CAT(0) space $X$. We shall consider the visual boundary as being equipped with the \emph{cone topology}, which is determined by the neighbourhood basis consisting of the following sets. Given a geodesic ray $b$ emanating from a fixed basepoint $\go$, for constants $r \geq 0$ and $\epsilon >0$, let
\[
U(b^\infty,r,\epsilon):=\{c^\infty \in \partial X \::\: c(0)=\go \text{ and }\dist(c(r), b)<\epsilon\}.
\]

Let $\cal B_L$ be the subspace of the visual boundary $\partial X$ consisting of all geodesic rays $b$ emanating from $\go$ such that there is an infinite $L$-chain crossed by $b$, and let $\cal B=\bigcup_{L\ge0}\cal B_L$. These are equipped with the subspace topology from the cone topology on the visual boundary. We call $\cal B$ the \emph{curtain boundary} of $X$.

\begin{theorem}\label{thm:Gromov_boundaries_inside_visual_boundaries} 
Let $X$ be a proper CAT(0) space. For each $L \geq 0$, we have the following.
\begin{enumerate}
\item Each point in $\cal B_L$ is a visibility point of the visual boundary $\partial X$. \label{item:visibility}
\item The subspace $\cal B_L \subseteq \partial X$ is $\isom X$--invariant. \label{item:isom}
\item The identity map $\iota:X \rightarrow X_L$ induces an $\isom X$--equivariant homeomorphism $\partial \iota:\cal B_L \rightarrow \partial X_L$. \label{item:homeomorphism}
\end{enumerate}
\end{theorem}

Recall that $b^\infty \in \partial X$ is said to be a \emph{visibility point} if for any other $c^\infty$ in $\partial X$ there exists a geodesic line $l$ at finite Hausdorff-distance from $b \cup c$. In other words, Theorem~\ref{thm:Gromov_boundaries_inside_visual_boundaries} says that the visual boundary $\partial X$ contains $\isom X$--invariant copies of the Gromov boundaries $\partial X_L$.

\begin{definition}[Separation from boundary points, crossing] 
Let $b:[0,\infty) \rightarrow X$ be a geodesic ray and let $h$ be a curtain. We say that $h$ \emph{separates} $b^\infty$ from $A\subset X$ if there exists $t_0$ such that $h$ separates $A$ from $b|_{(t_0, \infty)}$. We say that a geodesic ray $b$ based at $\go$ \emph{crosses} an infinite chain $\{h_i\}$ of curtains if every $h_i$ separates $\go$ from $b^\infty$.
\end{definition}

\begin{remark}
Because curtains may not be convex, it is \emph{a priori} possible for a geodesic ray $b$ to meet every element of an infinite chain of curtains $\{h_i\}$, none of which separates $b(0)$ from $b^\infty$. However, Lemma~\ref{lem:convexity_around_L_chains} ensures that if $b$ meets every element of an infinite $L$-chain, then it crosses it.
\end{remark}

The following lemma establishes part~\eqref{item:visibility} of Theorem~\ref{thm:Gromov_boundaries_inside_visual_boundaries}.

\begin{lemma}
If $X$ is proper, then every point $b^\infty \in \cal B$ is a visibility point of $\partial X.$
\end{lemma}

\begin{proof}
Let $\{h_i\}$ be an infinite $L$--chain dual to $b$, which exists by Lemma~\ref{lem:updating_curtains}, and orient the $h_i$ so that $\go\in h_i^-$. Let $c$ be any geodesic ray with $c(0)=b(0)=\go$ and $c^\infty\ne b^\infty$. According to Corollary~\ref{cor:diverging_geodesics}, there is some $k$ such that $c\subset h_{k-1}^-$. Let $x_n=[c(n),b(n)]$. Lemma~\ref{lem:bounding_distance_with_dual_curtains} tells us that there exist $p \in h_k \cap b$, an integer $m \geq 1$, and points $p_n \in [c(n),b(n)]$ such that $\dist(p_n,p) \leq 2L+2$ for all $n \geq m$. Since balls in $X$ are compact, the statement follows from \cite[Lem.~II.9.22]{bridsonhaefliger:metric}.
\end{proof}

The action of $\isom X$ on $X$ induces an action on $\partial X$. Indeed, if $\alpha$ is a geodesic ray based at $\go$ and $g\in\isom X$, then $g\alpha$ is a geodesic ray, and there is a unique ray $\beta$ based at $\go$ with $\beta^\infty=(g\alpha)^\infty$ \cite[II.8.2]{bridsonhaefliger:metric}. We declare $g(\alpha^\infty)=\beta^\infty$.

\begin{lemma} \label{lem:invariance_of_B_L}
For any CAT(0) space $X$, the set $\cal B_L$ is $\isom X$--invariant.
\end{lemma}

\begin{proof}
By Lemma~\ref{lem:updating_curtains}, any geodesic ray $b$ with $b^\infty\in\cal B_L$ crosses an infinite $L$--chain dual to $b$. For any $g\in\isom X$, the geodesic ray $gb$ crosses an infinite $L$--chain $\{h_i\}$ dual to $gb$. The unique geodesic ray $c$ emanating from $\go$ with $c^\infty=gb^\infty$ lies at finite Hausdorff-distance from $gb$, so $\pi_{gb}(c)$ is unbounded, and hence $c$ crosses all but finitely many $h_i$.
\end{proof}

As a step towards Theorem~\ref{thm:Gromov_boundaries_inside_visual_boundaries}, and for its own interest, we use curtains to introduce a new topology on $\partial X$. We relate it to the standard cone topology in Theorem~\ref{thm:comparing_topologies}.

\begin{definition}[Curtain topology] 
Let $b$ be a geodesic ray emanating from $\go$. For each curtain $h$ dual to $b$, let 
\[
U_h(b^\infty)=\{a^\infty \in \partial X \,:\, h \text{ separates } \go=a(0) \text{ from } a^\infty\}.
\]
% Consider the collection
% \[
% \mathcal{B} = \{U_h(b^\infty) \::\: b \text{ is a geodesic ray starting at $\go$ and } \,h \,\text{ is a curtain dual to }b\}.
% \] 
We define the \emph{curtain topology} as follows: a set $U\subset\partial X$ is open if for each $b^\infty \in U$, there is some $U_h(b^\infty)$ with $U_h(b^\infty)\subset U.$ It is immediate that such a description yields a topology on $\partial X.$ 
\end{definition}

We remark that the definition of the curtain topology above is inspired by the hyperplane topology introduced by Incerti-Medici in \cite{incertimedici:comparing}.

\begin{example} \label{eg:cone_vs_curtain}
As an example of the difference between the curtain and cone topologies, consider the Euclidean plane $\bf E^2$. A simple computation shows that the curtain topology on $\partial\bf E^2$ is the trivial topology, in contrast to the cone topology, which is that of the circle $S^1$. Note that $\cal B=\varnothing$ in this example.
\end{example}

Example~\ref{eg:cone_vs_curtain} fits the ideology that the curtain topology should detect only negative curvature. Our next result essentially shows that it sees all of it. It also shows that the cone topology on $\partial X$ is either equal to or finer than the curtain topology.
 
\begin{theorem} \label{thm:comparing_topologies} 
The identity map $(\partial X,\cal T_{\mathrm{Cone}})\to(\partial X,\cal T_{\mathrm{Curtain}})$ is continuous. Moreover, the curtain and cone (subspace) topologies agree on $\cal B$.
\end{theorem}

\begin{proof} 
We start by showing that every open set $O$ in $\cal T_{\mathrm{Curtain}}$ is also open in $\cal T_{\mathrm{Cone}}$. Let $b^\infty\in O$. By definition, there is some $h$ dual to $b$ such that $U_h(b^\infty)\subseteq O$. Let $r=\dist(\go,h)+5$. We shall show that $U(b^\infty,r,1)\subset U_h(b^\infty)$. For this, let $c^\infty\in U(b^\infty,r,1)$, so that $\dist(c(r),b)<1$, and note that $\dist(b(r),c(r))<2$. By the choice of $r$, we have $\dist(h,c(r))>\dist(h,b(r))-2=2$, and $h$ separates $\go$ from $\pi_b(c(r))$. Suppose that $c|_{(r,\infty)}$ meets $h$ at $c(t)$. In this case, there is some $s\in(0,r)$ such that $\pi_bc(s)=\pi_bc(t)\in h$. By convexity of the metric, $\dist(c(s),\pi_bc(s))<1$, so
\begin{align*}
\dist(c(s),c(t)) \,&\le\, \dist(c(s),\pi_bc(s))+\dist(\pi_bc(s),c(t)) \\
    &=\, \dist(c(s),\pi_bc(s))+\dist(b,c(t)) \\
    &\le\, 1+\dist(b,c(r))+\dist(c(r),c(t)) \\
    &<\, 2+\dist(c(r),c(t)).     
\end{align*}    
But $c(s)\in h$, so $\dist(c(s),c(r))>2$, contradicting the fact that $c$ is a geodesic. Thus $h$ separates $\go$ from $c^\infty$, so $c^\infty\in U_h(b^\infty)$. This shows that $\cal T_{\mathrm{Curtain}}\subset\cal T_{\mathrm{Cone}}$. 

We now show that, when restricted to $\cal B$, the curtain topology is at least as fine as the cone topology. Fix some $U(b^\infty,r,\epsilon)$ for some geodesic ray $b$ with $b^\infty \in \cal B$. Let $\{h_i\}$ be an infinite $L$-chain dual to $b$, which exists by Lemma~\ref{lem:updating_curtains}, and let $t_i$ be such that $b(t_i)\in h_i$. Fix $k$ large enough that $\frac r{t_k-1}(2L+1)<\epsilon$. We shall show that $U_{h_{k+1}}(b^\infty)\subset U(b^\infty,r,\epsilon)$. For this, let $b'^\infty\in U_{h_{k+1}}(b^\infty)$. According to Lemma~\ref{lem:bounding_distance_with_dual_curtains}, if $b'(t'_k)\in h_k$ then $\dist(b'(t'_k),b)\le2L+1$. By convexity of the metric, we have $\dist(b'(r),b)\le\frac r{t'_k}(2L+1)$. Because $h_k$ is dual to $b$, we also have $t'_k\ge t_k-1$. Hence $\dist(b'(r),b)<\epsilon$, so $b'^\infty\in U(b^\infty,r,\epsilon)$. 
\end{proof}

This gives a purely combinatorial description of the subspace topology on $\cal B$ (and hence on each $\cal B_L$) via curtains. We show in Corollary~\ref{cor:Morse} below that the topology on the \emph{Morse boundary} can also be described combinatorially via curtains.

It remains to establish item~\eqref{item:homeomorphism} of Theorem~\ref{thm:Gromov_boundaries_inside_visual_boundaries}, which we do in the next proposition.

\begin{proposition} \label{prop:boundary_embeds}
For a proper CAT(0) space $X$, the identity map $\iota:X \rightarrow X_L$ induces an $\isom X$--equivariant homeomorphism $\partial \iota: \cal B_L \rightarrow \partial X_L.$ 
\end{proposition} 

\begin{proof} The proof consists of the following four steps.
\begin{enumerate}[leftmargin=7mm]
\item (Existence and continuity.) 
Since the map $\iota:X \rightarrow X_L$ is $(1,1)$--coarsely Lipschitz and $X_L$ is roughly isometric to a geodesic hyperbolic space, the existence and continuity of $\partial \iota$ are exactly given by \cite[Lem.~6.18]{incertimedicizalloum:sublinearly}. 

\medskip

\item (Injectivity.) \label{item:injective}
If $b^\infty\in\cal B$, then by Lemma~\ref{lem:updating_curtains} there is an infinite $L$--chain $\{h_i\}$ dual to $b$ with $\go\in h_i^-$ for all $i$. For any other $c^\infty\in\cal B$, Corollary~\ref{cor:diverging_geodesics} shows that $c\subset h_k^-$ for some $k$. In particular, if $x_n\in b\cap h_{k+n}$, then $\dist_L(x_n,c)\ge n$, so $b$ does not lie in a finite $X_L$--neighbourhood of $c$.

\medskip
    
\item (Surjectivity.) \label{item:onto}
Let $q:[0,\infty) \rightarrow X_L$ be any quasigeodesic ray with $q(0)=\iota(\go)$, and let $(x_n)\subset q$ be an unbounded sequence with $x_0=\iota(\go)$. Consider the path $q'=\bigcup_{n \in \mathbf{N}}[x_{n-1},x_n] \subset X$, which crosses an infinite $L$--chain $\{h_i\}$ by Corollary~\ref{cor:QG_cross_chains}. Let $p_i\in q'\cap h_i$. As $X$ is proper, there is some geodesic ray $b$ in $X$ emanating from $\go$ and some subsequence $(p_{i_j})$ such that the geodesics $[\go, p_{i_j}]$ converge uniformly to $b$ on compact sets. 
    
Let us show that $b$ meets every $h_i$. Since $\{h_i\}$ is a chain, it suffices to show that $b$ meets infinitely many $h_i$. Suppose that, on the contrary, there is some $k$ such that $b\in h_i^-$ for all $i \geq k$. Let $k'=k+(4L+10)(2L+3)+2$. The curtains $h_{k}$, $h_{k'}$ are separated by an $L$--chain of length $(4L+10)(2L+3)$, and hence, using Lemma~\ref{lem:updating_curtains} with $A=h_{k}$ and $B=h_{k'},$ they are separated by an $L$--chain $\{m_1,\dots,m_{2L+4}\}$ whose elements are all dual to $[\go, p_{k'}]$. Lemma~\ref{lem:bounding_distance_with_dual_curtains} now yields a point $p \in [\go, p_{k'}] \cap m_{2L+3}$ and points $y_i \in [\go, p_i] \cap m_{2L+3}$ such that $\dist(y_i,p) \leq 2L+2$ for all $i \geq k'$. Since $[\go,p_i] \rightarrow b$, the geodesic $b$ must contain a point $y$ with $\dist(y,p) \leq 2L+2$. On other hand, since $b \in h_k^-$ and $\{m_1,\cdots m_{2L+2}\}$ is an $L$--chain separating $h_k$ from $y_i$, we must have $\dist(p,y) \geq 2L+3$. This is a contradiction, so $b$ meets every $h_i$. 
    
Since each $p_i \in q'$, the unparametrised rough geodesics $\iota[\go, p_i]$ lie in a uniform neighborhood of $q$. Therefore, the unparametrised rough geodesic $\iota(b)$ is also in a uniform neighborhood of $q$, by Lemma~\ref{lem:bounding_distance_with_dual_curtains}, which concludes the proof of surjectivity. Note that Corollary~\ref{cor:diverging_geodesics} means that $b$ is the unique geodesic ray crossing all the $\{h_i\}$.

\medskip
    
\item (Continuity of the inverse map.)
We prove sequential continuity, which is enough because $\cal B_L$ and $\partial X_L$ are first-countable. Let $\delta_L$ be a hyperbolicity constant for $X_L$, and let $q_n, q:[0,\infty) \rightarrow X_L$ be $(1,10\delta_L)$--quasigeodesics such that $q_n \rightarrow q$ uniformly on compact subsets of $[0,\infty)$ (see \cite[Rem.~2.16]{kapovichbenakli:boundaries}). From items~\eqref{item:injective} and~\eqref{item:onto}, there are unique geodesic rays $b_n,b$ in $X$ based at $\go$ with $\partial\iota(b_n^\infty)=q_n^\infty$ and $\partial\iota(b^\infty)=q^\infty$. Moreover, there is an integer $K$ depending only on $L,\delta_L$ such that the rough geodesic rays $\iota b$ and $q$ are at Hausdorff-distance at most $K$, and similarly for the pairs $(\iota b_n,q_n)$. In light of Theorem~\ref{thm:comparing_topologies}, to prove continuity we must show that for any curtain $h$ dual to $b$, we have $b_n^\infty \in U_h(b^\infty)$ for all but finitely many~$n$. 

Since $b$ crosses an infinite $L$-chain, by Lemma~\ref{lem:updating_curtains} it must cross an infinite $L$-chain $\{h_i\}$ dual to $b$, with every $h_i\in h^+$. Let $m=2K+L+4$, let $p\in h_m\cap b$, and note that $\dist(p,h_{L+2}^-)\ge2K+2$. There exists $p'\in q$ with $\dist_L(p,p')\le K$. As $q_n \rightarrow q$ in $X_L$, there exists $k$ such that for each $n \geq k$ there is some $p_n'\in q_n$ with $\dist_L(p',p'_n)\le 1$. For each $n\ge k$, there is some $p_n''=b_n(t_n)$ with $\dist_L(p_n',p_n'') \leq K$. By the triangle inequality, $\dist_L(p,p_n'')\le2K+1$. This shows that no $p_n''$ lies in $h_{L+2}^-$. Lemma~\ref{lem:convexity_around_L_chains} now tells us that $b_n|_{(t_n,\infty)}$ is disjoint from $h_1$, hence from $h$. Thus $b_n^\infty\in U_h(b^\infty)$ for all $n\ge k$. \qedhere
\end{enumerate} 
\end{proof}

The proof of Item~(3) of Proposition~\ref{prop:boundary_embeds} uses Corollary~\ref{cor:QG_cross_chains}, which is a consequence of the following lemma.

\begin{lemma} \label{lem:lifting_quasigeodesics}
For all $\lambda$ there exists $M$ as follows. If $P\colon[a,b]\to X$ is an unparametrised $(\lambda,\lambda)$--quasigeodesic of $X_L$, then any $L$--chain $\{c_i\}$ separating $\{x_1,x_3\}$ from $x_2$, where $x_1, x_2, x_3 \in P$ are consecutive points, has length at most $M$.
\end{lemma}

\begin{proof}
Let $\gamma=[x_1,x_3]$ in $X$. By Corollary~\ref{cor:double_crossing}, $\gamma$ meets at most $1+ \lfloor\frac{L}{2} \rfloor$ elements of $\{c_i\}$, so $\dist_L(\gamma, x_2) \geq \vert \{c_i\}\vert  - (1+ \lfloor\frac{L}{2} \rfloor)$. Proposition~\ref{prop:cat0_rough_geodesic} states that $\gamma$ is an unparametrised $(1,q)$--quasigeodesic of $X_L$. Since $X_L$ is hyperbolic, the Hausdorff-distance between $\gamma$ and $P$ is uniformly bounded in terms of $\lambda$ and $q$. Thus the distance between $\gamma$ and $x_2$ is uniformly bounded, bounding $\vert \{c_i\} \vert$.
\end{proof}

\begin{corollary}\label{cor:QG_cross_chains}
If $P\colon[0,\infty)\to X$ is an unparametrised $(\lambda,\lambda)$--quasigeodesic ray of $X_L$, then there is a sequence $(x_i)_{i=0}^\infty\subset P$ and an $L$--chain $\{c_i:i\in\mathbf N\}$ such that $c_1,\dots,c_n$ separate $x_0$ from $x_n$.
\end{corollary}

We finish this section by observing that the curtain topology provides a combinatorial description of the strata of the \emph{Morse boundary} \cite{charneysultan:contracting,cordes:morse}.

\begin{definition}[Morse boundary] \label{def:Morse_boundary} 
For a fixed base point $\go \in X$ and $D \geq 0$, define $\partial^DX=\{b^\infty\,:\,b(0)=\go \text{ and }b \text{ is }D \text{--contracting}\}$ as a set, and equip it with the subspace topology inherited from the cone topology on $\partial X$. The \emph{Morse boundary} of $X$ is defined to be 
$$
\partial^\star X:=\bigcup_{D \geq 0}\partial^DX,
$$
where $U \subset \partial^\star X$ is open if and only if $U \cap \partial^D X$ is open for every $D\geq0$. 
\end{definition}

\begin{corollary} \label{cor:Morse}
A set $U\subset\partial^\star X$ is open if and only if for every $D \geq0$ and every $b \in U \cap \partial^D X$, there exists a curtain $h$ dual to $b$ with $U_h(b^\infty) \subset U \cap \partial^D X.$
\end{corollary}

\begin{proof}
By Theorem~\ref{thm:contracting_curtain_characterisation}, we have $\partial^\star X \subset \cal B$ as sets, and Theorem \ref{thm:comparing_topologies} shows that the curtain topology agrees with the cone topology on $\partial^DX\subset\partial^\star X$. In particular, $U$ is open in $\partial^\star X$ if and only if $U \cap \partial^DX$ is open in the curtain topology for every $D$.
\end{proof}

More explicitly, the Morse boundary consists of the points in $\cal B$ that are represented by rays that cross an infinite $L$--chain \emph{at a uniform rate}.

\begin{remark}
In \cite{cashen:quasiisometries}, Cashen exhibits two quasiisometric CAT(0) spaces whose Morse boundaries, with the subspace topologies inherited from the respective visual boundaries, are not homeomorphic. As the curtain boundary $\mathcal{B}$ contains the Morse boundary as a subspace, it cannot be preserved by quasiisometries.
\end{remark}

\begin{remark}
In \cite{murray:topology}, Murray shows that the Morse boundary is dense in $\partial X$ with respect to the cone topology whenever it is nonempty and $X$ admits a proper cocompact group action. In particular $\cal B$ is dense in $\partial X$ in this case.
\end{remark}

%%%%%%%%%%%%%%%%%%%%%%%%%%%%%%%%%%%%%%%%%%%%%%%%%%%%%%%%%%%%%%%%%%%%%%
\section{The curtain model} \label{sec:universal}

In this section we use the family of metrics $\dist_L$ to produce an $\isom X$--invariant metric $\Dist$ such that $(X, \Dist)$ is a \emph{single} hyperbolic metric space with properties analogous to the family $X_L$. To achieve this, we shall rescale the metrics $\dist_L$ to make summable their evaluation at each pair of points, and obtain $\Dist$ as that weighted sum.

Fix a sequence of numbers $\lambda_L\in(0,1)$ such that
\begin{align*}
\sum_{L=1}^\infty\lambda_L \;<\; \sum_{L=1}^\infty L\lambda_L \;<\; \sum_{L=1}^\infty L^2\lambda_L \;=\; \Lambda \;<\;\infty.
\end{align*}
For $x,y\in X$, define $\Dist(x,y)=\sum_{L=1}^\infty\lambda_L\dist_L(x,y)$. The \emph{curtain model} of $X$ is the space $\X=(X,\Dist)$. 

\begin{remark} 
We make the following two observations regarding the space $\X$.
\begin{itemize}
\item   By rescaling, we could assume that either $\sum_{L=1}^\infty \lambda_L = 1$ or $\Lambda = 1$. In the former case, this would allow us to interpret the sequence of numbers $\lambda_L$ as a probability distribution on the natural numbers $\mathbf{N}$ (with full support and finite second moment). If $Y$ is the random variable associated to the distribution, then $\Dist (x,y)$ is the expected value $\mathbf{E}[\dist_Y(x,y)]$. On the other hand, if we assume that $\Lambda= 1$ then several equations become simpler. We choose not to pursue either path because it is a needless restriction, and keeping $\Lambda$ separate clarifies the origins of certain expressions.
\item   The condition on $\sum_{L=1}^\infty L^2\lambda_L$ is only needed in Section~\ref{subsec:boundary:D}; for the results on hyperbolicity and isometries of $\X$ we only need that $\sum_{L=1}^\infty L\lambda_L$ to be finite.
\end{itemize}
\end{remark}

\begin{lemma} \label{lem:D_metric}
The function $\Dist$ is a metric, and $\isom X\le\isom \X$.
\end{lemma}

\begin{proof}
Recall from Remark~\ref{rmk:new_metric_bounded} that $\dist_L\le\dist_\infty<1+\dist$ for all $L$. From this we see that $\Dist(x,y)<\sum_{L=1}^\infty\lambda_L(1+\dist(x,y))=\Lambda(1+\dist(x,y))$ is finite. Symmetry, separation of points, and the triangle inequality are immediate from the definition of $\Dist$. If $g\in\isom X$, then $g\in\isom X_L$ for all $L$, so clearly $g\in\isom\X$.
\end{proof}

By definition, $\X$ is intimately related to the spaces $X_L$ and thus it is reasonable to expect that properties of the $X_L$ can be translated into properties of $\X$. The next remark illustrates this for the results of Section~\ref{sec:diameter}. 

\begin{remark} \label{rem:XD_diameter}
From the definition of $\Dist$ it is clear that $\X$ is unbounded as soon as some $X_L$ is unbounded. Converses to this are provided by Proposition~\ref{prop:uniform_diameter} and Corollary~\ref{cor:diameter_2}: if $X$ is cobounded and has the geodesic extension property, or if $X$ admits a proper cocompact group action, then $\X$ can only be unbounded if some $X_L$ is, and otherwise both $\X$ and $X_L$ are uniformly bounded. 

Actually, as a consequence of the results of this section it turns out that if $\X$ is unbounded then one only needs to assume that $\isom X$ acts coboundedly on $\X$ to deduce that some $X_L$ is unbounded. Because $\X$ is hyperbolic (Theorem~\ref{thm:D_hyperbolic}), coboundedness ensures a loxodromic element of $\isom X$, which is contracting (Corollary~\ref{cor:contracting_qie:D}) and hence acts loxodromically on some $X_L$ (Theorem~\ref{thm:contracting_curtain_characterisation}).
\end{remark}

As discussed in the introduction, a simpler candidate for the curtain model would be to equip $X$ with the metric $\lim_{L\to \infty}\dist_L(x,y)$. In general, a procedure of this type will not produce a hyperbolic space, but it does coincide with $\X$ if the various $X_L$ eventually stabilise as long as the scaling sequence is chosen appropriately. 

\begin{lemma}\label{lem:X_L_stab_XD_stab}
Let $X$ be a CAT(0) space. Assume that there is $L_0$ and a sequence $(q_L)$ such that for all $L\geq L_0$ the identity map $X_L\to X_{L_0}$ is a $(q_L,q_L)$--quasiisometry. If $\sum_{L=L_0}^\infty\lambda_Lq_L<\infty$, then the identity map $\X\to X_{L_0}$ is a quasiisometry.
\end{lemma}

\begin{proof}
For every $x,y\in X$ we have $\lambda_{L_0}\dist_L(x,y)\le \Dist(x,y)$. In the other direction, let $\Lambda_0=\sum_{L=L_0}^\infty\lambda_L$. For $L\le L_0$, set $q_L=1$. We have 
\[
\Dist(x,y) \,\le\, \sum_{L=1}^\infty\lambda_L(q_L\dist_{L_0}(x,y)+q_L) 
\,=\, (1+\dist_{L_0}(x,y))\sum_{L=1}^\infty\lambda_Lq_L. \qedhere
\]
\end{proof}

Note that the restriction on quasiisometry constants is necessary. For example, consider the space obtained by gluing end-to-end copies of $[0,i]\times[0,10^i]$ for all $i$, in a staircase-like fashion. However, Theorem~\ref{thmi:X_hypebolic_iff_X_is_XL} will show that this phenomenon does not occur if $X$ is hyperbolic.

%%%%%%%%%%%%%%%%%%%%%%%%%%%%%%
\subsection{Hyperbolicity} \label{subsec:hyperbolicity:D}

As in Section~\ref{sec:hyperbolicity}, we shall use the guessing geodesics criterion (Proposition~\ref{prop:guessing_geodesics}) to show that the curtain model is hyperbolic. We start by observing that CAT(0) geodesics are unparametrised rough geodesics of $\X$.

% \begin{proposition} \label{prop:upqg:D}
% There is a constant $q$ such that every CAT(0) geodesic $\alpha:I\to X$ is an unparametrised $(q,q)$--quasigeodesic of $\X$.
% \end{proposition}

% \begin{proof}
% After a translation of $\R$, we may assume that $0\in I$. Let $t_0=0$. For $i>0$, given $t_{i-1}$ let $t_i$ be minimal such that $\Dist(\alpha(t_{i-1}),\alpha(t_i))\ge8\Lambda$. For $i<0$, given $t_{i+1}$, let $t_i$ be maximal such that $\Dist(\alpha(t_i),\alpha(t_{i+1}))\ge8\Lambda$. 

% We claim that $i\mapsto \alpha(t_i)$ is a uniform quasigeodesic in $\X$. Clearly it is coarsely Lipschitz. Let $x_i=\alpha(t_i)$. For each $L$, let $c^L_i$ be an $L$--chain realising $\dist_L(x_i,x_{i+1})$. From the proof of Proposition~\ref{prop:cat_unparametrised_quasigeodesic}, we know that if $i<j$ then there is an $L$--chain separating $x_i$ from $x_j$ that is obtained from $\bigcup_{k=i}^{j-1}c^L_k$ by removing at most $2L+5$ elements of each $c^L_k$. Hence $\dist_L(x_i,x_j)\ge\sum_{k=i}^{j-1}\big(\dist_L(x_k,x_{k+1})-2L-5\big)$. We therefore compute
% \begin{align*}
% \Dist(x_i,x_j) \
%     &\ge \ \sum_{L=1}^\infty\lambda_L\left(\Big(\sum_{k=i}^{j-1}\dist_L(x_k,x_{k+1})\Big)-(j-i)(2L+5)\right) \\
%     &= \ \left(\sum_{k=i}^{j-1}\sum_{L=1}^\infty\lambda_L\dist_L(x_k,x_{k+1})\right) 
%         - \sum_{L=1}^\infty\lambda_L(j-i)(2L+5) \\
%     &\ge \ \sum_{k=i}^{j-1}\Dist(x_k,x_{k+1}) - 7\Lambda(j-i) 
%         \ \ge \ (j-i)\Lambda,
% \end{align*}
% which shows that $i\mapsto\alpha(t_i)$ is colipschitz.
% \end{proof}

\begin{proposition} \label{prop:cat0_rough_geodesic:D}
There is a constant $q$ such that CAT(0) geodesics are unparametrised $q$--rough geodesics of $\X$.
\end{proposition}

\begin{proof}
According to Proposition~\ref{prop:cat0_rough_geodesic}, if $\alpha$ is a CAT(0) geodesic and $x,y,z\in\alpha$ have $z\in[x,y]$, then for all $L$ we have $\dist_L(x,y)\ge\dist_L(x,z)+\dist_L(z,y)-3L-3$. Hence
\[
\Dist(x,y) \,\ge\, \sum_{L=1}^\infty\lambda_L\left(\dist_L(x,z)+\dist_L(z,y)-3L-3\right) \,\ge\, \Dist(x,z)+\Dist(z,y)-6\Lambda. \qedhere
\]
\end{proof}

Because $X_L$ is $\delta_L$--hyperbolic, every geodesic triangle $T$ of $X$ will have a coarse centre when viewed in $X_L$, i.e. three points on the sides of $T$ at uniformly bounded distance. Unfortunately there is no guarantee that an arbitrary choice of these points will still be a coarse centre when $L$ changes. This misalignment as $L$ varies makes the proof that CAT(0) triangles are thin in $\X$ rather more difficult than in $X_L$. 

Luckily, this can be solved, and in the following lemma we find a coarse centre that works for all $X_L$. Our main tool will be triangular chains (Definition~\ref{def: triangular chain}). Let us fix some conventions regarding the corners of a triangular chain. We enumerate the $x$--corner of a triangular chain as $\{h_1,\dots,h_n\}$, so that $h_i$ is closer to $x$ than $h_j$  whenever $i<j$. Given (a subset of) an $x$--corner, the \emph{maximal element} is the curtain that is furthest away from $x$. Finally, we orient the curtains of an $x$--corner such that $x\in h_i^-$.

\begin{lemma} \label{lem:coarse_centre}
For any CAT(0)-geodesic triangle $T=[x_1,x_2,x_3]$ in $X$, there are points $p_{ij}\in[x_i,x_j]$ such that $\diam_{X_L}\{p_{12},p_{13},p_{23}\}\le8L +15$ for all $L$. 
\end{lemma}

\begin{proof}
The strategy of the proof is as follows. By considering the ``deep'' parts of the corners of a triangular $L$--chain that are far from the maximal elements, we can guarantee that those deep regions remain disjoint from one another, even when the value of $L$ changes. The coarse centres will then be provided by the complements of these deep regions, using Lemma~\ref{lem: triangular distance same as L distance}.

Fix a value of $L$ and let $b^L$ be a maximal triangular $L$--chain for $\{x_1, x_2,x_3\}$. Let $b_i^L$ be the $x_i$--corner of the chain, which is an $L$--chain. Let $a^L_i\subset b^L_i$ be the complement of the final $L+1$ elements of $b^L_i$. We claim that if $j\ne i$, then every element of $a^L_i$ is disjoint from every element of $a^{L'}_j$ for all $L'$. Suppose that some $h\in a^L_1$ meets some $k\in a^{L'}_2$, where $L\le L'$. Since every element of $b^L_1$ separates $x_1$ from $\{x_2,x_3\}$ and every element of $b^{L'}_2$ separates $x_2$ from $\{x_1,x_3\}$, we see that every element of $b^L_1\smallsetminus a^L_1$ must meet every element of $b^{L'}_2\smallsetminus a^{L'}_2$. However, this shows that some chain of curtains of length $L'+1\ge L+1$ (namely $b^{L'}_2\smallsetminus a^{L'}_2$) meets at least two elements of $b^L_1$, contradicting the fact that it is an $L$--chain. Thus every element of $a^L_i$ is disjoint from every element of $a^{L'}_j$ for all $L'$ if $j\ne i$.

Now consider the maximal element  $h\in a^L_1$ and the maximal element $k\in a^L_2$. These define a nonempty closed interval $I^L_{12} = [x_1,x_2]\smallsetminus(h_L^-\cup k_L^-)$. (If $h$ or $k$ does not exist, then replace $h^-$ or $k^-$ by the empty set in the definition of $I^L_{12}$.) Because every element of $a^L_1$ is disjoint from every element of $a^{L'}_2$ for all $L'$, the sequence $\left(\bigcap_{L=1}^nI^L_{12}\right)_n$ is a nested collection of closed intervals. There is therefore some point $p_{12}\in[x_1,x_2]$ that lies in every $I^L_{12}$. The point $p_{12}$ is separated from $x_1$ by all but the final element of $a^L_1$ (which may contain it), and likewise from $x_2$ by all but the final element of $a^L_2$, for every $L$. We can similarly construct points $p_{13}\in[x_1,x_3]$ and $p_{23}\in[x_2,x_3]$. No element of any $a^L_i$ separates any $p_{jk}$ from any $p_{jl}$, so
\[
|p_{ij},p_{ik}|_{b^L} \,\le\, \left|b^L\smallsetminus(a^L_1\cup a^L_2\cup a^L_3)\right| \le 3L+3.
\]
According to Lemma~\ref{lem: triangular distance same as L distance}, we conclude that $\diam_{X_L}\{p_{12}, p_{13},p_{23}\} \leq 8L + 15$.
\end{proof}

Whilst it may at first seem that Lemma~\ref{lem:coarse_centre} and the fact that CAT(0) triangles are thin in every $X_L$ should be enough to establish condition~\ref{ass:guessing3}, again the misalignment as $L$ varies makes things more complicated. Given $x$ in a CAT(0)-geodesic triangle and a value of $L$, we can find some point $y_L$ on one of the other sides (even the same side, thanks to Lemma~\ref{lem:coarse_centre}) with $\dist_L(x,y_L)$ uniformly bounded in $L$, but as $L$ varies the points $y_L$ might be dramatically different. In order to show that $x$ is close to another side of the triangle in $\X$, we need \emph{a single point} $y$ such that $d_L(x,y)$ is controlled for all $L$. We therefore need better control of the parametrisation of images of CAT(0) geodesics in $X_L$. The following two lemmas show that, more than just being quasigeodesics of $X_L$, the CAT(0) geodesics of $X$ actually \emph{synchronously} fellow-travel in $X_L$.

\begin{lemma}[Short tails] \label{lem:tails}
Let $\alpha$ and $\beta$ be CAT(0) geodesics in $X$ from $x$ to $p$ and $p'$ respectively, where $\dist_L(p,p')\le r$. If $y\in\alpha$ has $\dist(x,y)\ge\dist(x,p')$, then $\dist_L(y,p)\le r+L+3$.
\end{lemma}

\begin{proof}
Let $\underline p'=\alpha(\dist(x,p'))$. If $\dist(\underline p',p)\le L+1$, then $\dist_L(y,p)\le\dist_L(\underline p',p)\le L+1$. Otherwise, let $z$ be a point of $[\underline p',p]$ with $\dist(\underline p',z)\in(L+1,L+2)$. We know from Lemma~\ref{lem:chain_distance} that there is a chain $\kappa$ of curtains dual to $[\underline p',z]$ that is of length $L+1$. Because $\dist(x,\underline p')=\dist(x,p')$, we have $\pi_\alpha\beta\subset[x,\underline p']$, so no element of $\kappa$ can meet $\beta$, and hence every element of $\kappa$ separates $p$ from $p'$.

Now let $c$ be an $L$--chain separating $y$ from $p$ and realising $\dist_L(y,p)=1+|c|$. At most $r-1$ elements of $c$ can separate $p$ from $p'$, and at most one can contain $p'$. Since any curtain that separates $z$ from $\{p,p'\}$ must cross every element of $\kappa$, there can be at most one such element of $c$. At most one element of $c$ contains $z$. The remainder all separate $y$ from $z$. Because $y\in[\underline p',p]$ and $\dist(\underline p',z)<L+2$, no more than $L$ elements of $c$ can do this. Thus $|c|\le r+L+2$, so $\dist_L(y,p)\le r+L+3$.
\end{proof}

\begin{lemma}[Synchronous fellow-travelling] \label{lem:synchronous}
Let $\alpha$ and $\beta$ be unit-speed geodesics in $X$ from $x$ to $p$ and $p'$, respectively, where $\dist_L(p,p')\le r$. If $\dist_L(\alpha(t),p)\ge r+8L+21$, then $\dist(\alpha(t),\beta(t))\le4L+2$.
\end{lemma}

\begin{proof}
Suppose that $t$ is such that $\dist_L(\alpha(t),p)\ge r+8L+21$. By Lemma~\ref{lem:tails}, we necessarily have $t\le\dist(x,p')$. Let $c$ be a maximal $L$--chain separating $\alpha(t)$ from $p$, so that $|c|\ge r+8L+20$. Because $\dist_L(p,p')\le r$, at most $r-1$ elements of $c$ can separate $p$ from $p'$, and at most one can contain $p'$. Let $c'$ be the subchain consisting of all the other curtains in $c$, which all separate $\alpha(t)$ from $\{p,p'\}$. We have $|c'|\ge|c|-r\ge2(4L+10)$. 

Let $h_1$ be the element of $c'$ closest to $\alpha(t)$, and let $h_2$ be the element of $c'$ furthest from $\alpha(t)$. According to Lemma~\ref{lem:updating_curtains}, there is an $L$--chain $\kappa$ dual to $[\alpha(t),p]$ of length three whose elements separate $h_1$ from $h_2$. In particular, every element of $\kappa$ separates $\alpha(t)$ from $\{p,p'\}$. Applying Lemma~\ref{lem:bounding_distance_with_dual_curtains} with $A=\{x,\alpha(t)\}$, $B=\{p,p'\}$ shows that there is some point $q'\in\beta$ that is $(2L+1)$--close to $[\alpha(t),p]$ in the CAT(0) metric. Let $q\in[\alpha(t),p]$ be a point with $\dist(q,q')\le2L+1$.

Let $t'=t\frac{\dist(x,q')}{\dist(x,q)}$. By convexity of the metric, $\dist(\alpha(t),\beta(t'))\le\dist(q,q')\le2L+1$. Moreover, the reverse triangle inequality shows that 
\[
|t-t'| \,=\, |\dist(x,\alpha(t))-\dist(x,\beta(t'))| \,\le\, \dist(\alpha(t),\beta(t')) \,\le\, 2L+1,
\]
which implies that $\dist(\alpha(t),\beta(t))\le\dist(\alpha(t),\beta(t'))+|t-t'|\le4L+2$.
\end{proof}

% \dcomment{Added the following Lemma, that is a trade-off of the syncronous fellow-travel with less syncronicity but better constants}. 

% \begin{lemma}
% Let $\alpha$ and $\beta$ be unit-speed geodesics in $X$ from $x$ to $p$ and $p'$ respectively. For $y\in \alpha$ such that $\dist_L(x,y) \leq \dist_L(x,p) - \dist_L(p,p')- 5L -13$, there exists $y'\in \beta$ such that $\dist(y,y')\leq 2L+1$.
% \end{lemma}
% \begin{proof}
% Let $\mc{T}$ be a maximal triangular chain for $x, p, p'$. If we denote by $c(p)$ and $c(p')$ the $p$ and $p'$--corners respectively, we have $\vert c(p) \vert + \vert c(p')\vert = \vert p, p'\vert_{\mc{T}} \leq \dist_L(p, p')$. By Lemma~\ref{lem: triangular distance same as L distance}, we have $\dist_L(x,p) \leq \vert x,p \vert_{\mc{T}} + 5L +10$. Combining it with the hypotheses, we get $\dist_L(x,y) \leq \dist_L(x,p) - \dist_L(p,p')- 5L -13 \leq \vert x,p \vert_{\mc{T}} -\vert p,p'\vert_{\mc{T}} -3$. Thus, there are at least two curtains that separate $y$ from $p, p'$. The result then follows from Lemma~\ref{lem:bottleneck} and convexity of the CAT(0) distance.
% \end{proof}

The fact that CAT(0) geodesics fellow-travel synchronously in every $X_L$ allows us to upgrade Lemma~\ref{lem:coarse_centre}: in $\X$, CAT(0)-geodesic triangles do not merely have a coarse centre, they are actually thin.

\begin{lemma} \label{lem:XD_g3}
If $[x_1,x_2,x_3]$ is a CAT(0)-geodesic triangle, then, as subsets of $\X$, the set $[x_1,x_2]$ is contained in the $125\Lambda$--neighbourhood of $[x_1,x_3]\cup[x_2,x_3]$.
\end{lemma}

\begin{proof}
From Lemma~\ref{lem:coarse_centre}, there are points $p_{ij}\in[x_i,x_j]$ such that $\dist_L(p_{ij},p_{ik})\le 8L+15$ for all $L$. Let $y\in[x_1,x_2]$. Our goal is to find $\bar y$ on $[x_1, x_3]\cup [x_2, x_3]$ for which we can bound $\dist_L(y, \bar y)$ for all $L$. Either $y$ lies on the CAT(0) geodesic $[x_1,p_{12}]$ or it lies on $[p_{12},x_2]$, and the argument is the same in either case, so let us assume the former.

Let $\alpha$ and $\beta$ be the unit-speed CAT(0) geodesics from $x_1$ to $p_{12}$ and $p_{13}$, respectively. There exists $t$ such that $y=\alpha(t)$. If $\dist(x_1,p_{13})\le t$, then Lemma~\ref{lem:tails} tells us that $\dist_L(y,p_{12})\le (8L+15)+L+3$, and hence $\dist_L(y,p_{13})\le17L+33$ for all $L$. In this case, let $\bar y=p_{13}$.

Otherwise, let $\bar y=\beta(t)$. We shall bound $\dist_L(y,\bar y)$ in terms of $L$. If $\dist_L(y,p_{12})\ge (8L+15)+8L+21$, then Lemma~\ref{lem:synchronous} tells us that $\dist(y,\bar y)\le4L+2$, and hence $\dist_L(y,\bar y)\le4L+2$. Similarly, if $\dist_L(\bar y,p_{13})\ge 16L+36$ then $\dist_L(y,\bar y)\le4L+2$. If neither of these inequalities hold, then
\[
\dist_L(y,\bar y) \,\le\, \dist_L(y,p_{12})+\dist_L(p_{12},p_{13})+\dist_L(p_{13},\bar y)
    \,\le\, 2(16L+35)+(8L+15).
\]

Thus, in all possible cases we have a point $\bar y\in\beta$ such that $\dist_L(y,\bar y)\le40L+85$ for all $L$. We can therefore compute
\[
\Dist(y,\bar y) \,=\, \sum_{L=1}^\infty\lambda_L\dist_L(y,\bar y) 
    \,\le\, \sum_{L=1}^\infty(40L+85)\lambda_L \,\le\, 125\Lambda. \qedhere
\]
\end{proof}

We can now apply the guessing geodesics criterion to establish hyperbolicity of $\X$.

\begin{theorem} \label{thm:D_hyperbolic}
There is a constant $\delta$ such that the curtain models of all CAT(0) spaces are $\delta$--hyperbolic.
\end{theorem}

\begin{proof}
$\X$ is a $q$--rough geodesic space by Proposition~\ref{prop:cat0_rough_geodesic:D}. Given $x,y\in\X$, let $\eta_{xy}$ be the unique CAT(0) geodesic from $x$ to $y$. As unparametrised $q$--rough geodesics of $\X$, the $\eta_{xy}$ are coarsely connected, and condition~\ref{ass:guessing1} of Proposition~\ref{prop:guessing_geodesics} holds. Condition~\ref{ass:guessing2} holds because CAT(0) geodesics are unique, and condition~\ref{ass:guessing3} is provided by Lemma~\ref{lem:XD_g3}. Applying Proposition~\ref{prop:guessing_geodesics} shows that $\X$ has thin quasigeodesic triangles, and hence is hyperbolic by Proposition~\ref{prop:Rips_coarsely_injective}.
\end{proof}

\begin{remark}
One can also establish four-point hyperbolicity of $\X$ via a more combinatorial, less geometric argument; see \cite[Prop.~5.14]{petytzalloum:constructing}.
\end{remark}

%%%%%%%%%%%%%%%%%%%%%%%%%%%%%%%%%%%%%%%%
\subsection{Isometries and subgroups}

The goal of this subsection is to transfer results on $\isom X$ and its subgroups from the family $\{X_L\}$ to the curtain model. Before moving on to more involved considerations, we note the following.

\begin{remark} \label{rem:ivanov:D}
The Ivanov-type results of Section~\ref{subsec:Ivanov} also apply to $\X$. Indeed, the only place that the $X_L$ are explicitly used in that section is in Proposition~\ref{prop:preserving}, and it is easy to see that the same argument applies to $\X$, but with $n$--balls and $n$--spheres replaced by $(n\sum_{L=1}^\infty\lambda_L)$--balls and $(n\sum_{L=1}^\infty\lambda_L)$--spheres.
\end{remark}

It is immediate from the definition of $\X$ that whenever a subset $Y \subset X$ is quasiisometrically embedded in some $X_L$, it must also be quasiisometrically embedded in $\X$. It turns out that the converse of this statement is also true.

\begin{proposition} \label{prop:qie_in_XD_iff_XL}
Let $Y\subset X$. If $Y$ is quasiisometrically embedded in $X_L$, then $Y$ is quasiisometrically embedded in $\X$. Conversely, for all $q$ there exists $L_q$ such that if $Y$ is $(q,q)$--quasiisometrically embedded in $\X$, then $Y$ is $(q',q')$--quasiisometrically embedded in $X_{L_q}$, where $q'=2q\max\{\Lambda,\frac1\Lambda\}$. 
\end{proposition}

\begin{proof}
The first statement is immediate from the definition of $\X$. For the second statement, suppose that $Y\subset X$ is $(q,q)$--quasiisometrically embedded in $\X$. Recall that $\dist_L(x,y)\le1+\dist(x,y)$ for all $x,y\in Y$ and all $L$, so it suffices to obtain a lower bound on $\dist_L(x,y)$. 

Let $L_q$ be such that $\Lambda'=\sum_{L=L_q+1}^\infty\lambda_L\le\min\{\frac{1}{2q},q\}$, which is possible because $\sum_{L=1}^\infty\lambda_L$ converges. Given $x,y\in Y$ we can use the fact that $\dist_L(x,y)$ is increasing in $L$ but bounded above by $\dist(x,y)+1$ to compute
\begin{align*}
\frac1q\dist(x,y)-q \,\le\,  \Dist(x,y) \,
    &=\,    \sum_{L=1}^\infty\lambda_L\dist_L(x,y) \\
    &\le\,  \sum_{L=1}^{L_q}\lambda_L\dist_{L_q}(x,y) \ +  \sum_{L=L_q+1}^\infty\lambda_L(\dist(x,y)+1) \\
    &\le\,  \sum_{L=1}^\infty\lambda_L\dist_{L_q}(x,y) \ +\ \Lambda'(\dist(x,y)+1) \\
    &\le\,  \Lambda\dist_{L_q}(x,y) + \frac1{2q}\dist(x,y)+q, 
\end{align*}
from which we deduce that $\dist_{L_q}(x,y)\ge\frac1\Lambda(\frac1{2q}\dist(x,y)-2q)$. 
\end{proof}

We deduce the following from Theorem~\ref{thm:contracting_curtain_characterisation} and Proposition~\ref{prop:qie_in_XD_iff_XL}.

\begin{corollary} \label{cor:contracting_qie:D}
$\alpha$ is a contracting geodesic in a CAT(0) space $X$ if and only if $\alpha$ is quasiisometrically embedded in $\X$.
\end{corollary}

Combining Proposition~\ref{prop:qie_in_XD_iff_XL} with Proposition \ref{prop: stable versus models}, we get the following. 

\begin{corollary} \label{cor:stable:D}
Let $G$ be a group acting properly coboundedly on a CAT(0) space $X$. A subgroup $H$ of $G$ is stable if and only if it is finitely generated and orbit maps $H \rightarrow \X$ are quasiisometric embeddings. 
\end{corollary}

Applying Corollary~\ref{cor:contracting_qie:D} to axes of contracting isometries, we see that every contracting isometry of $X$ acts loxodromically on $\X$. We can actually say rather more than this in the presence of a proper group action. 

% \begin{proposition} \label{prop:hyp_contracting_D}
% Let $g$ be a hyperbolic isometry of $X$. If $g$ acts loxodromically on $\X$, then $g$ is contracting.
% \end{proposition}

% \begin{proof}
% Let $x$ lie on an axis for the action of $g$ on $X$. Since $g$ acts loxodromically on $\X$, there is some $n$ such that $\Dist(x,g^nx)>5\Lambda$. Because $\sum_{L=1}^\infty\lambda_L\le\Lambda$ and $\sum_{L=1}^\infty L\lambda_L\le\Lambda$, there must be some $L$ such that $\dist_L(x,g^nx)>L+4$. According to Theorem~\ref{thm:rank_one_characterisation}, $g$ is contracting. \dcomment{proof Ok}
% \end{proof}

\begin{proposition} \label{prop:nonuniform_acyl:D}
If a group $G$ acts properly on a CAT(0) space $X$, then the action of $G$ on $\X$ is non-uniformly acylindrical. In particular, every contracting element of $G$ acts WPD on $\X$.
\end{proposition}

\begin{proof}
Given $\eps$, let $R=2\eps+28\Lambda$. Given $x,y\in X$ with $\Dist(x,y)\ge R$, for each $L$ let $\Chain_L$ be a maximal $L$--chain separating $x$ from $y$. Suppose that $g\in G$ is such that there is no $L$ for which at least $8L+20$ elements of $\Chain_L$ separate $gx$ from $gy$. Because $gx$ and $gy$ each lie in at most one element of $\Chain_L$, all but at most $8L+21$ elements of $\Chain_L$ either separate $x$ from $gx$ or separate $y$ from $gy$ (some may even do both). Thus
\begin{align*}
\Dist(x,gx)+\Dist(y,gy) \,&\ge\, 
    \sum_{L=1}^\infty\lambda_L((1+|\{h\in\Chain_L\,:\,h\text{ separates }x,gx\}|) 
    + (1+|\{h\in\Chain_L\,:\,h\text{ separates }y,gy\}|) \\
&\ge\, \sum_{L=1}^\infty\lambda_L(2+|\Chain_L|-(8L+21)) \\
&\ge\, \Dist(x,y)-28\Lambda \,\ge\, 2\eps.
\end{align*}
By the contrapositive, for any $g\in G$ with $\max\{\Dist(x,gx),\Dist(y,gy)\}<\eps$ there is some $L$ for which at least $8L+20$ elements of $\Chain_L$ separate $gx$ from $gy$.  Applying Lemma~\ref{lem:updating_curtains}, there is an $L$--chain of length three dual to $[x,y]$ that separates $gx$ from $gy$. From Lemma~\ref{lem:bottleneck}, we find that $[gx,gy]$ comes $(2L+\dist(x,y))$--close to $[x,y]$. Curtains are thick, so $\dist(x,y)>8L+20>2L$, which means $[gx,gy]$ comes $2\dist(x,y)$--close to $[x,y]$. Hence any $g\in G$ with $\max\{\Dist(x,gx),\Dist(y,gy)\}<\eps$ sends $x$ to a point in the $4\dist(x,y)$--ball about $x$, and there are only finitely many such $g$ by properness of the action of $G$ on $X$.
\end{proof}

We finish this section by considering A/QI triples with associated hyperbolic space $\X$, just as in Section~\ref{sec:diameter}.

\begin{lemma} \label{lem:Uniform_acylindricity_master_guy}
Suppose that a group $G$ acts uniformly properly on a CAT(0) space $X$. If $Y\subset X$ is such that $[y_1,y_2]$ is $D$--contracting for all $y_1,y_2\in Y$, then the action of $G$ on $\X$ is acylindrical along $Y$.
\end{lemma}

\begin{proof}
Define $\Lambda_D=\sum_{L=10D+3}^\infty\lambda_L$. Given $\eps>0$, let $y_1,y_2\in Y$ have $\Dist(y_1,y_2)>(16D\lceil\frac\eps{\Lambda_D}\rceil+10)\Lambda$. This implies that $\dist(y_1,y_2)>16D\lceil\frac\eps{\Lambda_D}\rceil+9$. Let $g\in G$. Because 
\[
\Dist(y_i,gy_i) \,\ge\, \sum_{L=10D+3}^\infty\lambda_L\dist_L(y_i,gy_i) 
    \,\ge\, \sum_{L=10D+3}^\infty\lambda_L\dist_{10D+3}(y_i,gy_i),
\] 
if $g$ satisfies $\Dist(y_i,gy_i)\le\eps$, then it must also satisfy $\dist_{10D+3}(y_i,gy_i)\le\frac\eps{\Lambda_D}$. According to Lemma~\ref{lem:synchronous_fellow_travelling:contracting}, $g$ satisfies $\dist(y_1,gy_1)<75D\lceil\frac\eps{\Lambda_D}\rceil$. By uniform properness, there are uniformly finitely many such $g$.
\end{proof}

The following result summarises the situation for A/QI triples, and shows that the curtain model is, in a sense, universal for them.

\begin{theorem} \label{thm:A/QI:D}
The following are equivalent for a finitely generated subgroup $H$ of a group $G$ acting properly coboundedly on a CAT(0) space $X$.
\begin{enumerate}
\item\label{item:1AQI}   $H$ is stable.
\item\label{item:2AQI}   $(G,\X,H)$ is an A/QI triple.
\item\label{item:3AQI}   There exists an A/QI triple $(G,Z,H)$ for some hyperbolic space $Z$.
\end{enumerate}
\end{theorem}

\begin{proof}
The implication $\ref{item:1AQI}\Rightarrow \ref{item:2AQI}$ is obtained by combining Lemma~\ref{lem:Uniform_acylindricity_master_guy} and Corollary~\ref{cor:stable:D}. The implication $\ref{item:2AQI}\Rightarrow \ref{item:3AQI}$ is obvious; and the implication $\ref{item:3AQI} \Rightarrow \ref{item:1AQI}$ is provided in \cite[Thm~1.5]{abbottmanning:acylindrically}.
\end{proof}

%%%%%%%%%%%%%%%%%%%%%%%%%%%%%%
\subsection{Boundaries} \label{subsec:boundary:D}

Since $\X$ is a hyperbolic space, it has a Gromov boundary, $\partial\X$. In light of the results of Section~\ref{section:curtain_boundaries}, it is natural to wonder how $\partial\X$ relates to the visual boundary $\partial X$ and the curtain boundary $\cal B$ of the CAT(0) space $X$. The following example shows that $\partial\X$ can contain strictly more information than $\cal B$. 

\begin{example} \label{eg:XD_boundary}
Let $c$ be the intersection in $\R^2$ of the $\log(1+x)$--curve with the region $\{x\ge0\}$, and let $X$ be the CAT(0) space obtained by taking the convex hull of $c$ and the $x$--axis. It is straightforward to see that $X_L$ is bounded for every $L$, which means that $X$ has empty curtain boundary. However, one also finds (if the scaling sequence $(\lambda_L)$ decays subexponentially) that $\X$ is a quasiray, and so $\partial\X$ is a point.
\end{example}

In view of the fact that CAT(0) geodesics are unparametrised rough geodesics of $\X$, we can consider the subspace $\partial_{\Dist}X$ of the visual boundary $\partial X$ consisting of all points represented by rays that also represent points of $\partial\X$. The goal of this subsection is to prove the following.

\begin{theorem} \label{thm:boundary:D}
If $X$ is a proper CAT(0) space, then $\partial_{\Dist}X$ is $\isom X$--invariant, and its points are all visibility points in $\partial X$. Moreover, the identity map induces an $\isom X$--equivariant homeomorphism $\partial_{\Dist}X\to\partial\X$. 
\end{theorem}

The good behaviour with respect to isometries is immediate because $\isom X<\isom\X$. We first aim to show that the map is a homeomorphism. 

The idea behind $\X$ is that, though it may be that no individual $X_L$ can adequately capture all the hyperbolicity of $X$, yet the totality of all the spaces $X_L$ does. Correspondingly, when working with $\X$ it can sometimes greatly simplify matters to consider all $L$ simultaneously, as opposed to fixing a single value. This motivates the following definition.

\begin{definition}
A \emph{$\Dist$--chain} $\mc{C}$ is the data of an $L$--chain $\Chain_L$ for each $L$. We say that $\Chain$ \emph{separates} $x$ from $y$ if every $\Chain_L$ does. Given a $\Dist$--chain $\mc{C}$, we write $\|\mc{C}\|=\sum_{L= 1}^\infty \lambda_L \left(\vert \mc{C}_L\vert + 1\right)\in\R\cup\{\infty\}$. A \emph{maximal $\Dist$--chain} separating $x$ and $y$ is a $\Dist$--chain $\mc{C}$ such that each $\mc{C}_L$ is a maximal $L$--chain separating $x$ from $y$.
\end{definition}

Note that $\Dist$--chains need not themselves be chains of curtains, and indeed can contain the same curtain in infinitely many $\Chain_L$. If $\mc{C}$ separates $x$ from $y$, then $\Dist(x,y) \geq \|\Chain\|$, and we have equality if and only if $\Chain$ is maximal. 

We start by providing a way to show that a CAT(0) geodesic is unbounded in $\X$, even in the absence of any infinite $L$--chain that it crosses. We then prove the main technical statement that will be needed to show that the identity map induces a homeomorphism of boundaries.

\begin{lemma} \label{lem:limit_also_crosses}
Let $X$ be a CAT(0) space and let $\go\in X$. Let $\{h_1,\dots,h_n\}$ be a separated chain of curtains in $X$ with $\go\in h_1^-$, and suppose that $(\alpha_j)$ is a sequence of geodesic rays based at $\go$ such that every $\alpha_j$ crosses $h_n$ (i.e. $\alpha_j(t)\in h_n^+$ for all sufficiently large $t$). If the $\alpha_j$ converge to a geodesic ray $\alpha$, then $\alpha$ crosses $h_{n-1}$.
\end{lemma}

\begin{proof}
Let $\bar h_n^+$ be the subset of $\bar X=X\cup\partial X$ consisting of all limits of sequences contained in $h_n\cup h_n^+$, and define $\bar h_{n-1}^-$ similarly. These are closed subsets of $\bar X$, so $\alpha^\infty\in\bar h_n^+$. The same argument as in Proposition~\ref{prop:separated_means_perpendicular} shows the $\bar h_{n-1}^-$ is disjoint from $\bar h_n^+$, and it follows that $\alpha$ crosses $h_{n-1}$.
\end{proof}

\begin{lemma} \label{lem:converge_upstairs_and_downstairs}
Let $X$ be a proper CAT(0) space, and let $(\gamma_i)$ be a sequence of uniform quasigeodesics in $\X$, based at $\go$, that converge to a quasigeodesic ray $\gamma\subset\X$ representing $\xi\in\partial\X$. Suppose that $(t_i)$ is a sequence with $t_i\to\infty$ such that $\gamma_i(t_i)$ is uniformly close to $\gamma$, and let $\alpha_i=[\go,\gamma_i(t_i)]$. There is a subsequence of the $\alpha_i$ whose Arzel\`a--Ascoli limit $\alpha$ represents $\xi$. 
\end{lemma}

\begin{proof}
Let $M$ be such that every $\gamma_i(t_i)$ is $M$--close to $\gamma$. By hyperbolicity of $\X$ and the fact that CAT(0) geodesics are uniform unparametrised rough geodesics of $\X$, we can take $M$ to be large enough that $\alpha_i$ is also $M$--close to $\gamma$ for all $i$. After passing to a subsequence of the $t_i$ and relabelling, we may assume that $\Dist(\go,\gamma_i(t_i))\ge i+2M+3\Lambda$. Let us write $x_i=\gamma_i(t_i)$.

We shall proceed in an inductive manner as follows. Given $x_i$, we shall produce a $\Dist$--chain $\Chain^i$ with $\|\Chain^i\|\ge i$ and a subsequence of $(x_j)_{j>i}$ such that $\Chain^i$ separates $\go$ both from $x_i$ and from every $x_j$ in the subsequence. For convenience, we then relabel to assume that the sequence $(x_k)_{\mathbf N}$ is $(x_1,\dots,x_i)$ together with this subsequence. According to Lemma~\ref{lem:limit_also_crosses}, the Arzel\`a--Ascoli limit $\alpha$ of the subsequence of $(\alpha_j)$ inductively obtained in this way must cross all but at most one element of each $\Chain^i_L$, implying that $\diam_{\X}\alpha\ge\|\Chain^i\|-\Lambda$ for all $i$. This shows that the $x_j$ simultaneously converge to $\alpha^\infty$ in $X\cup\partial X$ and to $\xi$ in $\X\cup\partial\X$, which proves the result.

So, given $x_i$, let us consider a maximal $\Dist$--chain $\Chain^{i,1}$ separating $\go$ from $x_i$. Let $\Chain^{i,2}$ be the finite subset $\bigsqcup_{L=1}^{\lfloor\dist(\go,x_i)\rfloor}\Chain^{i,1}_L$. Observe that 
\[
\|\Chain^{i,1}\|-\|\Chain^{i,2}\| \,=\, \sum_{L=\lfloor\dist(\go,x_i)\rfloor+1}^\infty\lambda_L\dist_L(\go,x_i) \,\le\, \sum_{L=\lfloor\dist(\go,x_i)\rfloor+1}^\infty\lambda_LL \,\le\, \Lambda.
\]
For each $L$ and each $j>i$, let $\B^{i,j,1}\subset\Chain^{i,2}$ be the subset consisting of all curtains that separate $\go$ from $x_j$, and let $\B^{i,j,0}$ be its complement. Since $\Chain^{i,2}$ is finite, there is some infinite subsequence of $(x_j)_{j>i}$ whose elements all have the same corresponding set $\B^{i,j,1}$. Pass to that subsequence, and set $\Chain^i=\B^{i,j,1}$. For each $L$, at most $L$ elements of $\Chain^{i,2}_L$ can meet $\alpha_j$ without separating $\go$ from $x_j$, because of Corollary~\ref{cor:double_crossing}, and at most one can contain $x_j$. We therefore have
\begin{align*}
\|\B^{i,j,0}\| \,&\le\, \sum_{L=1}^{\lfloor\dist(\go,x_i)\rfloor}\lambda_L(L+1) 
    + \sum_{L=1}^{\lfloor\dist(\go,x_i)\rfloor} 
        \lambda_L(1+|\{h\in\B^{i,j,0} \,:\, h\text{ separates }x_i\text{ from }\alpha_j\}|) \\
&\le\, \sum_{L=1}^\infty2\lambda_LL + \Dist(x_i,\alpha_j) \,\le\, 2\Lambda+2M.
\end{align*}
We deduce that $\|\Chain^i\| = \|\B^{i,j,1}\| \ge \|\Chain^{i,2}\|-2\Lambda-2M \ge \Dist(\go,x_i)-3\Lambda-2M \ge i$. As discussed above, this suffices to prove the lemma.
\end{proof}

We now proceed to the proof that the identity map induces a homeomorphism $\partial_{\Dist}X\to\partial\X$. Let us write $\pi:X\cup\partial_{\Dist}X\to\X\cup\partial\X$ for the identity map together with the map it induces on $\partial_{\Dist}X$.

\begin{proposition} \label{prop:homeo:D}
If $X$ is proper, then the map $\pi$ restricts to a homeomorphism $\partial_{\Dist}X\to\partial\X$.
\end{proposition}

\begin{proof}
Surjectivity of the map is immediate from Lemma~\ref{lem:converge_upstairs_and_downstairs}.

Suppose that $\pi Y\subset\partial X$ is closed, and let $\alpha\in\partial X$ be the limit of a sequence $\alpha_j\in Y$. Because $\pi|_X$ is Lipschitz, the sequence $\pi\alpha_j$ converges in $\partial\X$, and the fact that $\pi Y$ is closed means that the limit is in $\pi Y$. Hence $\alpha\in Y$, so $\pi$ is continuous.

Suppose that $Y\subset\partial_{\Dist}X$ is closed, and let $\xi\in\partial\X$ be a limit of points $\xi_j\in\pi Y$. Let $\alpha_j$ be the CAT(0) geodesic based at $\go$ that represents $\xi_j$. By Lemma~\ref{lem:converge_upstairs_and_downstairs}, there is a subsequence of the $\alpha_j$ with initial subsegments converging to a geodesic $\alpha$ such that $\pi\alpha^\infty=\xi$. Since $Y$ is closed, $\alpha^\infty\in Y$, so $\pi Y$ is closed, showing that $\pi$ is open.

Finally let us show that $\pi|_{\partial_{\Dist}X}$ is injective. Suppose that $\alpha$ and $\beta$ are CAT(0) geodesics based at $\go$ that both represent $\xi\in\partial\X$. By hyperbolicity there is a constant $C\ge1$ such that the Hausdorff-distance between $\pi\alpha$ and $\pi\beta$ is less than $C$. Let $x_n\in\alpha$ be such that $\Dist(\go,x_n)>100nC\max\{1,\Lambda\}$, and let $y_n\in\beta$ have $\Dist(y_n,x_n)<C$. According to Lemma~\ref{lem:niftier_correct_proportions_in_L}, there is a number $L$ such that 
\begin{align*}
&\dist_L(\go,x_n) \,\ge\, \frac12\dist_L(x_n,y_n)\frac{\Dist(\go,x_n)}{\Dist(x_n,y_n)} \,>\, 50\dist_L(x_n,y_n) \\
\text{and } & \dist_L(\go,x_n) \,\ge\, \frac1{2\Lambda}L^2\Dist(\go,x_n) \,>\, 50L^2n.
\end{align*}
Let $c_n$ be an $L$--chain realising $\dist_L(\go,x_n)=1+|c_n|$, and let $c'_n$ be the subchain of elements that separate $\go$ from $y_n$. We have 
\begin{align*}
|c'_n| \,&\ge\, (|c_n|-1-|\{h\in c_n\,:\,h\text{ separates }x_n\text{ from }y_n\}| \\
&\ge\, (|c_n|-\dist_L(x_n,y_n)) \,>\, \frac{49}{50}|c_n| \,\ge\, 49L^2n.
\end{align*}
Applying Lemma~\ref{lem:updating_curtains} to the final $8L+20$ elements of $c'_n$ gives an $L$--chain of length three that is dual to $\alpha$ and separates all the previous elements of $c'_n$ from $\{x_n,y_n\}$. From this, Lemma~\ref{lem:gluing_disjoint_chains} provides an $L$--chain $c''_n$ of length at least $49L^2n-(8L+20)+3-1>20L^2n$ that separates $\go$ from $\{x_n,y_n\}$ and whose final three elements are dual to $\alpha$.

Let $s_n$ and $t_n$ be such that $\alpha(s_n)$ and $\beta(t_n)$ lie in the penultimate element of $c''_n$. Note that $s_n,t_n>10L^2n$, because curtains are thick. Lemma~\ref{lem:bounding_distance_with_dual_curtains} now tells us that $\dist(\alpha(s_n),\beta(t_n))\le2L+1<\sqrt{10L^2}$. This shows that $\alpha$ and $\beta$ diverge sublinearly, and we deduce from convexity of the metric that $\alpha=\beta$.

Let $s_n$ and $t_n$ be such that $\alpha(s_n)$ and $\beta(t_n)$ lie in the penultimate element of $c''_n$. Note that $s_n,t_n>10L^2n$, because curtains are thick. Lemma~\ref{lem:bounding_distance_with_dual_curtains} now tells us that $\dist(\alpha(s_n),\beta(t_n)) \le 2L+1 < \sqrt{10L^2} < \sqrt{t_n}$. This shows that $\alpha$ and $\beta$ diverge sublinearly, and we deduce from convexity of the metric that $\alpha=\beta$: see \cite[Prop.~2.12]{zalloum:convergence}.
\end{proof}

The proof of Proposition~\ref{prop:homeo:D} used the following lemma, which allows one to transfer metric data from $\X$ to some $X_L$.

% \begin{lemma} \label{lem:nifty_distance_in_XD_come_from_XL}
% Given $x,y\in X$ there exists $L$ such that $\dist_L(x,y) \geq \frac{L^2\Dist(x,y)}{\Lambda}$.
% \end{lemma}

% \begin{proof}
% If there were $x,y$ such that $\dist_L(x,y) < \frac{L^2\Dist(x,y)}{\Lambda}$ for all $L$, then we would have 
% \begin{align*}
% \Dist(x,y) &= \sum_{L=1}^\infty \lambda_i \dist_L(x,y) < \sum_{L=1}^\infty \lambda_L \frac{L^2\Dist(x,y)}{\Lambda}
%     = \frac{\Dist(x,y)}{\Lambda} \sum_{L=1}^{\infty} \lambda_LL^2 \leq \Dist(x,y),
% \end{align*}
% a contradiction.
% \end{proof}

\begin{lemma} \label{lem:niftier_correct_proportions_in_L}
Let $K,T>0$ be such that $K+T\leq 1$. For all $x_0,x_1 \in X$ there exists $L_1$ such that for all $x_2\in X$ there is some ${L_{12}}\leq L_1$ such that both of the following hold.
\begin{enumerate}
\item   $\dist_{L_{12}}(x_0, x_1) \,\geq\, K \dist_{L_{12}}(x_1, x_2) \frac{\Dist(x_0, x_1)}{\Dist(x_1, x_2)}$. \label{item:1_very_nifty}
\item   $\dist_{L_{12}}(x_0, x_1) \,\geq\, \frac{T}{\Lambda}(L_{12})^2\Dist(x_0, x_1)$. \label{item:2_very_nifty}
\end{enumerate}
\end{lemma}

\begin{proof}
If there were no value of $L$ for which $\dist_L(x_0,x_1)\ge\frac1\Lambda L^2\Dist(x_0,x_1)$, then we would have 
\begin{align*}
\Dist(x,y) &= \sum_{L=1}^\infty \lambda_i \dist_L(x,y) < \sum_{L=1}^\infty \lambda_L \frac{L^2\Dist(x,y)}{\Lambda}
    = \frac{\Dist(x,y)}{\Lambda} \sum_{L=1}^{\infty} \lambda_LL^2 \leq \Dist(x,y),
\end{align*}
which is a contradiction. We can therefore consider the nonempty set $\mc{L}$ of all values of $L$ for which Item~\ref{item:2_very_nifty} holds. Let $L_1$ be the maximal element of $\cal L$. Suppose for the sake of contradiction that no $L\in \mc{L}$ satisfies Item~\ref{item:1_very_nifty}. We would then have
\begin{align*}
\Dist(x_0, x_1) \,&=\, \sum_{L\in \mc{L}}\lambda_L \dist_L(x_0, x_1)  
        \qquad+\qquad \sum_{L\in \mathbf{N}\setminus \mc{L}}\lambda_L \dist_L(x_0, x_1) \\
    &<\, \sum_{L\in \mc{L}}\lambda_L K \dist_{L}(x_1, x_2) \frac{\Dist(x_0, x_1)}{\Dist(x_1, x_2)} 
        \,+\, \sum_{L\in \mathbf{N}\setminus \mc{L}}\lambda_L\frac{T}{\Lambda}{L}^2\Dist(x_0, x_1) \\
    &\leq\, K \frac{\Dist(x_0, x_1)}{\Dist(x_1, x_2)}\sum_{L\in \mc{L}}\lambda_L \dist_{L}(x_1, x_2) 
        \,+\, \frac{T}{\Lambda}\Dist(x_0, x_1)\sum_{L\in \mathbf{N}\setminus \mc{L}}\lambda_L{L}^2 \\
    &\leq\, K\Dist(x_0, x_1) \,+\, {T}\Dist(x_0, x_1) \;\leq\; \Dist(x_0, x_1),
\end{align*}
which is a contradiction. Thus there exists $L_{12}\in\cal L$ satisfying both items.
\end{proof}

All that remains is to show that every point of $\partial_{\Dist}X$ is a visibility point of $\partial X$. Our argument for this uses the following lemma, which shows that closest-point projection $\pi_\alpha$ to a CAT(0) geodesic coarsely agrees (in $\dist_L$ and $\Dist$, respectively) with closest-point projection in $X_L$ and $\X$.

\begin{lemma} \label{lemma:closest_point_projection}
Let $\alpha$ be a CAT(0) geodesic, let $a\in\alpha$, and let $o\in X$ be any point. We have $\dist_L(o,\pi_\alpha o)\le\dist_L(o,a)+L+3$ for all $L$, and also $\Dist(o,\pi_\alpha o) \le \Dist(o,a)+4\Lambda$. 
\end{lemma}

\begin{proof}
Let $c$ be an $L$--chain realising $\dist_L(o,\pi_\alpha o)=1+|c|$, and suppose that at least $L+1$ elements of $c$ meet $[\pi_\alpha o,a]\subset\alpha$. Let $h$ be an element of $c$ meeting $[\pi_\alpha o,a]$ that is separated from $\pi_\alpha o$ by $L$ elements of $c$. There is a point $b\in h$ with $\dist(\pi_\alpha o,b)>L+1$, so there is a chain $\kappa$ dual to $[\pi_\alpha o,b]$ of length $L+1$. Note that every element of $\kappa$ separates $[o,\pi_\alpha o]$ from $[b,a]$. This implies that any element of $c$ that meets $[\pi_\alpha o,a]$ and is separated from $\pi_\alpha o$ by $h$ must intersect every element of $\kappa$. Because $c$ is an $L$--chain, it can have at most one such element. This shows that at most $L+2$ elements of $c$ can separate $\pi_\alpha o$ from $a$. Of the remaining elements of $c$, only one can contain $a$, and the rest separate it from $o$. We therefore compute
\[
\dist_L(o,a) \,\ge\, 1+|\{h'\in c\,:\, h'\text{ separates }o\text{ from }a\}| \,\ge\, 1+|c|-(L+3) \,=\, \dist_L(x,\pi_\alpha o)-L-3.
\]
Finally, we have $\Dist(o,\pi_\alpha x)\le\sum_{L=1}^\infty\lambda_L(\dist_L(o,a)+L+3)\le\Dist(o,a)+4\Lambda$.
\end{proof}

\begin{proposition}
If $X$ is a proper CAT(0) space, then every point of $\partial_{\Dist}X$ is a visibility point of $\partial X$.
\end{proposition}

\begin{proof}
Let $\alpha$ be a geodesic ray based at $\go$ that represents a point of $\partial_{\Dist}X$, and let $\beta$ be any other geodesic ray based at $\go$. If $\pi_\alpha\beta$ is unbounded, then Lemma~\ref{lemma:closest_point_projection} shows that $\beta$ has unbounded closest-point projection to $\alpha$ in the hyperbolic space $\X$. As they are both uniform quasigeodesics, $\beta$ represents the same point of $\partial\X$ as $\alpha$, contradicting the fact that the identity map induces a bijection $\partial_{\Dist}X\to\partial\X$. Thus $\pi_\alpha\beta$ is bounded. Let $x\in\alpha$ be such that $\pi_\alpha\beta\subset[\go,x]$. Let $y\in\alpha\smallsetminus[\go,x]$ be such that $\Dist(x,y)>28\Lambda$. We have 
\[
\sum_{L=1}^\infty\lambda_L\dist_L(x,y) \,=\, \Dist(x,y) \,>\, 28\Lambda \,\ge\, \sum_{L=1}^\infty\lambda_L(8L+20),
\]
so there is some $L$ such that $\dist_L(x,y)>8L+20$. According to Lemma~\ref{lem:updating_curtains}, there is an $L$--chain $\{h_1,h_2,h_3\}$ dual to $[x,y]$. For some $t_0$, all three $h_i$ separate $\alpha(t)$ from $\beta$ for all $t>t_0$. Thus, if we let $\gamma_t=[\beta(t),\alpha(t)]$ for $t>t_0$, then Lemma~\ref{lem:bounding_distance_with_dual_curtains} shows that every $\gamma_t$ passes through a CAT(0) ball of radius $2L+1$ centred in $h_2$. As balls are compact, the Arzel\`a--Ascoli limit $\gamma$ of the $\gamma_t$ exists. It is a geodesic with $\gamma^{-\infty}=\beta^\infty$ and $\gamma^\infty=\alpha^\infty$.
\end{proof}

\begin{remark} 
It follows \emph{a posteriori} from Theorems~\ref{thm:Gromov_boundaries_inside_visual_boundaries} and~\ref{thm:boundary:D} that $\partial\X$ contains every $\partial X_L$. Note that a similar argument to that of Theorem~\ref{thm:comparing_topologies} shows that the curtain and cone topologies agree on $\partial\X$.
\end{remark}

\appendix

\section{Non-geodesic hyperbolic spaces} 

Conventions about non-geodesic hyperbolic spaces are not uniformly established in the literature. The aim of this appendix is to avoid confusion by making certain statements formal. Most likely the statements here are well-known to experts, but their inclusion makes the article more self-contained.

The following ``guessing geodesics'' criterion for hyperbolicity is used in Sections~\ref{sec:hyperbolicity} and~\ref{sec:universal}. Bowditch states a version of this criterion in terms of a ternary map that ends up being the median of the hyperbolic space \cite[Prop.~3.1]{Bowditch:Intersection}, but we follow Hamenst\"adt's formulation \cite[Prop.~3.5]{hamenstadt:geometry:complex} in terms of paths. We have included a proof because the spaces we deal with are not geodesic, and we cannot assume the paths to be continuous. 

For a path $\alpha:[0,n]\to X$, write $\ell(\alpha)=n$ for the parametrisation-length of $\alpha$. Following L\"oh \cite[\S7]{loeh:geometric}, we say that $X$ is \emph{quasihyperbolic} if there is a constant $q_0$ such that $X$ is $(q_0,q_0)$--quasigeodesic and for every $q$ there is a constant $\delta$ such that every $(q,q)$--quasigeodesic triangle of $X$ is $\delta$--thin.

\begin{proposition}\label{prop:guessing_geodesics}
Let $(X,\dist)$ be a $(q_0,q_0)$--quasigeodesic space. Assume that for some constant $D>0$ there are $D$--coarsely-connected paths $\eta_{xy}=\eta(x,y):[0,1]\to X$ from $x$ to $y$, for each pair $x, y \in X$. If the following three conditions are satisfied, then $X$ is quasihyperbolic. 
\begin{enumerate}
\enumlabel[(G1)]   The diameter of $\eta_{xy}=\eta_{xy}[0,1]$ is at most $\frac D2\dist(x,y)$. \label{ass:guessing1}
\enumlabel[(G2)]   For any $s\le t$, the Hausdorff-distance between $\eta_{xy}[s,t]$ and $\eta(\eta_{xy}(s),\eta_{xy}(t))$ is at most~$D$. \label{ass:guessing2}
\enumlabel[(G3)]   For any $x,y,z\in X$ the set $\eta_{xy}$ is contained in the $D$-neighborhood of  $\eta_{xz} \cup \eta_{zy}$. \label{ass:guessing3}
\end{enumerate}
\end{proposition}

\begin{proof}
It suffices to bound the Hausdorff-distance between $\eta_{xy}$ and an arbitrary $(q,q)$--quasigeodesic from $x$ to $y$, for then~\ref{ass:guessing3} implies that $(q,q)$--quasigeodesic triangles are thin.

Let $\alpha:[0,2^n]\to X$ be any $(q,q)$--coarsely-Lipschitz path. Write $\eta^0=\eta(\alpha(0),\alpha(2^{n-1}))$ and $\eta^1=\eta(\alpha(2^{n-1}),\alpha(2^n))$. By~\ref{ass:guessing2}, the $D$--neighbourhood of $\eta^0\cup\eta^1$ contains $\eta(\alpha(0),\alpha(2^n))$. Repeat this subdivision for $\eta^0$ and $\eta^1$, and inductively we find that $\eta(\alpha(0),\alpha(2^n))$ lies in the $nD$--neighbourhood of the concatenation of the paths $\eta(\alpha(i),\alpha(i+1))$. According to~\ref{ass:guessing1}, each of these has diameter at most $Dq$, so is contained in the $Dq$--neighbourhood of $\alpha[i,i+1]$. We therefore have that $\eta(\alpha(0),\alpha(2^n))$ is contained in the $(D\log_2\ell(\alpha)+qD)$--neighbourhood of~$\alpha$. 

Now let $\gamma:[0,n]\to X$ be a $(q,q)$--quasigeodesic from $x$ to $y$. Let $t$ maximise the distance from $\eta_{xy}(t)$ to $\gamma$. Write $r$ for this distance, and let $s$ satisfy $\dist(\eta_{xy}(t),\gamma(s))=r$. If $\dist(x,\eta_{xy}(t))\le qr+q+r$, then let $t_1=0$. Otherwise, consider $\eta(x,\eta_{xy}(t))$. Coarse connectivity and~\ref{ass:guessing2} imply that there is some $t_1<t$ such that $\dist(\eta_{xy}(t_1),\eta_{xy}(t))\in(qr+q+r,qr+q+r+D]$. Define $t_2>t$ similarly. By the choice of $r$ and $t$, there are $s_i$ such that $\dist(\eta_{xy}(t_i),\gamma(s_i))\le r$. (If $t_i\in\{0,1\}$ then take $s_i=t_i$.) Observe that $\dist(\gamma(s_1),\gamma(s_2))\le2qr+2q+4r+2D$.

Let $\alpha$ be the $(q,q)$--coarsely Lipschitz path obtained by concatenating: a $(q,q)$--quasigeodesic from $\eta_{xy}(t_1)$ to $\gamma(s_1)$, the subpath $\gamma[s_1,s_2]$, and a $(q,q)$--quasigeodesic from $\gamma(s_2)$ to $\eta_{xy}(t_2)$. Since $\alpha$ is comprised of $(q,q)$--quasigeodesics, we have $\ell(\alpha)\le q(2qr+2q+6r+2D+q)$. As we have seen, this implies that the $(D\log_2(2q^2r+2q^2+6qr+2Dq+q^2)+qD)$--neighbourhood of $\alpha$ contains $\eta(\eta_{xy}(t_1),\eta_{xy}(t_2))$. In turn, this is at Hausdorff-distance at most $D$ from $\eta_{xy}[t_1,t_2]$, by~\ref{ass:guessing2}. Thus the $(D\log_2(2q^2r+2q^2+6qr+2Dq+q^2)+(q+1)D)$--neighbourhood of $\alpha$ contains $\eta_{xy}[t_1,t_2]$. However, we also know from the choice of $t_i$ that $\dist(\eta_{xy}(t),\alpha)=\dist(\eta_{xy}(t),\gamma(s))=r$. Thus $r$ satisfies the inequality $r\le D\log_2(2q^2r+2q^2+6qr+2Dq+q^2)+(q+1)D$, and so is bounded above by some universal constant $\kappa=\kappa(D,q)$. We have shown that $\eta_{xy}$ lies in a uniform neighbourhood of any $(q,q)$--quasigeodesic from $x$ to $y$.

It remains to show that every $(q,q)$--quasigeodesic $\gamma=\gamma[0,n]$ from $x$ to $y$ lies in a uniform neighbourhood of $\eta_{xy}$. We know that the set 
\[  S=\{s\in[0,n]\hspace{1mm}:\hspace{1mm}\text{there exists } t \text{ satisfying } \dist(\gamma(s),\eta_{xy}(t))\le\kappa\}    \]
is nonempty. Given $s\in S$, let $t$ be maximal such that there exists $s''\le s$ with $\dist(\gamma(s''),\eta_{xy}(t))\le\kappa$. If $\dist(\eta_{xy}(t),y)\le D$, then $\dist(\gamma(s),y)\le\kappa+D$, so $\gamma[s,n]$ lies in the $q^2(\kappa+D+q+1)$--neighbourhood of $\eta_{xy}$.

Otherwise, consider $\eta(\eta_{xy}(t),y)$. Coarse connectivity and~\ref{ass:guessing2} imply that there is some $t'>t$ for which $\dist(\eta_{xy}(t),\eta_{xy}(t'))\le D$. Fix $s'\in[0,n]$ such that $\dist(\gamma(s'),\eta_{xy}(t'))\le\kappa$. We have $s'>s$ by the choice of $t$. Moreover, we have 
\[  \dist(\gamma(s),\gamma(s')) \le \dist(\gamma(s),\eta_{xy}(t)) + \dist(\eta_{xy}(t),\eta_{xy}(t')) + \dist(\eta_{xy}(t'),\gamma(s')) \le 2\kappa+D. \]
Since $\gamma$ is a $(q,q)$--quasigeodesic, this implies that $|s-s'|\le q(2\kappa+D+q)$. This shows that $S$ is $q(2\kappa+D+q)$--dense in $[0,n]$. Hence $\gamma(S)$ is $q^2(2\kappa+D+q+1)$--dense in $\gamma$, so $\gamma$ is contained in the $(q^2(2\kappa+D+q+1)+\kappa)$--neighbourhood of $\eta_{xy}$.
\end{proof}

Although quasihyperbolicity is more general than hyperbolicity (in the sense of the four-point condition), Proposition~\ref{prop:Rips_coarsely_injective} below characterises which quasihyperbolic spaces are hyperbolic. In particular it says that roughly geodesic quasihyperbolic spaces are hyperbolic, which is why quasihyperbolicity is not mentioned in Sections~\ref{sec:hyperbolicity} and~\ref{sec:universal}. This characterisation is presumably well known, but we were unable to find another reference.

A geodesic space is \emph{injective} if every family of pairwise intersecting balls has nonempty total intersection \cite{aronszajnpanitchpakdi:extension}. Every metric space $X$ has an essentially unique \emph{injective hull}: an injective space $E(X)$ with an isometric embedding $e:X\to E(X)$ such that any other isometric embedding of $X$ in an injective space factors via~$e$ \cite{isbell:six}. Moreover, isometries of $X$ naturally extend to $E(X)$ \cite[Prop.~3.7]{lang:injective}.

Lang proved hyperbolicity of the injective hull of any hyperbolic space, and used this to deduce that geodesic hyperbolic spaces are coarsely dense in their injective hulls \cite[Prop.~1.3]{lang:injective}. These conclusions do not hold in general for quasihyperbolic spaces: consider the graph of $|x|$ inside $(\R^2,\ell^1)$, for example. 

A metric space $X$ is \emph{weakly roughly geodesic} if there is a constant $C$ such that for all $x,y\in X$ and all $r\in[0,\dist(x,y)]$, there is a point $z$ such that $\dist(x,z)\le r+C$ and $\dist(z,y)\le(\dist(x,y)-r)+C$.

\begin{proposition} \label{prop:Rips_coarsely_injective}
Let $X$ be a metric space. The following are equivalent.
\begin{enumerate}
\item   $X$ is quasihyperbolic and weakly roughly geodesic. \label{item:quasihyperbolic}
\item   $X$ is hyperbolic in the sense of the four-point condition and weakly roughly geodesic.  \label{item:four_point}
\item   $X$ is equivariantly roughly isometric to the geodesic space $E(X)$, and $E(X)$ is \mbox{hyperbolic}. \label{item:coarsely_injective}
\end{enumerate}
\end{proposition}

\begin{proof}
It is clear that Item~\eqref{item:coarsely_injective} implies Items~\eqref{item:quasihyperbolic} and~\eqref{item:four_point}. The proof that Item~\eqref{item:four_point} implies Item~\eqref{item:coarsely_injective} is identical to \cite[Prop.~1.3]{lang:injective}, where Lang proves it for geodesic four-point hyperbolic spaces. Finally, we show that if $X$ is a weakly roughly geodesic quasihyperbolic space then it is coarsely dense in $E(X)$. According to \cite[Prop.~3.12]{chalopinchepoigenevoishiraiosajda:helly} (also see \cite[Prop.~1.1]{haettelhodapetyt:coarse}), it suffices to show that there is a constant $\eps$ such that for any collection of balls $B(x_i,r_i)$ with $r_i+r_j\ge\dist(x_i,x_j)$ for all $i,j$, the intersection $\bigcap B(x_i,r_i+\eps)$ is nonempty. 

By quasihyperbolicity, balls in $X$ are uniformly quasiconvex. Let $\{B_i=B(x_i,r_i)\}$ be a collection of balls in $X$ such that $\dist(x_i,x_j)\le r_i+r_j$ for all $i,j$. Due to weak rough geodesicity, uniform thickenings $B_i'$ of the $B_i$ intersect pairwise. Let $f:X\to X'$ be a quasiisometry to a graph (see \cite[Prop.~3.B.6]{cornulierdelaharpe:metric}, for example), which is necessarily hyperbolic, and let $\bar f$ be a quasiinverse. The $fB_i'$ are uniformly quasiconvex and intersect pairwise. According to \cite[Thm~5.1]{chepoidraganvaxes:core}, this implies that there is some point $z\in X'$ that is uniformly close to every $fB_i'$. Thus $\bar f(z)$ is uniformly close to every $B_i'$, and hence to every $B_i$ as desired.
\end{proof}

\bibliography{bio}{}
\bibliographystyle{alpha}
\end{document}